\documentclass[a4paper]{article}
\usepackage{amsmath,amssymb}
\usepackage{mathtools}
\usepackage{amsthm}
\usepackage{geometry}
\usepackage{tikz}
\usetikzlibrary{cd}
\usepackage{enumerate}
\usepackage{enumitem}
\usepackage{amsrefs}
\usepackage{url}
\usepackage[unicode,pdfencoding=auto]{hyperref}

\mathtoolsset{showonlyrefs}

\numberwithin{equation}{section}
\usepackage{appendix}

\theoremstyle{plain}
\newtheorem{theorem}{Theorem}[section]

\newtheorem{proposition}[theorem]{Proposition}
\newtheorem{lemma}[theorem]{Lemma}

\theoremstyle{definition}
\newtheorem{Definition}[theorem]{Definition}

\theoremstyle{remark}

\newtheorem{remark}[theorem]{Remark}

\newcommand{\N}{\mathbb{Z}_{>0}}
\newcommand{\integer}{\mathbb{Z}}
\newcommand{\R}{\mathbb{R}}
\newcommand{\Sph}{\mathbb{S}}

\newcommand{\embed}{\mathcal{I}_{L, \tau}(M,d_g)}
\newcommand{\data}{\mathcal{X}_n}
\newcommand{\datangraph}{\Gamma_{\epsilon}^N(\data, \tilde{d})}
\newcommand{\datagraph}{\Gamma_{m, \epsilon}(\data,\tilde{d})}

\newcommand{\RicClass}{\mathcal{M}_m(K,D)}

\newcommand{\RicClassBV}{\mathcal{M}_m^{\mathrm{v}}(K,D,v)}
\newcommand{\RicClassL}{\mathcal{M}_m^{\mathrm{GH}}(K,D,v\colon \alpha,\mathcal{L}, H)}

\newcommand{\Lnorm}[4]{\|#1\|_{L^{#2}(#3, #4)}}

\DeclareMathOperator{\diam}{\mathrm{diam}}
\DeclareMathOperator{\vol}{\mathrm{vol}}

\DeclareMathOperator{\Hess}{\mathrm{Hess}}
\DeclareMathOperator{\Lip}{\mathrm{Lip}}

\DeclareMathOperator{\volgraph}{\vol_{\Gamma}}

\DeclareMathOperator{\volmm}{\mathcal{H}^\mathnormal{m}}

\DeclareMathOperator{\sn}{\operatorname{sn}}
\DeclareMathOperator{\inj}{\mathrm{inj}}
\DeclareMathOperator{\Sect}{\mathrm{Sect}}
\DeclareMathOperator{\Ric}{\mathrm{Ric}}

\DeclareMathOperator{\diffgraph}{\delta_\epsilon}
\DeclareMathOperator{\Diffgraph}{\Delta}

\DeclareMathOperator{\diffcont}{\nabla}
\DeclareMathOperator{\esssup}{\mathrm{esssup}}
\DeclareMathOperator{\spanv}{\mathrm{span}}

\date{\today}

\title{
  Spectral convergence of graph Laplacians with Ricci curvature bounds and in non-collapsed Ricci limit spaces
}
\author{Masato Inagaki}
\begin{document}
\maketitle

\begin{abstract}
  This paper establishes quantitative high-probability bounds on the eigenvalues and eigenfunctions of $\epsilon$-neighborhood graph Laplacians constructed from i.i.d.~random variables on $m$-dimensional closed Riemannian manifolds $(M,g)$ that satisfy a uniform lower Ricci curvature bound $\Ric_g\ge -(m-1)K$, a positive lower volume bound, and an upper diameter bound.
  These results extend to non-collapsed Ricci limit spaces that are measured Gromov-Hausdorff limits of such manifolds, and the bounds give a spectral approximation of weighted Laplacians on manifolds with non-smooth points.
\end{abstract}
\tableofcontents

\section{Introduction}
\subsection*{Background}
For a closed Riemannian manifold $(M, g)$, B\'erard-Besson-Gallot studied an embedding of $M$ into the Hilbert space $L^2(M)$ of real-valued functions via the heat kernel \cite{MR1280119}.
Let $\{f_i\}_{i=1}^\infty$ be the Laplace-Beltrami eigenfunctions corresponding to the eigenvalues $\{\lambda_i\}_{i=0}^\infty$.
Then the eigenfunction map \[M\ni x\mapsto \bigl(f_1(x),f_2(x),\dots, f_k(x)\bigr) \in \R^k\] and the heat kernel embedding
\[M \ni x\mapsto \bigl(e^{-\lambda_1 t}f_1(x), e^{-\lambda_2 t}f_2(x),\dots, e^{-\lambda_k t}f_k(x)\bigr) \in \R^k\] for small $t>0$ are smooth embeddings into a Euclidean space for some $k \in \N$ \cites{MR3256785, MR3455592}.
More recently, Ambrosio-Honda-Portegies-Tewodrose studied the embeddings of possibly non-smooth metric measure spaces satisfying certain geometric conditions into $L^2$ via the heat kernel \cite{MR4224838}.

In data science, we typically have access to only finitely many samples.
Let \(n\in\mathbb{N}\) be a large integer, and let \(\data = \{x_1,\dots,x_n\}\) be i.i.d.\ random variables drawn from a probability measure \(\mu\) on a high-dimensional Euclidean space \(\R^d\).
Spectral embedding algorithms for dimensionality reduction ---such as \emph{Laplacian Eigenmaps}
\cite{NIPS2001_f106b7f9} and \emph{diffusion maps} \cite{MR2238669}---are proposed as analogues of these continuous embeddings mentioned above, and construct graph Laplacians on $\data$, or other variants of it, and transform the data onto its leading eigenvectors.

In many applications, we may assume that $M$ is isometrically embedded in $\R^d$ with \(m < d\), and \(\mu\) is supported on \(M\).
In this situation, the eigenvalues and eigenfunctions of the graph Laplacians constructed from $\data$ can approximate those of the Laplacian on $M$ with high probability as $n\to \infty$ \cites{10.5555/2976456.2976473, MR4393800, MR4804972}.
Then, Laplacian Eigenmaps also approximate the eigenfunction map.
Related works for graph discretization of the Laplacian are in \cite{preOtsu2003}.

For these discrete approximations, previous error estimates assume a sectional curvature bound $|\Sect_g| \le K_{\mathrm{sec}}$ and a positive lower bound of injectivity radius, or more strict conditions \cites{MR4804972, MR4393800, MR3299811}.
The limit spaces of sequences of manifolds under such conditions in measured Gromov-Hausdorff sense have $C^{1,\alpha}$ Riemannian metrics \cite{MR3469435}*{Theorem 11.4.7}.
Hence these frameworks exclude many natural but not smooth cases: for example, the completion of $(0, \pi) \times \Sph^{m-1}$ with the metric $d\theta^2 + \frac{1}{2}\sin^2\theta ds_{m-1}^2$, where $ds_{m-1}^2$ is the standard metric on $\Sph^{m-1}$.

We remove these restrictions.
Assuming a uniform lower Ricci curvature bound, a positive volume lower bound, and an upper diameter bound on manifolds, we derive
high-probability, quantitative error bounds for the eigenvalues and
eigenfunctions of the graph Laplacians on the data set $\data$. 
We also show that the results extend to non-collapsed Ricci limit spaces that are limits of the Gromov-Hausdorff convergence of such manifolds.

\subsection*{Main theorems}\label{sec:MainThms}

Let us explain the setting of the main theorems.
Let $(M,g)$ be an $m$-dimensional closed Riemannian manifold satisfying
\begin{equation}
  \Ric_g \ge -(m-1)K, 
  \qquad 
  \diam(M,d_g) \le D, 
  \qquad 
  \vol_g(M) \ge v>0,
\end{equation}
with constants $K,D\ge1$ and $v\in(0,1)$.
For $\alpha\ge 1$, $\mathcal L\ge 1$, and $ H \in \R$ fix a probability density $\rho\colon M\to(0,\infty)$ that is $C^2$ and $\mathcal L$-Lipschitz function satisfying \[\frac{\max\rho}{\min\rho}\le \alpha\]
and obeying the Hessian bound \[\Hess(\log\rho)\le H.\]
We define the weighted Laplacian $\Delta^N_\rho: H^{2,2}(M) \to L^2(M)$ by $\Delta^N_\rho f = \Delta_g f - 2\langle \diffcont \log \rho, \diffcont f\rangle$ as in Section~\ref{sec:Riemann}, where $\Delta_g f = - \mathrm{tr} \Hess f$.
Let $L \ge 1$, and suppose that $(M, g)$ is isometrically embedded in $\R^d$ with \[d_g(x,y)\le L\,d_{\R^d}(x,y)\] for all $x,y\in M$.
Let \(\data=\{x_1,\dots,x_n\}\) be $n$ independent random variables distributed according to \(\rho\,\vol_g\).
Setting \[\epsilon = \left(\frac{\log n}{n}\right)^{\!\frac{1}{m+2}},\] 
we define the matrices $A_{n}, D_{n}, L_{n} \in M_n(\R)$ by
\begin{equation}
  (A_{n})_{ij} = \left\{
    \begin{aligned}
      1, &\quad \textrm{if $\|x_i - x_j\|_{\R^d} < \epsilon$},\\
      0, &\quad \textrm{otherwise},
    \end{aligned}
  \right.
  \quad
  (D_{n})_{ij} = \delta_{ij}\sum_{k=1}^n (A_{n})_{ik},
  \quad L_{n} = 2\epsilon^{-2}\left(I - D_{n}^{-1}A_{n}\right).
\end{equation}
This $L_{n}$ is the graph Laplacian on $\Gamma^N_\epsilon(\data, d_{\R^d})$ in Definition~\ref{def:Graph} and does not depend on $d_g$.
The matrix $L_n$ is called a normalized random walk Laplacian.

To estimate error terms in this paper, we employ two integrals
\begin{equation}
  V_{p,\epsilon}(M)
  :=\!
  \Bigl(
    \int_M
    \Bigl|1-\tfrac{\vol_g(B(x,\epsilon))}
    {V_K(\epsilon)}\Bigr|^{p}
    \,d\vol_g(x)
  \Bigr)^{\!\!1/p}
\end{equation}
for $p \ge 1$ and
\begin{equation}
  S_\epsilon(M,d_g,d_{\R^d}):= \int_M \vol_g\bigl(B_{\R^d}(x, \epsilon)\setminus B(x, \epsilon)\bigr)\, d\vol_g,
\end{equation}
where $V_{p, \epsilon}(M)$ measures the average deviation of small geodesic balls from the
constant-curvature model of curvature $-K$, while
$S_\epsilon(M,d_g,d_{\R^d})$ quantifies the metric distortion between
$d_g$ and $d_{\R^d}$ (see
Definitions~\ref{def:LpEuclid}--\ref{def:S}).

Let $\lambda_k(\Delta_\rho^N)$ and $\lambda_k\!(L_n)$ denote the $k$-th eigenvalues ($k=0, 1, \dots, n-1$) of the weighted Laplacian $\Delta^N_\rho$ and the matrix $L_n$, respectively.
We have the following theorem for a sufficiently large sample size $n$.
\begin{theorem}\label{thm:MainEV}
  For every $k\in \N$ there exists a constant
  $C=C\!\bigl(m,K,D,v,\alpha,\mathcal L, H,L,k\bigr)>0$
  such that for every $\beta \ge 1$ if $\epsilon\sqrt{\beta} C \le 1$
  we have
  \begin{equation}\label{eq:MainEV}
    \begin{split}
  &\bigl|
  \lambda_k(\Delta_\rho^N)
  - (m+2)\,\lambda_k\!\bigl(L_n\bigr)
  \bigr|\\
  &\qquad \;\le\;
  C\Bigl(
    \epsilon^{\frac{m}{m+2}}
    + V_{m+2,\epsilon}(M)\,\epsilon^{-\frac{2}{m+2}}
    + S_\epsilon(M,d_g,d_{\R^d})\,\epsilon^{-m}
  \Bigr)
    \end{split}
  \end{equation}
  with probability at least \(1-Cn^{-\beta}\).
\end{theorem}
The full version of this theorem is Theorem~\ref{thm:AppEigenv}.
We also show the \(L^2\)-approximations of the eigenfunctions.
\begin{theorem}\label{thm:MainEF}
  For every $k\in \N$ there exists a constant
  $C=C\!\bigl(m,K,D,v,\alpha,\mathcal L, H,L,k\bigr)>0$
  such that assuming $\gamma: = \frac{1}{2}\min\{\lambda_k(\Delta^N_\rho) - \lambda_{k-1}(\Delta^N_\rho), \lambda_{k+1}(\Delta^N_\rho) - \lambda_{k+1}(\Delta^N_\rho), 1\} >0$, for every $\beta\ge 1$, if $\epsilon\sqrt{\beta}C \le \gamma^2$,
   the following property holds with probability at least \(1-Cn^{-\beta}\):
   for every eigenvector $u^k=(u^k_i)_{i=1}^n \in \R^n$ of the matrix $L_n$ with $\sum_{i=1}^n (u^k_i)^2\frac{(D_n)_{ii}}{n(n-1)\omega_m\epsilon^m} = 1$ corresponding to $\lambda_k(L_n)$,
  there exists an eigenfunction  $f_k: M \to \R$ of $\Delta^N_\rho$ with $\int_M f^2\,\rho^2d\vol_g = 1$ corresponding to the eigenvalues $\lambda_k(\Delta^N_\rho)$ such that
  \begin{equation}\label{eq:MainEF}
    \begin{split}
  &\frac{1}{n}\sum_{i=1}^n|f_k(x_i) - u^k_i|^2\frac{(D_n)_{ii}}{n(n-1)\omega_m\epsilon^m}\\
  &\qquad \;\le\;
  \frac{C}{\gamma^2}\Bigl(
    \epsilon^{\frac{m}{m+2}}
    + V_{m+2,\epsilon}(M)\,\epsilon^{-\frac{2}{m+2}}
    + S_\epsilon(M,d_g,d_{\R^d})\,\epsilon^{-m}
  \Bigr)
    \end{split}
  \end{equation}
  holds.
\end{theorem}
The full version of this theorem is Theorem~\ref{thm:AppEigenNonSingf}.

Here, two error terms $V_{m+2,\epsilon}(M)\,\epsilon^{-\frac{2}{m+2}}$ and $S_\epsilon(M,d_g,d_{\R^d})\,\epsilon^{-m}$ converge to $0$ as $n \to \infty$ for $(M,g)$.
Setting an upper bound of sectional curvature $|\Sect_g| \le K_{\mathrm{sec}}$, a lower bound of injectivity radius $\inj_g \ge i_0 >0$, and an upper bound of total second fundamental form $S \ge \int_M |II| d\vol_g$,
if $\epsilon \le i_0$ we have $V_{m+2, \epsilon}(M)\epsilon^{-2}\le C(m,K,K_{\mathrm{sec}},D)$ by the comparison theorem \cite{MR1390760}*{VI Theorem~3.1 (1)} and $S_\epsilon(M,d_g,d_{\R^d})\,\epsilon^{-m-1} \le C(m,K,D,L,S)$ by \cite{MR4804972}*{Lemma~3.2 (10)}.
Hence the right hand side of the inequality \eqref{eq:MainEV} is bounded by $C\epsilon^{m/(m+2)}$ for some $C=C(m,K,K_{\mathrm{sec}},D,v,\alpha,\mathcal L,  H, L, k)$.

Its rate is worse than that of Aino $\mathcal O(\epsilon)$ \cite{MR4804972}, but the assumptions of Theorems~\ref{thm:MainEV} and~\ref{thm:MainEF} are weaker: we no longer depend on the bounds $K_{\mathrm{sect}}$ and $i_0$.

These estimates~\eqref{eq:MainEV} and \eqref{eq:MainEF} extend verbatim to {non-collapsed Ricci limit spaces} that arise as measured Gromov-Hausdorff limits of manifolds obeying the above geometric conditions: see
Theorems~\ref{thm:LimAppEigenv} and \ref{thm:LimAppEigenf}, respectively.
For example, we now cover the spindle $((0, \pi) \times \Sph^{m-1}, d\theta^2 + \frac{1}{2}\sin^2\theta ds_{m-1}^2)$ embedded in $\R^d$ via
\begin{equation}
  (\theta, u) \mapsto \left(\int_0^\theta \sqrt{ 1- \tfrac12\cos^2\phi} \, d\phi, (\tfrac{1}{\sqrt{2}}\cos\theta)u, 0, \dots, 0\right)
\end{equation}
for \(m\ge3\) since $V_{p, \epsilon}(M)\epsilon^{-2/p}$ and $S_\epsilon(M,d_g, d_{\R^d})\epsilon^{-m}$ converge to $0$ if $m\ge 3$.
For \(m=2\) Theorems~\ref{thm:LimAppEigenv} and \ref{thm:LimAppEigenf} do not guarantee the convergence but still provide uniform bounds.

\subsection*{Strategy of the proof}\label{sec:Strategy}

Our proof rests on two ideas that remove the need for pointwise sectional-curvature control:

\begin{enumerate}[label=\textbf{\roman*)}]
  \item \textbf{Integral---rather than pointwise---control of the interpolation kernel.}
        We employ the interpolation map
        \(
          \Lambda_\epsilon \colon L^2(\data)\to\Lip(M)
        \)
        (Definition~\ref{def:IMap}) to bound
        \(\lambda_k(\Delta_\rho^N)\) from above by
        \(\lambda_k\bigl(\Gamma_\epsilon^N(\data,d_{\R^d})\bigr)\).
      Earlier work \cites{MR4804972, MR3299811} relied on upper sectional-curvature bounds to obtain pointwise estimates for
        the regularization term \(\theta_{n,\epsilon}\colon M\to\R\) that appears in their interpolation maps. 
        We overcome this by estimating, instead of pointwise gradients, the sum
        \(
          \sum_{i=1}^n \lvert\nabla\theta_{n,\epsilon}(x_i)\rvert^2
        \)
        in terms of \(V_{p,\epsilon}(M)\), which measures, in an
        $L^p$-sense, how far \(M\) deviates from the constant curvature
        model. 
  \item \textbf{$L^q$ bounds for graph eigenfunctions.}
        To avoid pointwise control of \(\theta_{n,\epsilon}\), we require
        \(L^q\) estimates for graph eigenfunctions.
        Adapting the Moser iteration to the discrete setting, we derive such bounds.
        The weighted \(\epsilon\)-neighborhood graphs constructed from datasets fail to satisfy volume-doubling at arbitrarily small scales, so we cannot yield uniform $L^\infty$ bounds for the graph eigenfunctions by the classical Moser iteration.
        Instead, Section~\ref{sec:DataSet} shows that, with high probability, the graph on the data set satisfies \emph{rough volume-doubling property} and
        Poincar\'e inequalities, and Section~\ref{sec:LpEst} establishes a
        rough Nash-type inequality using these regularities.
        Iterating the Nash-type inequality yields sharp \(L^q\) bounds of graph eigenfunctions
        (Theorem~\ref{thm:EstEigenfGraph}) for every fixed \(q<\infty\).
        For comparison, on a fixed Riemannian manifold, we can recover the
        \(L^\infty\) estimates for graph eigenfunctions constructed from i.i.d.\ random samples with high probability \cite{MR4384039}.
\end{enumerate}
Based on these two estimates, we no longer need to assume an upper sectional-curvature bound or a positive injectivity radius.

\subsection*{Organization}\label{sec:Org}

Section~\ref{sec:Pre} recalls weighted Riemannian manifolds, non-collapsed Ricci
limit spaces, and the graph constructions used throughout this paper.  We also introduce the error-controlling terms $V_{p, \epsilon}$ and $S_\epsilon$ here.
Section~\ref{sec:LpEst} establishes \(L^p\) estimates for graph eigenfunctions
under the rough volume-doubling property and Poincar\'e inequalities.  
In Section~\ref{sec:DataSet}, we show that the graphs based on the data sets satisfy these properties with high probability.  
Sections~\ref{sec:DiscMap}--\ref{sec:IntMap} compare discrete and continuous
Rayleigh quotients via a \emph{discretization} map and an
\emph{interpolation} map, yielding matching lower and upper eigenvalue bounds. The
\(L^p\) estimates from Section~\ref{sec:LpEst} are crucial in
Section~\ref{sec:IntMap}.  
In Section~\ref{sec:Limit}, the combination of discussions in Sections~\ref{sec:DiscMap} and~\ref{sec:IntMap} yields spectral convergence on manifolds and Ricci limit spaces, and quantifies eigenspace approximation.  
Appendix~A collects supplementary \(L^\infty\) and gradient estimates for
manifold eigenfunctions under Ricci curvature and
\(\Hess(\log\rho)\) bounds.

\paragraph*{Acknowledgments.}
This paper forms a part of my phD thesis.
The author would like to thank Shinichiroh Matsuo for helpful discussions and Masayuki Aino for bringing this problem to his attention.
This work was financially supported by JST SPRING, Grant Number JPMJSP2125.
The author would like to take this opportunity to thank the
``THERS Make New Standards Program for the Next Generation Researchers''.

\section{Preliminaries}\label{sec:Pre}

Throughout this paper, \( C, C_1, C_2,\dots, \) denote positive constants whose values may vary from one occurrence to the next.
We write $C_1(K)$, $C_2(m,K,D,v)$, etc., to indicate constants depending only on the respective parameters.

\subsection{Weighted Riemannian manifolds and their limit spaces}\label{sec:Riemann}
This subsection explains some basic concepts of Riemannian geometry, then introduces several classes of manifolds, probability densities, and their associated weighted Laplacians, which will be central to our analysis.
In addition, we recall the measured Gromov-Hausdorff convergence and non-collapsed Ricci limit spaces needed for Section~\ref{sec:Limit}.

Let \((M,g)\) be a closed \(m\)-dimensional Riemannian manifold.  For \(p\in M\) write the unit tangent sphere as \(U_pM:=\{u\in T_pM\colon \|u\|_g=1\}\).  The exponential map \(\exp_p:T_pM\to M\) is defined by \(\exp_p(tu)=c_u(t)\), where \(c_u\) is the geodesic with \(c_u(0)=p,\;c_u'(0)=u\).  Set
\[t(u):=\sup\{t>0: t=d_g\bigl(p,c_u(t)\bigr)\},\qquad
  \tilde U_p:=\{(t,u):0<t<t(u)\},\qquad U_p:=\exp_p(\tilde U_p).
\]
The map \((t,u)\mapsto \exp_p(tu)\) is a diffeomorphism from \(\tilde U_p\) onto \(U_p\setminus\{p\}\).
We denote its Jacobian at \((t,u)\) by \(\Theta_u(t)\) and extend by zero for \(t\ge t(u)\).
Fix \(K\ge1\) and define the model functions
\[\sn_K(r)=\frac{\sinh(\sqrt K\,r)}{\sqrt K},\qquad
  V_K(r)=\vol(\Sph^{m-1})\int_0^r \sn_K(t)^{m-1}\,dt,\qquad r>0.\]
We frequently use the Bishop-Gromov comparison theorem for manifolds with \(\Ric_g\ge-(m-1)K\): see \cite{MR1390760}*{IV Theorem.~3.1}.
\begin{theorem}[Bishop-Gromov]\label{thm:RiccComp}
Assume \(\Ric_g\ge-(m-1)K\) and fix \(p\in M\).
\begin{enumerate}
  \item The map \(t\mapsto \Theta_u(t)/\sn_K(t)^{m-1}\) is non-increasing on \((0,t(u))\) and tends to $1$ as $t\to0$.
  \item The map \(r\mapsto \vol_g\bigl(B(p,r)\bigr)/V_K(r)\) is non-increasing on \((0,\infty)\) and tends to $1$ as $r\to0$.
\end{enumerate}
\end{theorem}
 Immediate consequences follow.
\begin{itemize}
  \item If $0 < t\le 1/\sqrt K$, then 
    \begin{equation}
      \Theta_u(t)\le  (1+C(m) K t^2)t^{m-1} \quad \textrm{and} \quad
      \vol_g\bigl(B(p,r)\bigr)\le C(m) r^m.
    \end{equation}
  \item If $\diam(M,d_g)\le D$, \[C(m,K,D)\,\vol_g(M)\,r^m\le \vol_g\bigl(B(p,r)\bigr)\le C(m,K,D)r^m\] for all $r\ge0$.
\end{itemize}

In this paper, we work with the following classes of Riemannian manifolds.
\begin{Definition}\label{def:ClassRieman2}
  Let $m \in \N$. Fix $K\ge 1, D \geq 1$, and $v \in (0,1)$.
  We introduce two classes of $m$-dimensional closed Riemannian manifolds:
  \begin{enumerate}
    \item The class $\RicClass$ consists of $m$-dimensional closed Riemannian manifolds $(M, g)$ such that
      \[
        \Ric_g \geq -(m-1)K \quad \text{and} \quad \diam(M,d_g) \leq D.
      \]
    \item The class $\RicClassBV$ consists of $(M,g) \in \RicClass$ such that $\vol_g(M) \geq v$.
  \end{enumerate}
\end{Definition}

We also introduce two classes of probability density functions.
Sections~4--7 consider i.i.d.\ random variables distributed according to the densities defined as below.
\begin{Definition}
  Let $(M,g)$ be a closed $m$-dimensional Riemannian manifold.
  For constants $\alpha \ge 1$, $\mathcal{L} \geq 1$, and $ H \in \R$,
  we define the following classes of positive functions on $M$:
  \begin{enumerate}
    \item The class $\mathcal{P}(M\colon \alpha,\mathcal{L})$ consists of Lipschitz functions $\rho: M \to (0, \infty)$ such that
      \[
        \int_M \rho \, d\vol_g = 1, \quad \frac{\max \rho}{\min \rho} \leq \alpha, \quad \Lip \rho \leq \mathcal{L}.
      \]
    \item The class $\mathcal{P}(M\colon \alpha,\mathcal{L},  H)$ consists of $C^2$ functions $\rho:M \to (0, \infty)$ such that $\rho \in \mathcal{P}(M:\alpha,\mathcal{L})$ and
      \begin{equation}
        \sup_{p\in M}(\Hess\log\rho)_p \leq  H.
      \end{equation}
  \end{enumerate}
\end{Definition}

Note that if $(M, g) \in \RicClassBV$, then $\rho\colon  M \to (0, \infty)$ in this definition satisfies 
\[
  \frac{1}{\alpha V_K(D)} \le \rho \le \frac{\alpha}{v}.
\]

For closed Riemannian manifolds $(M,g)$ and Lipschitz functions $\rho\colon M \to (0, \infty)$,
we consider two Laplace operators $\Delta_\rho, \Delta_\rho^N\colon  H^{2,2}(M) \to L^2(M,g)$ defined by 
\[
  \Delta_\rho^N f = \Delta_g f - 2\langle \diffcont \log \rho, \diffcont f \rangle \quad \textrm{and} \quad
  \Delta_\rho f = \vol_g(M) \rho \Delta_\rho^N f
\]
for any $f \in H^{2,2}(M)$,
where $ \Delta_g = - \mathrm{tr} \Hess$.
Then we have
\begin{align}
  &\int_M |\diffcont f|^2 \rho^2\, d\vol_g = \int_M  (\Delta_\rho^N f)f \rho^2\, d\vol_g,\\
  &\int_M |\diffcont f|^2 \rho^2\, d\vol_g = \frac{1}{\vol_g(M)}\int_M  (\Delta_\rho f) f\rho\, d\vol_g.
\end{align}
Counting multiplicities, their discrete spectra are given by
\begin{align}
  &0=\lambda_0(\Delta_\rho) < \lambda_1(\Delta_\rho) \leq \dots \leq \lambda_k(\Delta_\rho) \to \infty,\\
  &0=\lambda_0(\Delta_\rho^N) < \lambda_1(\Delta_\rho^N) \leq \dots \leq \lambda_k(\Delta_\rho^N) \to \infty.
\end{align}
Sections~\ref{sec:DataSet}--\ref{sec:Limit} estimate bounds for the eigenvalues and eigenfunctions for $\Delta_\rho$ and $\Delta^N_\rho$.

\begin{remark}\label{rem:EstCEigenV}
  Let $(M,g)\in \RicClass$ with $\diam (M, d_g) \ge D^{-1}$, and let $\rho\colon M\to (0, \infty)$ be a Lipschitz function with $\max\rho/\min\rho \le\alpha$.
  Comparing Rayleigh quotients yields $\lambda_k(\Delta_\rho), \lambda_k(\Delta^N_\rho)\le \alpha^2\lambda_k(\Delta_g)$ for every $k \in \N$.
Combining this with \cite{MR378001}*{Theorem~2.1}, we obtain
  $\lambda_k(\Delta_\rho), \lambda_k(\Delta^N_\rho) \le C(m,K,D,\alpha,k)$.
\end{remark}

Section~\ref{sec:Limit} extends our approximation results to limit spaces in the sense of the measured Gromov-Hausdorff convergence.
We recall the definition of this convergence as follows.
\begin{Definition}\label{def:GH}
  Let \( \{ (M_t, d_t) \}_{t=1}^\infty \) and \((M, d_M)\) be compact metric spaces.
  We say that \emph{\( (M_t, d_t) \) converges to \( (M,d_M) \) in Gromov-Hausdorff (GH) sense} if there exists a family of maps \( \{ \Phi_t \colon  M_t \to M \}_{t=1}^\infty \) and a sequence \( \{\delta_t > 0 \}_{t=1}^\infty \) with $\lim_{t\to \infty} \delta_t = 0$ such that
  \begin{enumerate}
    \item for each $t\in \N$, \( B_M\bigl( \Phi_t ( M_t), \delta_t\bigr) \supset M \);
    \item for each $t\in \N$ and each $x, y\in M_t$, \( \bigl| d_t(x,y) - d_M\bigl( \Phi_t(x), \Phi_t(y) \bigr) \bigr| < \delta_t \).
  \end{enumerate}
  Given finite Borel measures $\mu_t$ on $M_t$ and $\mu$ on $M$,
  we say that \emph{\( (M_t, d_t, \mu_t) \) converges to \( (M,d_M, \mu) \) in the measured Gromov-Hausdorff (mGH) sense} if, in addition,
  \begin{enumerate}
    \item[(iii)]  the maps \( \{ \Phi_t \colon  M_t \to M \}_{t=1}^\infty \) are Borel measurable and $(\Phi_t)_*\mu_t \to \mu$ in weak* topology, i.e.,
      \[\int_{M_t} f\circ \Phi_t\,d\mu_{t}\to \int_{M} f\,d\mu\]
      as $t \to \infty$ for every continuous function $f\colon M \to \R$,
      where $(\Phi_t)_*\mu_t$ denotes push-forward of $\mu_t$ by $\Phi_t$.
  \end{enumerate}
  Whenever $(M_t,d_t)\to (M,d_M)$ in GH sense, for each sequence $\{ x_t \in M_t \}_{t=1}^\infty$ and $x \in M$, we simply write \emph{$\Phi_t(x_t) \to x$} by $x_t \to x$.
\end{Definition}

Using this convergence, we introduce the following class.
\begin{Definition}\label{def:ClassLimit}
  We denote by $\RicClassL$ the set of triples $(M, d_M, \rho)$ such that $(M,d_M)$ is a compact metric space, $\rho\colon M\to (0,\infty)$ is a function, and there exist sequences $\{(M_t,g_t)\}_{t=1}^\infty \subset \RicClassBV$ and $\{\rho_t\subset \mathcal{P}(M_t:\alpha,\mathcal{L}, H)\}_{t=1}^\infty $ satisfying
  \begin{enumerate}
    \item $(M_t, d_{g_t}) \to (M, d_M)$ in Gromov-Hausdorff sense as $t \to \infty$,
    \item $\rho_t(x_t)\to \rho(x)$ as $t\to \infty$ for each $x_t \to x$.
  \end{enumerate}
\end{Definition}
From $\vol_{g_t}(M_t)\ge v$, we have
$(M_t,d_{g_t},\operatorname{vol}_{g_t})\to(M,d_M,\mathcal H^m)$ in the mGH sense \cite{MR1484888}*{Theorem~5.9},
where $\volmm$ is the $m$-dimensional Hausdorff measure.
Then $(M,d_M,\mathcal H^m)$ is called a non-collapsed Ricci limit space.
Moreover, since $\rho_t$ are uniformly $\mathcal L$-Lipschitz, we obtain
$(\Phi_{t})_*(\rho_t^i\vol_{g_t})\to\rho^i\mathcal H^m$ ($i=1,2$) in weak* topology,
which yields the following canonical self-adjoint operator on $L^2(M,\rho^2\mathcal H^m)$.

\begin{theorem}\label{thm:mGHSpecConv}
  Let $(M, d_M, \rho) \in \RicClassL$.
  Assume that $\{(M_t,g_t)\}_{t=1}^\infty \subset \RicClassBV$ and $\{\rho_t \subset \mathcal{P}(M_t\colon \alpha,\mathcal{L}, H)\}_{t=1}^\infty$ are as described in Definition~\ref{def:ClassLimit}.
  Then there exists a unique self-adjoint operator $\Delta_\rho^N$ on $L^2(M, \rho^2\volmm)$ such that
  \[
    \int_M \Lip(f)^2 \,\rho^2d\volmm = \int_M (\Delta_N^\rho f) f \,\rho^2d\volmm
  \]
  for every $f \in L^2(M, \rho^2\volmm)$ in the domain of $\Delta^N_\rho$. In addition, the following properties hold: 
  \begin{enumerate}
    \item[\textnormal{(a)}]
      $\Delta_\rho^N$ has the eigenvalues
      $0=\lambda_0(\Delta_\rho^N) < \lambda_1(\Delta_\rho^N) \leq \dots \leq \lambda_k(\Delta_\rho^N) \to \infty$.
    \item[\textnormal{(b)}]
      For every $k \in \N$, $\lambda_k\bigl(\Delta^N_{\rho_t}\bigr) \to \lambda_k\bigl(\Delta^N_\rho\bigr)$ as $t \to \infty$.
    \item[\textnormal{(c)}]
      Let $f_1,\dots,f_k \in L^2\bigl(M,\rho^2\volmm\bigr)$ be orthonormal eigenfunctions for $\Delta^N_{\rho}$
      corresponding to $\lambda_1(\Delta_\rho^N),\dots,\lambda_k(\Delta_\rho^N)$.
      Then, by possibly taking a subsequence for $t$, there exist a sequence of orthonormal eigenfunctions $\{\{f^t_1,\dots,f^t_k\} \subset L^2 \!\bigl(M_t,\rho_t^2\vol_{g_t}\bigr)\}_{t=1}^\infty$ corresponding to $\{\lambda_1(\Delta^N_{\rho_t}),\dots,\lambda_k(\Delta^N_{\rho_t})\}_{t=1}^\infty$ such that for every $i \in \{1,\dots,k\}$, we have $f^t_i \to f_i$ in $L^2$,
      i.e.,
      \[
        \|f^t_i - f_i \circ \Phi_t\|_{L^2} \to 0
      \]
      as $t \to \infty$.
  \end{enumerate}
\end{theorem}

This may be familiar to experts, but we give a proof outline here for the reader's convenience.

\begin{proof}[Proof Outline]
We first show the existence of the operator and (a).
Observe that $(M,d_M)$ is a length space by \cite{MR1835418}*{Theorem 7.5}.
Moreover, $(M, d_M, \rho^2\volmm)$ has the segment inequality (2.3) in \cite{MR1815411} holds since each approximating manifold $(M_t, g_t, \rho_t^2\vol_{g_t})$ satisfies this inequality with a uniform constant $\tau$.
This implies the weak Poincar\'{e} inequality of $(1,2)$-type, in (1.5) of \cite{MR1815411}.
The metric measure space $(M, d_M, \rho^2\vol_g)$ also satisfies the volume-doubling property (0.5) in \cite{MR1815411}.
The volume-doubling property and the weak Poincar\'{e} inequality together allow for the deduction of the Poincar\'{e} inequality (1.6) in \cite{MR1815411} (see, for example, \cite{MR3469435}*{pp.~287--292}).
Application of \cite{MR1815411}*{Theorem~6.25} and \cite{MR1815411}*{Theorem~6.27} to $(M, d_M, \rho^2\volmm)$ then yields the desired self-adjoint operator satisfying (a). We denote by $\Delta^N_\rho$ this operator.

To show parts (b) and (c), it is sufficient to verify the assumptions of \cite{MR1815411}*{Theorem~7.3} for the measure $\rho^2\volmm$.
Using the inequality $\volmm \leq \alpha^{-2}\rho^2\volmm$, it follows from \cite{MR1815411}*{Theorem~5.5} and \cite{MR1815411}*{Theorem~5.7} that conditions (i)--(iii) in \cite{MR1815411}*{Section~5} hold for $\rho^2\volmm$.
It remains to show that conditions (7.4) and (7.5) in \cite{MR1815411} are satisfied.
Indeed, condition (7.4) follows from Theorem~\ref{lem:EstCEigenF}, and for the case $q=2$ inequality (7.5) is a direct consequence of
\begin{equation}\label{eq:WBoch}
  \Ric_{g_t} - 2\Hess(\log\rho_t)
  \ge -(m-1)K - 2 H,
\end{equation}
together with the Bochner inequality.
\end{proof}

\subsection{Error-controlling integrals of distortion on metric measure spaces}\label{sec:Graph}
This subsection introduces several integrals that quantify distortion in metric measure spaces.
We use these integrals to bound the approximation errors in Sections~\ref{sec:DiscMap}--\ref{sec:Limit}.

We employ the following integral to estimate error terms without assuming an upper sectional curvature bound.
Let $m \in \N$, $K \ge 1$, and let $(M, d_M)$ be a compact metric space.
\begin{Definition}\label{def:LpEuclid}
  Let $p\geq 1$. 
  We define
  \begin{equation}
    V_{p,r}(M) = V_{p,r}(M, d_M) = \left(\int_M \left( 1- \frac{\volmm(B(x, r))}{V_K(r)}\right)^p \,d\volmm(x)\right)^{\frac{1}{p}}
  \end{equation}
  for $r >0$.
\end{Definition}

To construct our weighted graphs introduced in the next subsection, we do not use the original metric $d_M$ but the following Borel pseudo-metrics, i.e., Borel functions $\tilde d : M\times M \to \R$ which are pseudo-metrics.
\begin{Definition}
  Let $\tau\in (0,1)$ and $L\geq 1$.
  We define a class $\mathcal{I}_{L,\tau}(M,d_M)$ consisting of Borel pseudo-metrics $\tilde d \colon  M \times M \to [0, \infty)$ such that
  \begin{equation}
    \tilde{d}(x, y) - \tau \leq d_M(x, y) \leq L\tilde{d}(x, y) + \tau
  \end{equation}
  holds for every $x, y \in M$.
  Let $\mathcal I_L (M, d_M)$ denote $\mathcal I_{L,0}(M, d_M)$.
\end{Definition}

\begin{remark}\label{rem:PMet}
  Let $\tilde d \in \mathcal I_L(M, d_M)$.
  Assume that $\{(M_t, d_{t})\}_{t=1}^\infty$ converges to $(M, d_M)$ in GH sense.
  Then the pull-back distance
  $\Phi_t^*\tilde d (x, y) = \tilde{d}(\Phi_t(x),\Phi_t(y))$
  satisfies $\Phi_t^*\tilde{d} \in \mathcal I_{L,L\delta_t}(M_t, d_t)$.
  We will use this fact in Section~\ref{sec:Limit}.
\end{remark}

We introduce the following integral to compare $\tilde{d}\in \mathcal I_{L,\tau}(M,d_M)$ with the original metric $d_M$.
\begin{Definition}\label{def:S}
  Let $\tilde d \colon M\times M\to [0, \infty)$ be a Borel pseudo-metric on $M$.
  Then, we define
  \begin{equation}
    S_r(M, \tilde d) = S_r(M, d_M, \tilde d) = \int_M \volmm\left(\tilde{B}(x, r)\setminus B(x, r)\right) \,d\volmm(x)
  \end{equation}
  for $r \in (0,1)$, where $\tilde B (x, r)= \{y \in M\colon  \tilde d( x, y ) < r\}$.
\end{Definition}

\begin{remark}\label{rem:V&S}
  Let $(M,d_M)$ be the mGH limit of a sequence of
  closed $m$-dimensional Riemannian manifolds
  $\{(M_t,d_{g_t})\}_{t=1}^\infty$ with a uniform Ricci curvature bound
  $\Ric_{g_t}\ge-(m-1)K$.
  Then, for every $p\ge1$ and $\epsilon>0$,
  \[
      V_{p,\epsilon}(M_t,d_{g_t}) \;\longrightarrow\;
      V_{p,\epsilon}(M,d_M)
      \quad(t\to\infty),
  \]
  e.g., by the Portmanteau theorem and Egorov's theorem.
  If, in addition, $\|\Sect_{g_t}\|_{L^\infty}\le K$ and
  $\inj_{g_t}\ge\epsilon$ for all $t$, the comparison estimate
\cite{MR1390760}*{IV Theorem~3.1 (1)} yields
  \[
      V_{p,\epsilon}(M,d_M)\le \epsilon^{2}C(m,K,D).
  \]
  Now fix a map $\iota\colon M\to\mathbb R^{d}$ and Riemannian immersions
  $\iota_t\colon M_t\to\mathbb R^{d}$ such that
  $\iota_t(x_t)\to\iota(x)$ whenever $\Phi_t(x_t)\to x$.
  Then
  \[
      S_{\epsilon}\bigl(M_t,d_{g_t},\iota_t^{*}d_{\mathbb R^{d}}\bigr)
      \longrightarrow
      S_{\epsilon}\bigl(M,d_M,\iota^{*}d_{\mathbb R^{d}}\bigr).
  \]
  In particular, if
  \[
      d_M \le L\,\iota^{*}d_{\mathbb R^{d}}
      \quad\text{and}\quad
      \int_{M_t}|II_t|_{\iota_t}\,d\vol_{g_t}\le S
      \ \ (t\in\mathbb N),
  \]
  for some constants $L,S\ge1$, then 
  \[
      S_{\epsilon}\!\bigl(M,d_M,\iota^{*}d_{\mathbb R^{d}}\bigr)
      \le
      C(m,K,D,L,S)\epsilon^{m+1}
  \]
by \cite{MR4804972}*{Lemma 3.2, (10)}, where $II_t$ is a second fundamental form for $\iota_t$.
  Consequently, in this situation, which is assumed in \cite{MR4804972}, we have the uniform bounds of $V_{p,\epsilon}(M)\epsilon^{-2}$ and $S_{\epsilon}(M,d_M,\iota^{*}d_{\mathbb R^{d}})\epsilon^{-m-1}$, and we also have $\iota^*d_{\R^d} \in \mathcal I_{L}(M, d_M)$.
\end{remark}

\subsection{Weighted graphs and their graph Laplacians constructed from data sets}

This subsection fixes notation for finite weighted graphs, their graph Laplacians, and how these graphs are built from i.i.d.\ samples drawn from a probability space with a pseudo-metric.
They provide the discrete objects required for the later sections.

Let $(V, E)$ be an undirected graph.
We assume $\#V < \infty$ and $\#E < \infty$ throughout this paper.
Let $w_V\colon  V\to [0, \infty)$, $w_E:E \to [0, \infty)$, and let $\epsilon>0$.
Then we call $\Gamma = (V,E,w_V,w_E,\epsilon)$ a \emph{weighted graph}.
For $x,y\in V$ define the graph distance
\[
  d_\Gamma(x,y)=
    \inf\bigl\{r\epsilon\colon x=x_0\sim x_1\sim\cdots\sim x_{r}=y\bigr\},
\]
and the discrete measure
\(\displaystyle\vol_\Gamma(W)=\sum_{x\in W}w_V(x)\).
The graph Laplacian is
\begin{equation}\label{eq:GLap}
  \Delta_\Gamma\phi(x)=
  \frac{2}{w_V(x)\epsilon^{2}}
  \sum_{\{x,y\}\in E}\bigl(\phi(x)-\phi(y)\bigr)\,w_E(\{x,y\}),
  \qquad \phi\in L^{2}(V,\vol_\Gamma).
\end{equation}
Putting \(\diffgraph\phi_{xy}=(\phi(x)-\phi(y))/\epsilon\),
we have
\[
  \langle\phi,\Delta_\Gamma\phi\rangle_{L^{2}(\vol_\Gamma)}
  =\sum_{x\in V}
     \sum_{y\colon \{x,y\}\in E}
     (\diffgraph\phi_{xy})^{2}\,w_E(\{x,y\})
  =:\|\diffgraph \phi\|^{2}_{L^{2}(\vol_\Gamma)}.
\]
Hence $\Delta_\Gamma$ is self-adjoint and non-negative with the eigenvalues denoted by 
\(
  0=\lambda_0(\Gamma)\le\lambda_1(\Gamma)\le\cdots\le
  \lambda_{|V|-1}(\Gamma).
\)

We now introduce the data sets from which we construct weighted graphs.
\begin{Definition}
  Let $(\Omega, \mathbb{P})$ and $(M, \mu)$ be probability spaces.
  Fix $n \in \N$.
  For independent random variables $x_1, \dots, x_n\colon  \Omega \to M$ distributed according to $\mu$,
we call $\data = (x_1, \dots, x_n)\colon  \Omega \to M^n$ a \emph{data set} drawn from $\mu$.
\end{Definition}

We fix $(\Omega, \mathbb{P})$ in the rest of this paper. For $\omega \in \Omega$ we sometimes use a notation $\data$ to denote $\{x_1(\omega),\dots,x_n(\omega)\}$.

We frequently use the following Bernstein-type inequality to approximate measurable functions by data sets.

\begin{lemma}\label{lem:Bern}
  For $f\in L^{\infty}(M,\mu)$ and $\delta>0$, setting
  \(
    \sigma^{2}=\int_M f^{2}\,d\mu-
               (\int_M f\,d\mu)^{2},
  \)
  we have
  \[
    \mathbb P\!\Bigl(
       \bigl|\tfrac1n\sum_{i=1}^n f(x_i)-\int_M f\,d\mu\bigr|
       \ge 2\|f\|_{\infty}\delta^{2}+4\sigma\delta
    \Bigr)\le 2e^{-n\delta^{2}}
  \]
  for any data set $\data = (x_1,\dots, x_n)\colon \Omega \to M^n$ from $\mu$.
\end{lemma}
\begin{proof}
  By the Bernstein inequality, for every $n$ random variables $X_1,\dots,X_n \colon  \Omega \to \R$, 
  if $\mathbb{P}( |X_i| \leq c) = 1$ and $E[X_i] = a$, then 
  \[
    \mathbb{P}\left( \left|\frac1n \sum_{i=1}^n X_i-a\right| >t \right) \leq 2\exp \left( \frac{-nt^2}{2\tilde\sigma^2 + \frac23 ct}\right)
  \]
  for any $t >0$, where $\tilde \sigma^2 =\frac1n \sum_{i=1}^n E[X_i^2]$. 
  Setting $X_i = f(x_i)$ ($i=1, \dots ,n$) and $t = 2\max\{\|f\|_\infty\delta^2, 2\sigma\delta\}$, 
  we obtain this lemma.
\end{proof}

This paper uses the following two constructions of weighted graphs.
\begin{Definition}\label{def:Graph}
  Let $(M, d_M)$ be a compact metric space with Hausdorff dimension $m$, and let $\tilde d\colon M\times M \to [0, \infty)$ be a Borel pseudo-metric.
  Fix $\data = \{x_1,\dots,x_n\} \subset M$.
  For $\epsilon>0$ set
\(
  E_\epsilon(\data)=\{\{x_i,x_j\}\subset X\colon  \tilde d(x_i,x_j)<\epsilon\}
\)
 and define
\[
\begin{aligned}
  \Gamma_{m,\epsilon} = \Gamma_{m, \epsilon}(\data, \tilde d)
  &=
  \Bigl(
    \data,E_\epsilon(\data),
    \tfrac1n,\,
    \tfrac{\volmm(M)}{n(n-1)\,\omega_m\epsilon^{m}},\,
    \epsilon
  \Bigr),\\[4pt]
  \Gamma_{\epsilon}^N =  \Gamma_{\epsilon}^{N}(\data, \tilde d)
  &=
  \Bigl(
    \data,E_\epsilon(\data),
    \tfrac{\deg(\,\cdot\,)}{n(n-1)\omega_m\epsilon^{m}},\,
    \tfrac1{n(n-1)\omega_m\epsilon^{m}},\,
    \epsilon
  \Bigr),
\end{aligned}
\]
where $\omega_m$ is the volume of the unit ball in $\mathbb R^{m}$, and
$\deg(x)$ is the degree of $x$ in $E_\epsilon$.
\end{Definition}
Their Laplacians are, respectively, scaled versions of the classical 
unnormalized graph Laplacians and random-walk graph Laplacians.

In Section~7, we prove that the eigenvalues and eigenfunctions of
$\Delta_{\Gamma_{m,\epsilon}}$ and $\Delta_{\Gamma_{\epsilon}^{N}}$ constructed from data sets
converge, with explicit rates, to those of the weighted Laplacians
$\Delta_{\rho}$ and $\Delta_{\rho}^{N}$ on a Riemannian manifold;
for $\Delta_{\rho}^{N}$ the result extends to non-collapsed Ricci limit spaces.

\section{\texorpdfstring{$L^p$}{Lp} bounds for the eigenfunctions of the graph Laplacians}\label{sec:LpEst}

Let $\Gamma = (V, E, w_V, w_E, \epsilon)$ be a weighted graph.
This section provides $L^p$ estimates of eigenfunctions for the graph Laplacian ($p>1$).
The estimates is used in Section~\ref{sec:IntMap}.

We consider the following two structures, which we will show on our weighted graphs constructed from data sets with high probability in the next section.
One of the structures is the following.
\begin{Definition}[Rough volume-doubling property]
  For a constant $Q\geq 1$,
  we say that \emph{$\Gamma$ satisfies the rough $Q$-volume-doubling property}
  if 
  \begin{equation}\label{eq:VolDoub}
    \volgraph(B_\Gamma(x, 2r))\leq Q \volgraph(B_\Gamma(x, r))
  \end{equation}
  for all $x\in V$ and $r>\epsilon$,
  where $B_\Gamma( y, l) =\{ z\in V: d_\Gamma(z, y) < l\}$ for every $y \in V$ and $l >0$.
\end{Definition}

We do not consider  $r \leq \epsilon$  in this definition since it is difficult to estimate the density of data sets locally in Section~\ref{sec:DataSet}.

For $\phi\colon V \to \R$ and for $W \subset V$, define 
\[
  \phi_W:=\frac{1}{\volgraph(W)}\sum_{x\in W}\phi(x)w_V(x)
\] 
and set norms
\begin{equation}
  \|\phi\|_{p, W} = 
  \begin{cases}
    \left(\frac{1}{\volgraph(W)}\sum_{x\in W} \phi(x)^p w_V(x)\right)^{1/p}, & \textrm{$p<\infty$,}\\
    \sup_{x\in W}|\phi(x)|, & \textrm{$p = \infty$.}
  \end{cases}
\end{equation}
Then the other structure is the next one.
\begin{Definition}[Poincar\'{e} inequality]
  For constants $P, \sigma\geq 1$, we say that \emph{$\Gamma$ satisfies the $(P, \sigma)$-Poincar\'{e} inequality} if
  \begin{equation}\label{eq:PoinIneq}
    \| \phi - \phi_{B_{\Gamma}(x, r)}\|_{2, B_\Gamma(x, r)} \leq rP \|\diffgraph \phi\|_{2, B_\Gamma(x, \sigma r)}
  \end{equation}
  for all $\phi\in L^2(V, \vol_\Gamma)$, $x\in V$, and $r>0$.
\end{Definition}

Next, we set \[\phi_s(x) = \frac{1}{\volgraph(B_{\Gamma} (x, s))}\sum_{y\in B_{\Gamma} (x, s)} \phi(y) w_V(y)\] for $\phi\colon  V\to \R$, $s>0$, and $x\in V$.
Using the above structures, we have the following lemma, close to \cite{MR4585410}*{Lemma 5.3}.
\begin{lemma}\label{lem:pPoincare}
  There exists $C = C(Q, P,\sigma)>0$ such that if $\Gamma$ satisfies \eqref{eq:VolDoub} and \eqref{eq:PoinIneq}, then
  \begin{equation}
    \|\phi - \phi_s\|_{2,V} \leq sC\| \diffgraph \phi \|_{2, V}
  \end{equation}
  holds for every $\phi\in L^2(V,\volgraph)$ and $s>0$.
\end{lemma}

\begin{proof}
  This lemma is trivial for $s<\epsilon$, so we assume $s>\epsilon$.
  Let $r \ge 3s$, and let $B=B_\Gamma(p, r)$ be a ball with radius $r$.
  The rough volume-doubling property implies
  \begin{align}
    \|\phi_{2B} - \phi_s\|^2_{2,B} &\leq \frac{1}{\volgraph(B)}\sum_{x\in B}\frac{1}{\volgraph(B_\Gamma(x,s))}\sum_{y\in B_\Gamma(x,s)}|\phi(y) - \phi_{2B}|^2 w_V(y)w_V(x)\\
                                   & \leq \frac{1}{\volgraph(B)}\sum_{y\in 2B}\left(\sum_{x\in B_\Gamma(y,s)}\frac{w_V(x)}{\volgraph(B_\Gamma(x,s))}\right)|\phi(y)-\phi_{2B}|^2w_V(y)\\
                                   &\leq Q^2 \|\phi -\phi_{2B}\|^2_{2, 2B},
  \end{align} 
  where $aB$ denotes $B_\Gamma(p, ar)$ for $B=B_\Gamma(p, r)$ and any $a > 0$. 
  Hence,
  \begin{equation}\label{eq:pPoincStep1}
    \begin{split}
      \|\phi - \phi_s\|_{2, B} & \leq \|\phi - \phi_{2B}\|_{2, B} + \|\phi_{2B} - \phi_s\|_{2,B}\\
                               & \leq \|\phi - \phi_{2B}\|_{2, B} + Q\|\phi_{2B} -\phi\|_{2, 2B}\\
                               & \leq 2Q\|\phi - \phi_{2B}\|_{2, 2B}
    \end{split}
  \end{equation}
  for every ball $B$ with radius $r \ge 3\epsilon$.
  There exists a covering $\{B_i\}_{i=1}^N= \{B_\Gamma(p_i,r) \}_{i=1}^N$ with $\frac{1}{3}B_i\cap \frac{1}{3}B_j = \emptyset$ for $i\neq j$.
  Then we have
  \begin{equation}\label{eq:pPoincStep2}
    \begin{split}
      \|\phi - \phi_s\|_{2,V}^2 & 
      \leq \sum_{i=1}^N \frac{\volgraph(B_i)}{\volgraph(V)}\|\phi - \phi_s\|_{2,B_i}^2\\
                                & \leq \sum_{i=1}^N \frac{\volgraph(B_i)}{\volgraph(V)}
                                4Q^2\|\phi - \phi_{2B_i}\|_{2,2B_i}^2\\
                                & \leq \sum_{i=1}^N \frac{\volgraph(B_i)}{\volgraph(V)}
                                16 r^2 Q^2P^2\|\diffgraph \phi \|_{2,2\sigma B_i}^2\\
                                & \leq \frac{16r^2Q^2P^2}{\volgraph(V)}\sum_{i=1}^N
                                \sum_{x\in 2\sigma B_i}|\diffgraph \phi|_x^2 w_V(x)\\
                                & = \frac{16r^2Q^2P^2}{\volgraph(V)}
                                \sum_{x\in V} \#\{i\in \{1, \dots, N\}\colon  x\in 2\sigma B_i\}|\diffgraph \phi|_x^2 w_V(x).
    \end{split}
  \end{equation}
  We used the inequality~\eqref{eq:pPoincStep1} in the second line and applied the Poincar\'{e} inequality to the third line.
  The last equality is changing the order of summations.

  Since $\{\frac{1}{3}B_i\}_{i=1}^N$ is disjoint family, the rough volume-doubling property implies
  \begin{align}
    &\#\{j\in \{1, \dots, N\}\colon  x\in 2\sigma B_j\}\volgraph (B_\Gamma(x, 2\sigma r + r/3))\\
    &\quad \leq  \sum_{j\colon x\in 2\sigma B_j}\volgraph (B_\Gamma(p_j, 4\sigma r + r/3))\\
    &\quad\leq  C(Q, \sigma )\sum_{j\colon x\in 2\sigma B_j} \volgraph\left(\frac{1}{3}B_j\right)\\
    &\quad\leq   C(Q, \sigma )\volgraph(B_\Gamma(x,2\sigma r + r/3)),
  \end{align}
  for every $x \in V$ and for every $r\ge 3s$.
  Hence $\#\{j\in \{1, \dots, N\}\colon  x\in 2\sigma B_j\} \leq C(Q, \sigma )$.

  Combining this and \eqref{eq:pPoincStep2} yields
  \begin{equation}
    \|\phi - \phi_s\|_{2,V}\leq rC(Q, P, \sigma )\|\diffgraph \phi \|_{2,V}
  \end{equation}
  for every $r \ge 3s$. Thus, we obtain this lemma.
\end{proof}

We will introduce additional structures to derive the next lemma.
For a constant $R\geq 1$, we say that $\Gamma$ is \emph{ the $R$-locally almost regular} if 
\begin{equation}\label{eq:AlmostReg}
  \frac{w_V(x)}{w_V(y)},\quad \frac{\deg(x)}{\deg(y)},\quad  \frac{w_E(xy)}{w_E(xz)} \leq R 
\end{equation}
for all $xy, xz \in E$.
Moreover, define $I\colon  L^2(V,\vol_\Gamma) \to L^2(V, \vol_\Gamma)$ by 
\[
  I\phi(x) = \frac{1}{\sum_{xy \in E} w_E(xy)}\sum_{xy \in E} w_E(xy) \phi(y)
\]
for every $\phi\in L^2(V, \vol_\Gamma)$ and $x\in V$.
Then, we have the following Nash-type inequality, close to \cite{MR4585410}*{Proposition~5.5}.

\begin{lemma}[Rough Nash-type inequality]\label{lem:weeknash}
  There exists a constant $C = C(R, Q, P, \sigma)>0$ such that if $\Gamma$ satisfies \eqref{eq:VolDoub}, \eqref{eq:PoinIneq}, and \eqref{eq:AlmostReg}, and has $\diam(V, d_\Gamma)\leq D$ for $D\geq 1$, then we have
  \begin{equation}
    \min\{\|\phi\|_{2, V},\|I\phi\|_{2, V} \}\leq 
    \left(C(D\|\diffgraph \phi\|_{2, V})^{\frac{\nu}{\nu +2}} + \|\phi\|_{1, V}^{\frac{\nu}{\nu +2}}\right) \|\phi\|_{1, V}^{\frac{2}{2+ \nu}}
  \end{equation}
  for all $\phi\in L^2(V, \vol_\Gamma)$, where $\nu = \log_2 Q$.
\end{lemma}

\begin{proof}
  We first show that there is a constant $C = C(Q, P, \sigma, R)>0$
  such that
  \begin{equation}\label{eq:ForNash}
    \min\{\|\phi\|_{2,V}, \|I\phi\|_{2, V} \} \leq (sC\|\diffgraph \phi\|_{2, V} + \max\{C(D/s)^{\nu/2}, 1\}\|\phi\|_{1, V})
  \end{equation}
  holds for every $\phi\in L^2(V, \vol_\Gamma)$ and every $s>0$.

  In the case of $s\in (\epsilon, D]$, the rough volume-doubling property yields
  \begin{equation}
    \|\phi_s\|_{1, V} \leq \frac{1}{\volgraph(V)}\sum_{y\in V}\left(\sum_{x\in B_{\Gamma}(y,s)}\frac{w_V(x)}{\vol_{\Gamma}(B_{\Gamma}(x,s))}\right)|\phi(y)|w_V(y) \leq Q\|\phi\|_{1, V},
  \end{equation}
  and
  \begin{equation}
    \|\phi_s\|_{\infty, V}  \leq Q \left(\frac{D}{s}\right)^\nu\|\phi\|_{1, V}.
  \end{equation}
  Combining these two inequalities, we have $\|\phi_s\|_{2, V} \leq\|\phi\|_{1,V}^{1/2}\|\phi\|_{\infty,V}^{1/2} \leq Q \left(\frac{D}{s}\right)^{\nu/2} \|\phi\|_{1, V}$.
  Then Lemma~\ref{lem:pPoincare} gives us 
  \begin{align}
    \|\phi\|_{2, V} &\leq \|\phi - \phi_s\|_{2, V} + \|\phi_s\|_{2, V}\\
                    &\leq C(Q, P, \sigma)\Bigl(s\|\diffgraph \phi\|_{2, V} + (D/s)^{\nu/2}\|\phi\|_{1, V}\Bigr).
                    \label{eq:NashStep1}
  \end{align}
  Hence we obtain \eqref{eq:ForNash} for $s\in (\epsilon, D]$.

  In the case of $s\in (0, \epsilon]$,
  we have  $\|I\phi\|_{\infty,V}\leq R^4\|\phi\|_{1, B_{\Gamma}(x, 2\epsilon)}$ and $\|I\phi\|_{1, V} \leq R^3 \|\phi\|_{1,V}$.
  Combining them with the rough volume-doubling property, we have
  \begin{align}
    \|I\phi\|_{2, V}^2 & \leq R^7Q\left(\frac{D}{2\epsilon}\right)^{\nu}\|\phi\|_{1, V}^2\\
                       & \leq (D/s)^{\nu}C(Q, R)\|\phi\|_{1, V}^2.\label{eq:NashStep2}
  \end{align}
  Hence \eqref{eq:ForNash} holds for $s \in (0, \epsilon]$

  Lastly, if $s>D$, we have $\|\phi_s\|_{2, V}\leq \|\phi\|_{1, V}$.
  Then, by Lemma \ref{lem:pPoincare},
  \begin{equation}
    \|\phi\|_{2,V}\leq sC(Q, P, \sigma)\|\diffgraph \phi\|_{2, V} + \|\phi \|_{1,V}.
  \end{equation}

  Therefore, we conclude \eqref{eq:ForNash} for any $s >0$.
  Setting 
  $s = \left(\frac{\|\phi\|_{1, V}D^{\nu/2}}{\|\diffgraph \phi\|_{2, V}}\right)^{\frac{2}{\nu + 2}}$
  in \eqref{eq:ForNash}, we obtain this lemma.
\end{proof}

For eigenfunctions of graph Laplacians, we obtain the following $L^p$-estimates ($p >1$) by applying Moser's iteration to our weighted graphs. 
These $L^p$-estimates will be used in Section~\ref{sec:IntMap}.
\begin{theorem}[Estimate for graph eigenfunctions]\label{thm:EstEigenfGraph}
  There exists a constant $C=C(R,Q,P,\sigma)>0$ such that the following holds.

  Assume $\Gamma$ satisfies the hypotheses of Lemma~\ref{lem:weeknash}, and define
  \[
    \alpha :=
    \max_{x\in V}
    \frac{w_V(x)}{\textstyle\sum_{xy\in E} w_E(xy)}.
  \]
  Let $\lambda\ge 1$ with $\lambda\alpha\epsilon^{2}\le 1$, and let 
  $\phi\colon V\to[0,\infty)$ be a non-negative function with
  \begin{equation}\label{eq:pEigenFunct}
    \Diffgraph_\Gamma \phi \le \lambda\phi.
  \end{equation}
  Then we have
  \begin{equation}\label{eq:EstEigenFunct}
    \|\phi\|_{p,V} \le
    p^{ 2\lambda\alpha\epsilon^{2}}
    \exp\bigl(C D\sqrt{\lambda}\bigr)
    \|\phi\|_{1,V}
  \end{equation}
  for every $p\ge 1$.
\end{theorem}

\begin{proof}
  By the inequality \eqref{eq:pEigenFunct},
  we have $(1-\alpha\lambda\epsilon^2/2)\phi\leq I\phi$, so
  \begin{align}
    \|\phi\|_{2^{k+1}, V} &= \|\phi^{2^k}\|_{2, V}^{1/2^k}\\
                          & \leq \frac{1}{1-\alpha\lambda\epsilon^2/2}\|I\phi^{2^k}\|_{2, V}^{1/2^k}.
  \end{align}
  Then, putting $\nu = \log_2 Q$,
  Lemma~\ref{lem:weeknash} implies
  \begin{equation}\label{eq:EstEigenFunctStep1}
    \|\phi\|_{2^{k+1}, V}\leq \frac{1}{1-\alpha\lambda\epsilon^2/2}
    \left( C( D\|\diffgraph \phi^{2^k}\|_{2, V} )^{ \frac{\nu}{\nu + 2} } + \|\phi^{2^k}\|_{2, V}^{ \frac{\nu}{\nu + 2} }\right)^{\frac{1}{2^k}}
    \|\phi^{2^k}\|_{1, V}^{\frac{2}{(\nu + 2)2^k}}.
  \end{equation}
  By easy computation, we have
  \begin{equation}\label{eq:EstEigenfGraphStep}
    (X^q-Y^q)^2\leq \frac{q^2}{2q-1}(X-Y)(X^{2q-1}-Y^{2q-1})
  \end{equation}
  for any $X, Y\geq 0$ and $q\in \N$.
  Combining \eqref{eq:pEigenFunct} with \eqref{eq:EstEigenfGraphStep} yields $\|\diffgraph \phi^q\|^2_{2, V}\leq\frac{q^2}{2q-1}{\lambda}\|\phi^q\|^2_{2, V}$.
  By this and \eqref{eq:EstEigenFunctStep1}, we have
  \begin{equation}
    \|\phi\|_{2^{k+1}, V} \leq \frac{1}{1-\alpha\lambda\epsilon^2/2}
    \left(C \left( \frac{2^kD\sqrt{\lambda}}{\sqrt{2^{k+1} -1}} \right)^{ \frac{\nu}{\nu + 2}} + 1\right)^{{2^{-k}}}\!\!\!\!\!
    \|\phi\|_{2^{k+1}, V}^{\frac{\nu}{\nu + 2}}\|\phi\|_{2^k, V}^{\frac{2}{\nu + 2}}.
  \end{equation}
  Since $1+x\leq e^x$ for $x\in\R$, we get
  \[
    \left(C \left( \frac{2^kD\sqrt{\lambda}}{\sqrt{2^{k+1} -1}} \right)^{ \frac{\nu}{\nu + 2}} + 1\right)^{{2^{-k}}} \!\!\!\!\!\!\!\!\leq \exp( 2^{-k/2} C(D\sqrt{\lambda})^{ \frac{\nu}{\nu + 2}}).
  \]
  These two inequalities imply
  \begin{equation}
    \|\phi\|_{2^{k+1}, V} \leq \frac{\exp\left( 2^{-k/2} C(D\sqrt{\lambda})^{ \frac{\nu}{\nu + 2}}\right)}{(1-\alpha\lambda\epsilon^2/2)^{\frac{\nu+2}{2}}}\| \phi \|_{2^k, V}.
  \end{equation}
  By iterating this formula, we obtain
  \begin{equation}
    \|\phi\|_{2^{k}, V} \leq \frac{\exp\left( C(D\sqrt{\lambda})^{ \frac{\nu}{\nu + 2} } \right)}{(1-\alpha\lambda\epsilon^2/2)^{k\frac{\nu+2}{2}}}\| \phi \|_{1, V}.
  \end{equation}
  For $k = \lceil \log_2 p \rceil$, we have
  \begin{equation}
    \left(1-\alpha\lambda\epsilon^2/2\right)^{-k}
    \leq \left( 1 + \frac{\alpha\lambda\epsilon^2}{2-\alpha\lambda\epsilon^2} \right)^{\log_2 2p } 
    \leq \exp\left( \frac{\alpha\lambda\epsilon^2}{2-\alpha\lambda\epsilon^2}(\log_2 2p)\right)
    \leq (2p)^{\frac{2\alpha\lambda\epsilon^2}{2-\alpha\lambda\epsilon^2}}.
  \end{equation}
  Therefore, we conclude the desired inequality \eqref{eq:EstEigenFunct}.
\end{proof}
If $\lambda = 0$, this theorem provides $L^\infty$-bounds $\|\phi\|_{\infty, V} \le\|\phi\|_{1, V}$.

\section{Datasets: The rough volume-doubling property and the Poincar\'{e} inequality}\label{sec:DataSet}
This section studies data sets drawn from the probability measure $\rho \vol_g$ on a manifold $(M,g)\in \RicClass$.
We prove that the weighted graphs built from data sets satisfy, with high probability, conditions \eqref{eq:VolDoub}, \eqref{eq:PoinIneq}, and \eqref{eq:AlmostReg}.
These properties allow us to apply the $L^p$-estimates of Theorem~\ref{thm:EstEigenfGraph} to the eigenfunctions for the graphs in Section \ref{sec:IntMap}.

We have the following discrete approximation maps from $M$ to the sets with high probability.
The following lemma generalizes \cite{MR4804972}*{Theorem C.2} and relaxes a condition $\alpha^{-1} \leq \rho \leq \alpha$ to  the weaker one $\max\rho/\min\rho \leq \alpha$.

\begin{lemma}\label{lem:ExistT}
  There exist constants
  $C=C(m,K,D)>0$ and $A=A(m,K,D,\alpha)>0$
  such that the following holds:
  Fix a closed Riemannian manifold $(M,g) \in \RicClass$.
  Let $\rho\colon M\to(0,\infty)$ be a Borel function with
  \begin{equation}\label{eq:ProbF}
    \int_M \rho\,d\vol_g = 1,
    \qquad
    \frac{\max\rho}{\min\rho}\le\alpha,
  \end{equation}
  and let $\tilde{\epsilon},\tilde{a}\in(0,1)$ with $\tilde{a}A\le 1$.
  Draw a data set
  $\data=(x_1,\dots,x_n)\colon \Omega\to M^n$
  from $\rho\,\vol_g$.
  Then, with probability at least
  \[
    1-\tilde{\epsilon}^{-m}C\exp\!\bigl(-n\tilde{a}^{2}\tilde{\epsilon}^{m}\bigr),
  \]
  there exists a Borel measurable map
  $T\colon M\longrightarrow \data$
  with $T(x_i)=x_i$ ($i=1,\dots,n$) such that
  \begin{equation}\label{eq:TDist}
    d_g\bigl(x,T(x)\bigr)<\tilde{\epsilon}
  \end{equation}
  for all $x\in M$, and
  \begin{equation}\label{eq:TMeas}
    \Bigl|
    \tfrac{1}{n}
    - \rho\vol_g\!\bigl(T^{-1}(\{x_i\})\bigr)
    \Bigr|
    <\frac{\tilde{a}A}{n}
  \end{equation}
  for every $i=1,\dots,n$.
\end{lemma}

\begin{proof}
  There exist $p_1, \dots, p_N \in M$ such that $\cup_{s=1}^N B(p_s, \tilde{\epsilon}/2) = M$ and $B(p_s, \tilde{\epsilon}/6) \cap B(p_t, \tilde{\epsilon}/6) = \emptyset$ for any $s \neq t$.
  By Theorem~\ref{thm:RiccComp}, $N \frac{V_K(\tilde\epsilon/6)}{V_K(D)}\leq \sum_{s=1}^N \frac{\vol_g(B(p_s, \tilde\epsilon/ 6))}{\vol_g(M)}\leq 1$ holds.
  Hence $N \le C(m,K,D)\tilde{\epsilon}^{-m}$.

  Define $\{V_s\subset M\}_{s=1}^N$ inductively by $V_1=B(p_1,{\tilde{\epsilon}}/{2})$ and $V_s = B(p_s, {\tilde{\epsilon}}/{2})\setminus \cup_{t<s}V_t$ for $s>1$.

  By Theorem~\ref{thm:RiccComp}, using \eqref{eq:ProbF}, we have 
  \[
    \epsilon^m \leq C(m, K, D, \alpha)\rho\vol_g(B(y_s,\tilde{\epsilon}/6))\leq C\rho\vol_g(V_s)
  \]
  for any $s \in \{1,2,\dots,n\}.$
  Combining this with Lemma~\ref{lem:Bern}, we have
  \begin{align}
      &\mathbb{P}\left( \,\left|\rho\vol_g(V_s) - \frac{\#V_s\cap \data}{n} \right| \geq \tilde{a} C(m, K, D, \alpha){\rho\vol_g(V_s)} \right) \\
      &\quad\leq \mathbb{P}\left(\,\left|\rho\vol_g(V_s) - \frac{\#V_s\cap \data}{n} \right| \geq 2\tilde a^2\tilde\epsilon^m + 4\tilde a\tilde\epsilon^{m/2}\sqrt{\rho\vol_g(V_s)} \right) \\
      &\quad\le 2\exp(- n\tilde a^2\tilde\epsilon^m).
  \end{align}
  By this and $N \leq C(m, K, D)\tilde\epsilon^{-m}$, we get
  \begin{equation}\label{eq:TStep2}
    \begin{split}
      &\mathbb{P}\left(\,
        \begin{aligned}
     &\textrm{there exists $s\in\{1,\dots,N\}$ such that}\\
     &\left|\rho\vol_g(V_s) - \frac{\#V_s\cap \data}{n} \right| \geq \tilde{a} C(m, K, D, \alpha){\rho\vol_g(V_s)}
        \end{aligned}
      \right)\\
      \leq& C(m,K,D)\tilde{\epsilon}^{-m}\exp(-na^2\tilde{\epsilon^m}).
    \end{split}
  \end{equation}

  Set $n_s:= \#V_s \cap \data$ and $\{x^s_1, \dots, x^s_{n_s}\} = V_s\cap \data$. 
  For each \( s\in \{1, \dots, N\} \), let \( W^s_1, \dots, W^s_{n_s} \subset V_s \) be measurable subsets with $x^s_t \in W^s_{t}$ such that equally divide the measure of \( V_s \) by means of $\rho\vol_g$.
  By the inequality \eqref{eq:TStep2}, $\frac{n_s}{n} \geq \rho\vol_g(V_s)( 1- \tilde a C)$ holds with probability at least $1-C(m,K,D)\exp(-n\tilde{ a}^2\tilde{\epsilon^m})\tilde{\epsilon}^{-m}$.
  Then, if $2\tilde  a C(m, K, D,\alpha)\leq 1$,  we have \[|\rho\vol_g(W^s_t) - 1/n| \leq \tilde a C(m, K, D,\alpha)/n\] for every $s\in\{1,\dots,N\}$ and every $t\in \{1,\dots,n_s\}$.
  We define $T:M \to \data$ by $T(x) = x^{s}_t$ for any $x \in W^{s}_t$.
  Then this $T$ is the desired map.
\end{proof}

\begin{remark}\label{rem:TMeas}
  Suppose the map $T$ in this lemma exists and $2\tilde{a}A\le1$.
  Then we have
  \begin{equation}
    \int_W (\phi \circ T) \rho\, d\vol_g \leq \frac{1 + \tilde{a} A}{n} \sum_{x_i \in \data \cap B(W, \tilde\epsilon)} \phi(x_i)
  \end{equation}
  and 
  \begin{equation}
    \frac{1}{n}\sum_{x_i \in \data \cap W} \phi(x_i) \leq (1 + 2\tilde{a} A) \int_{B(W, \tilde{\epsilon})} (\phi \circ T) \rho \,d\vol_g  
  \end{equation}
  for any Borel set $W\subset M$ and any $\phi\colon \data \to \R$.
\end{remark}

In the rest of this section, let $\tilde \epsilon, \tilde a, \tau \in (0, 1)$, $n, m \in \N$, and $A, K, D, L, \alpha \ge 1$.
Suppose that $(M, g) \in \mathcal M_m (K, D)$, $\rho\colon M \to (0, \infty)$ with \eqref{eq:ProbF},
and $\tilde{d} \in \mathcal I_{L, \tau}(M, d_g).$
Moreover, let $n \in \N$ and  $\data = \{x_1,\dots, x_n\} \subset M$,
and assume that we have $T\colon M \to \data$ with \eqref{eq:TDist} and \eqref{eq:TMeas}.

Thorough the map $T\colon M \to \data$, the original metric $d_g$ approximate $d_\Gamma$ for $\Gamma = \datagraph$ and $\datangraph$ as follows.
\begin{lemma}\label{lem:TDist}
  Let $\epsilon \in(0,1)$ with $4 \tau<{\epsilon}$ and $\tilde{\epsilon} \le \epsilon/8$ for \eqref{eq:TDist}.
  Then we have
  \begin{align}
    d_g(x_i, x_j) &\leq (L+1) d_\Gamma(x_i, x_j)\label{eq:TDist1},\\
    d_\Gamma(x_i, x_j) &\leq 4 d_g(x_i, x_j) + {\epsilon}\label{eq:TDist2}, 
  \end{align}
  for $i, j\in \{1,\dots,n\}$ and  $\Gamma = {\datagraph}, {\datangraph}$. 
\end{lemma}

\begin{proof}
  Let $y_0\dots y_l$ be a path of $\Gamma$ from $y_0 = x_i$ to $y_l = x_j$.
  Since $\tilde{d}\in \embed$, we have 
  \begin{equation}
    d_g(x_i, x_j) \leq \sum_{k=0}^{l-1} d_g(y_k, y_{k+1})
    \leq \sum_{k=0}^{l-1}( \tilde{d}(y_k,y_{k+1}) + \tau)
    \leq l\left(L\epsilon + \frac{\tau}{\epsilon}\right) .
  \end{equation}
  Taking the infimum of all the paths from $x_i$ to $x_j$, we obtain \eqref{eq:TDist1}.

  To prove the remaining part, set a geodesic curve $\gamma_{x_ix_j}\colon [0, 1]\to M$ from $x_i$ to $x_j$.
  Let $l$ be the minimum integer such that $l\geq \frac{d_g(x_i, x_j)}{\epsilon/2 - \tau}$.
  Then, by the choice of $\tilde{d}$, using the inequality \eqref{eq:TDist},
  \begin{equation}
    \tilde{d}\left(
      T\left( \gamma_{x_ix_j}  \left( \frac{k}{l} \right) \right) 
    , T\left(\gamma_{x_ix_j}\left(\frac{k+1}{l}\right)\right)\right)
    \leq {d}_g\left(
    T\left( \gamma_{x_ix_j}  \left( \frac{k}{l} \right) \right) , T\left(\gamma_{x_ix_j}\left(\frac{k+1}{l}\right)\right)\right) + \tau    < \epsilon.
  \end{equation}
  This implies that $T(\gamma_{x_ix_j}(0)) T(\gamma_{x_ix_j}(1/l)) \dots T(\gamma_{x_ix_j}(1))$ is a path of $\Gamma$ from $x_i$ to $x_j$.
  Hence we obtain
  \begin{equation}
    d_{\Gamma}(x_i, x_j) \leq \epsilon\frac{d_g(x_i, x_j)}{\epsilon/2 - \tau} + \epsilon
    \leq 4d_g(x_i, x_j) + \epsilon
  \end{equation}
  for any $x_i, x_j \in \data$.
\end{proof}

We will show that ${\Gamma_{m,\epsilon}}$ and ${\Gamma^N_{\epsilon}}$ are locally almost regular, and satisfy the rough volume-doubling property under the existence of the map $T$.
\begin{proposition}\label{prop:DataGraphVolD1}
  There exists a constant $C = C(m, K, D,\alpha, L)>0$ such that the following property holds.
  Let $\epsilon\in (0,1)$ with $4 \tau < {\epsilon}$.
  If $\tilde{\epsilon} \leq \epsilon/24$ and $2\tilde{a}A \le 1$ for \eqref{eq:TDist} and \eqref{eq:TMeas},
  then $\Gamma = \datagraph$ is $C$-locally almost regular and satisfies the rough $C$-volume-doubling property.
\end{proposition}

\begin{proof}
  We have $\diam (\data,{d_\Gamma}) \leq 4D +1$ by lemma~\ref{lem:TDist}.
  Hence, the inequality \eqref{eq:VolDoub} is obvious for $r > 4D +1$,
  so we can assume $r\leq 4D +1$.

  For $r > \epsilon/12$ and $x_i\in \data$, by the choice of $T$, using remark~\ref{rem:TMeas} and Theorem~\ref{thm:RiccComp},
  \begin{equation}\label{eq:VolDoubStep}
    \frac{\#B(x_i, 2r)\cap \data}{\# B(x_i, r)\cap \data}
    \leq 4\frac{\rho\vol_g(B(x_i, 2r + \epsilon/24))}{\rho\vol_g(B(x_i, r- {\epsilon}/24))}
    \leq 4\alpha \frac{\vol_g(B(x_i, 5r/2))}{\vol_g(B(x_i, r/2))}
    \leq 4\alpha \frac{V_K(5r/2)}{V_K(r/2)}.
  \end{equation}
  Hence, we have
  \begin{equation}
    \frac{\deg(x_i)}{\deg(x_j)} 
    \leq \frac{\#B(x_i,(L+1)\epsilon )\cap \data}{\# B(x_j, \frac{3\epsilon}{4})\cap \data}
    \leq C(m, K, D, L, \alpha)
  \end{equation}
  for $x_i, x_j \in \data$ with $\tilde{d}(x_i,x_j) < \epsilon$,
  using $4\tau < \epsilon$.
  Thus $\Gamma= \datagraph$ is $C(m, K, D, L, \alpha)$-locally almost regular.

  If $r \in [3\epsilon/2, 4D +1]$, by Lemma~\ref{lem:TDist}, using the inequality \eqref{eq:VolDoubStep}, we have
  \begin{equation}
    \frac{\volgraph(B_{\Gamma}( x,2r) )}{\volgraph(B_{\Gamma} (x, r))}
    \leq \frac{\#B(x, 2r(L+1))\cap \data}{\#B(x, \frac{r-\epsilon}{4}) \cap \data}
    \leq C(m, K, D, L, \alpha).
  \end{equation}
  Note that $\frac{\volgraph(B(x_i, 2r))}{\volgraph(B(x_i, r))}$ is constant for $r\in (\epsilon, 3\epsilon/2)$.
  Thus, a constant $C = C(m, K, D, L, \alpha)$ exists such that $\datagraph$ satisfies the rough $C$-volume-doubling property.
\end{proof}

\begin{proposition}\label{prop:DataGraphVolD2}
  Let $v \in (0, 1)$, there exists a constant $C = C(m, K, D,\alpha, L,v)>0$ such that the following property holds.
  Let $\epsilon\in (0,1)$ with $4 \tau < {\epsilon}$.
  If $\tilde{\epsilon} \leq \epsilon/24$ and $2\tilde{a}A \le 1$ for \eqref{eq:TDist} and \eqref{eq:TMeas},
  supposing $\vol_g(M) \ge v$,
  $\Gamma = \datangraph$ is $C$-locally almost regular and satisfies the rough $C$-volume-doubling property.
\end{proposition}

\begin{proof}
  Using Remark~\ref{rem:TMeas}, by Theorem~\ref{thm:RiccComp},
  \begin{equation}\label{eq:DegRegul}
    C(m, K, D,\alpha, L)^{-1}\epsilon^{m}\leq \frac{\deg(x_i)}{n} \leq \frac{C(m, K, L, \alpha)}{\vol_g(M)} \epsilon^m
  \end{equation}
  holds.
  Combining this with Proposition~\ref{prop:DataGraphVolD1} concludes this proposition.
\end{proof}

We will also  show that ${\Gamma_{m,\epsilon}}$ and ${\Gamma^N_{\epsilon}}$ satisfy the Poincar\'{e} inequality using the $T$.
\begin{proposition}\label{prop:DataPoincare1}
  Let $v \in (0, 1)$, there exist constants $C_1 = C_1(m, K, D,\alpha, L,v)>0$ and $C_2= C_2(L)>0$ such that the following property holds.
  Let $\epsilon\in (0,1)$ with $4 \tau < {\epsilon}$.
  If $\tilde{\epsilon} \leq \epsilon/24$ and $2\tilde{a}A \le 1$ for \eqref{eq:TDist} and \eqref{eq:TMeas},
  supposing $\vol_g(M) \ge v$,
  $\datagraph$ and $\datangraph$ satisfy the $(C_1, C_2)$-Poincar\'{e} inequality.
\end{proposition}

\begin{proof}
  We first prove this proposition for $\datagraph$.
  Let $\Gamma = \datagraph$.
  The inequality \eqref{eq:PoinIneq} is trivial for $r\in (0, \epsilon]$, so we assume $r>\epsilon$.
  Let $\phi\in L^2(V, \vol_\Gamma)$.
  We will show that for $x, y \in \data$ with $d_g(x, y)\leq \epsilon/2$,
  \begin{equation}\label{eq:DataPoincareStep1}
    |\phi(x)-\phi(y)|\leq \epsilon C(m,\alpha,K,D,v)(|\diffgraph \phi|_x + |\diffgraph \phi|_y).
  \end{equation}

  Since $\tilde{d} \in \mathcal{I}_{L,\tau}$, using $\epsilon> 4\tau$,
  \begin{equation}
    \tilde{B}(x, \epsilon)\cap \tilde{B}(y, \epsilon)\supset
    B\left(x, {\epsilon -\tau-d_g(x, y)}\right)
    \supset B(x, \epsilon/4),
  \end{equation}
  where $\tilde B(x,r)$ denotes the metric ball of $(M,\tilde d)$ at center $x\in M$ and radius $r>0$.
  Hence, using Remark~\ref{rem:TMeas}, by Theorem~\ref{thm:RiccComp}, we obtain
  \begin{align}
    \left|\frac{\phi(x) - \phi(y)}{\epsilon}\right| 
    \leq &\frac{\sum_{z\in \tilde{B}(x,\epsilon)\cap\tilde{B}(y,\epsilon)\cap\data} \left|\phi(x) - \phi(z)\right| + \sum_{z\in \tilde{B}(x,\epsilon)\cap\tilde{B}(y,\epsilon)\cap\data}\left| \phi(y) - \phi(z)\right|}{\epsilon\left( \#\tilde{B}(x,\epsilon)\cap\tilde{B}(y,\epsilon)\cap\data\right)}\\
    \leq &\left(\sum_{x_k\in \tilde{B}(x,\epsilon)}\frac{(\diffgraph \phi_{x_k x})^2}{\#B(x, \epsilon/4)\cap\data}\right)^{\frac12} + \left(\sum_{x_k\in \tilde{B}(y,\epsilon)}\frac{(\diffgraph \phi_{x_k y})^2}{\#B(y, \epsilon/4)\cap\data}\right)^{\frac12}\\
    \leq &C(m,K,D,\alpha,v)(|\diffgraph \phi|_x + |\diffgraph \phi|_y).
  \end{align}
  Thus the inequality \eqref{eq:DataPoincareStep1} holds for $d_g(x, y)\leq \epsilon/2$.

  Let $r>0$, and let $N \in \integer$ be the minimum integer such that $ \epsilon N> 3r(L+1)$. 
  For $x, y\in \data$ with $d_{\Gamma}(x , y) \leq r$, we have
  \begin{align}
    d_g\left(T\left(\gamma_{xy}\left(\frac{s}{N}\right)\right), T\left(\gamma_{xy}\left(\frac{s+1}{N}\right)\right)\right) 
    &\leq 2\tilde{\epsilon} + \frac{d_g(x, y)}{N}
    \leq \frac{\epsilon}{12} + \frac{(L+1)r}{N}
    < \frac{\epsilon}{2}
  \end{align}
  for every $s\in [0, N) \cap \integer$.
  Hence, by \eqref{eq:DataPoincareStep1}, we have 
  \begin{equation}\label{eq:DataPoincareStep2}
    |\phi(x) - \phi(y)|\leq \epsilon C(m,\alpha,K,D,v)\sum_{s=0}^N|\diffgraph \phi|_{T(\gamma_{xy}(\frac{s}{N}))} 
  \end{equation}
  for $d_{\Gamma}(x,y)\leq r$. 
  Therefore, for $B= B_{\Gamma}(p, r)$, we have 
  \begin{align}
    \|\phi- \phi_B\|_{2,B}^2 &\leq \frac{1}{n^2\volgraph(B)^2}\sum_{x\in B}\sum_{y\in B}|\phi(x)- \phi(y)|^2 \\
                             &\leq \frac{\epsilon^2 N C}{n^2\volgraph(B)^2}\sum_{s=0}^N \sum_{x\in B}\sum_{y\in B}|\diffgraph \phi|^2_{T\left(\gamma_{xy}\left(\frac{s}{N}\right)\right)}\\
                             &\leq \frac{\epsilon^2 N C }{\volgraph(B)^2}\sum_{s=0}^N 
                             \int_{B_M(B, \tilde{\epsilon})}\int_{B_M(B, \tilde{\epsilon})}|\diffgraph \phi|^2_{T\left(\gamma_{xy}\left(\frac{s}{N}\right)\right)}\rho(x)\rho(y)\, dxdy\\
                             &\leq \frac{\epsilon^2 N C }{\volgraph(B)^2}\sum_{s=0}^N 
                             \int_{B_M(B, \tilde{\epsilon})}\int_{B_M(B, \tilde{\epsilon})}|\diffgraph \phi|^2_{T\left(\gamma_{xy}\left(\frac{s}{N}\right)\right)}\rho\left(\gamma_{xy}\left(\frac{s}{N}\right)\right)^2\, dxdy\\
                             &\leq \frac{\epsilon^2 N^2 C}{\volgraph(B)^2} \int_{B_M(p, C(L)r)}|\diffgraph \phi|^2_{(T(x))}\, \rho^2 dx\\
                             &\leq r^2C\|\diffgraph \phi\|_{2,C(L)B}^2,
  \end{align}
  where we denoted $C(m,K,D,\alpha,L,v)$ by $C$.
  We used \eqref{eq:DataPoincareStep2} in the second inequality, Remark~\ref{rem:TMeas} in the third inequality.
  Theorem~\ref{thm:RiccComp} gives the fifth line similarly to the segment inequality \cite{MR3469435}*{Proposition~7.1.10}.
  Remark~\ref{rem:TMeas} and Lemma~\ref{lem:TDist} give the last line.
  Therefore $\Gamma = \datagraph$ satisfies the $C(m, K, D,\alpha, L,v)$-Poincar\'{e} inequality.

  The argument for $\datangraph$ is identical, using the inequality~\eqref{eq:DegRegul}.
\end{proof}

\section{Discretization maps: Lower bounds for the eigenvalues of Laplacians on Riemannian manifolds}
\label{sec:DiscMap}

For $m,n\in\N$ and $K,D,\alpha, \mathcal{L}\ge 1$,
draw a data set $\data\colon \Omega\to M^n$ from the probability measure $\rho\,\vol_g$ on a closed Riemannian manifold $(M,g)\in\mathcal{M}_m(K,D)$, where the density $\rho$ belongs to $\mathcal{P}(M:\alpha,\mathcal L)$.
For $L \ge 1$ and $\tau \in (0, 1)$,
let $\tilde d\in\mathcal I_{L,\tau}(M,d_g)$ be a pseudo-metric.
In this section,  for every Lipschitz function $f\colon M\to\R$ and for $\epsilon>0$,
we compare the continuous Dirichlet energy
\[
  \int_M |\diffcont f|^2 \rho^2 \, d\vol_g
\]
with the discrete Dirichlet energy
\[
  \frac{1}{n(n-1)\omega_m\epsilon^m}
  \sum_{i=1}^n\sum_{x_j\in\tilde B(x_i,\epsilon)}
  \bigl(\diffgraph(f|_{\data})_{ij}\bigr)^2,
\]
where $\tilde B(p,r)=\{y\in M\colon \tilde d(y,p)<r\}$ for any $p\in M$ and any $r>0$.
We carry out this comparison via the discretization map
\[
  \Lip(M)\;\ni\;f\;\mapsto\;f\bigl|_{\data}\colon \data\to\R.
\]

For every $k\in\N$, we will show that the comparison provides a sharp upper bound for the graph eigenvalue $\lambda_k\bigl(\Gamma_{m,\epsilon}(\data,\tilde d)\bigr)$ in terms of $\lambda_k(\Delta_\rho)$, and a similar estimate holds for $\lambda_k\bigl(\Gamma^N_\epsilon(\data,\tilde d)\bigr)$ using $\lambda_k(\Delta_\rho^N)$.

Define $\diffgraph f\colon M\times M\to\R$ by $\diffgraph f_{xy}=\dfrac{f(x)-f(y)}{\epsilon}$ for $x,y\in M$ and $f:M\to\R$.

We have the following approximations of integrals on $(M,g)$, which is close to \cite{MR4804972}*{Lemma~3.2} and \cite{MR4804972}*{Lemma~3.3}.

\begin{lemma}\label{lemma1ofdirichlet}
  There exists a constant $C=C(m,L)>0$
  such that for all $\epsilon,a\in(0,1)$ with $\tau<\epsilon\le 1/\sqrt{K}$,
  we have the two estimates hold:
  \begin{enumerate}
    \item For any Lipschitz function $f\colon  M \to \R$, we have
      \begin{equation}\label{eq:DirichletDisc}
        \begin{split}
          &\mathbb{P}\Bigl(
            \bigl|
            \tfrac1{n(n-1)}
            \textstyle\sum_{i=1}^n
            \sum_{x_j\in\tilde B(x_i,\epsilon)}
            (\diffgraph f_{x_ix_j})^2
            -\!\!
            \int_M\!\!\int_{\tilde B(x,\epsilon)}
            (\diffgraph f_{xy})^2
            \rho(x)\rho(y)\,dy\,dx
            \bigr|\\
          &\qquad >
          a\epsilon^mC (1 + \max\rho)(\Lip f)^2
        \Bigr)
        \le 2(en+1)e^{-na^2\epsilon^m}.
        \end{split}
      \end{equation}
    \item 
      Let $k\in\N$, and let $f_1,\dots,f_k\in\Lip(M)$.
      Then, with probability at least $1-2(ne+1)k^2e^{-na^2\epsilon^m}$, we have
      \begin{equation}
        \begin{split}
          &\Bigl|
          \tfrac1{n(n-1)}
          \textstyle\sum_{i=1}^n
          \sum_{x_j\in\tilde B(x_i,\epsilon)}
          (\diffgraph f_{x_ix_j})^2
          -\!\!
          \int_M\!\!\int_{\tilde B(x,\epsilon)}
          (\diffgraph f_{xy})^2
          \rho(x)\rho(y)\,dy\,dx
          \Bigr|\\
          &\quad\le
          a\epsilon^mCk(1 + \max\rho)
          \max_{1\le i\le k}(\Lip f_i)^2
        \end{split}
      \end{equation}
      for any $f=\sum_{s=1}^k a_s f_s$ with $\sum_{s=1}^k a_s^2=1$.
  \end{enumerate}
\end{lemma}

\begin{proof}
  We have $|f(x) - f(y)|\leq (\Lip f) (L+1)\epsilon$ for any $x, y\in M$.
  By Theorem~\ref{thm:RiccComp}, $\vol_g(\tilde{B}(x, \epsilon))\leq C(m,L)\epsilon^m$ for $x\in M$.
  Using these two inequalities, by Lemma~\ref{lem:Bern}, we obtain
  \begin{equation}
    \begin{split}
      &\mathbb{P}\Bigl(
        \bigl|
        \tfrac{1}{n-1}\textstyle\sum_{x_j\in\tilde{B}(x_i,\epsilon)}
        (\diffgraph f_{x_ix_j})^2 
        -
        \int_{\tilde{B}(x_i, \epsilon)} 
        (\diffgraph f_{x_iy})^2
        \rho(y)\, dy
        \bigr|\\
      &\qquad > 2a^2\epsilon^m(L+1)^2(\Lip f)^2 + 4 a\epsilon^m\sqrt{\max\rho}C(m,L) (\Lip f)^2
    \Bigr)
    \le  2\exp(-(n-1)a^2\epsilon^m)
    \end{split}
  \end{equation}
  for any $i \in \{1,\dots,n\}$.
  By Lemma~\ref{lem:Bern}, we also have
  \begin{equation}
    \begin{split}
      &\mathbb{P}\Bigl(
        \bigl|
        \tfrac{1}{n}\textstyle\sum_{i=1}^n\int_{\tilde{B}(x_i, \epsilon)}(\diffgraph f_{x_iy})^2 \, \rho(y) dy 
        - \int_M\int_{\tilde{B}(x, \epsilon)} (\diffgraph f_{xy})^2 \rho(x)\rho(y)\, dxdy
        \bigr|\\
      &\qquad > a\epsilon^{3m/2}(\max\rho)C(m,L) (\Lip f)^2
    \Bigr)
    \le   2\exp(-na^2\epsilon^m).
    \end{split}
  \end{equation}
  By these two estimates, we obtain (i). 
  Through the polarization identity, (i) implies the rest of this Lemma (ii).
\end{proof}

The following lemma gives a comparison $\int_M\int_{\tilde{B}}\diffgraph f^2 \, \rho\vol_g^2$ with $\int_M |\diffcont f|^2\, \rho\vol_g$.
This lemma modifies \cite{MR4804972}*{Lemma 3.2} and \cite{MR4804972}*{Lemma 3.3} using the integral $S_\epsilon$.
\begin{lemma}\label{lemma2ofdirichlet}
  There exists a constant $C = C(m,K,D,\alpha,\mathcal L)>0$ such that the for all $\epsilon\in (0,1)$ with $\tau<\epsilon\leq \frac{1}{\sqrt{K}}$,
  we have
  \begin{equation}\label{eq:Lem2Diri}
    \begin{split}
      &\int_M\int_{\tilde{B}(x, \epsilon)}\left(\diffgraph f_{xy}\right)^2\rho(x)\rho(y)\,d\vol_g(y)\vol_g(x)\\
      &\quad\leq  \frac{\omega_m \epsilon^{m}(1+\epsilon C)}{m+2}\int_M |\diffcont f|^2\rho^2\, d\vol_g + {S_{\epsilon}(M, \tilde d)}(\max\rho)^2(L+1)^2(\Lip f)^2
    \end{split}
  \end{equation}
  for $f\in \Lip(M)$.
\end{lemma}

\begin{proof}
  We begin with
  \begin{equation}\label{eq:grad-decomp}
    \begin{split}
      &\int_M\int_{\tilde B(x,\epsilon)}
      (\diffgraph f_{xy})^{2}\rho(x)\rho(y)\,dy\,dx\\
      &\quad\le
      \int_M\int_{B(x,\epsilon)}
      (\diffgraph f_{xy})^{2}\rho(x)\rho(y)\,dy\,dx
      +
      S_{\epsilon}(M,d_g,\tilde d)
      \bigl(\max\rho\bigr)^{2}
      (L+1)^{2}(\Lip f)^{2},
    \end{split}
  \end{equation}
  where the term $S_\epsilon(M,d_g,\tilde d)$ is defined in Definition~\ref{def:S}.

  Whenever $d_g(x,y)<\epsilon$, the bound
  \(
  \rho(x)\le\rho(y)\bigl(1+\alpha V_K(D)\mathcal L\,\epsilon\bigr)
  \)
  follows from $1\le(\min\rho)\,\alpha V_K(D)$.
  Combining this with Theorem~\ref{thm:RiccComp} yields
  \begin{align}
    &\int_M\!\!\int_{B(x,\epsilon)}
    (\diffgraph f_{xy})^{2}\rho(x)\rho(y)\,dx\,dy \notag\\
    &\quad=\epsilon^{-2}
    \int_M\int_{U_xM}\int_{0}^{\min\{t(u),\epsilon\}}
    \Bigl(
      \int_{0}^{r}(f\circ c_u)'(t)\,dt
    \Bigr)^{\!2}
    \Theta_u(r)\,
    \rho\bigl(c_u(0)\bigr)\,\rho\bigl(c_u(r)\bigr)
    \,dr\,du\,dx\\
    &\quad\le
    \epsilon^{-2}\!\bigl(1+\epsilon C(m,K,D,\alpha,\mathcal L)\bigr)
    \!\int_M\!\int_{U_xM}\!\int_{0}^{\epsilon}
    r^{m}
    \int_{0}^{r}
    \langle\diffcont f,c_u'(t)\rangle^{2}
    \rho\bigl(c_u(t)\bigr)^{2}
    \,dt\,dr\,du\,dx\\
    &\quad=
    \epsilon^{-2}\!\bigl(1+\epsilon C(m,K,D,\alpha,\mathcal L)\bigr)
    \!\int_{0}^{\epsilon}\!r^{m}
    \int_{0}^{r}\!\int_M\!\int_{U_xM}
    \langle\diffcont f,u\rangle^{2}
    \rho\bigl(c_u(t)\bigr)^{2}
    \,du\,dx\,dt\,dr.
  \end{align}
  For any continuous $F:UM\to\R$ one has
  \[
    \int_M\int_{U_xM} F\!\bigl(c_u'(t)\bigr)\,du\,dx
    \;=\;
    \int_M\int_{U_xM} F(u)\,du\,dx,
  \]
see, for instance, \cite{MR496885}*{Eq.\ (1.125)}.
  Hence, with
  \(
  \int_{U_xM} \langle\diffcont f,u\rangle^{2}du
  =\omega_m\,|\diffcont f|_x^{2},
  \)
  we get
  \[
    \int_M\!\!\int_{B(x,\epsilon)}
    (\diffgraph f_{xy})^{2}\rho(x)\rho(y)\,dx\,dy
    \;\le\;
    \frac{\omega_m\epsilon^{m}\bigl(1+\epsilon C(m,K,D,\alpha,\mathcal L)\bigr)}
    {m+2}
    \int_M|\diffcont f|^{2}\rho^{2}\,d\vol_g.
  \]
  Thus, combining this with \eqref{eq:grad-decomp}, we obtain the inequality~\eqref{eq:Lem2Diri}.
\end{proof}

Then we can compare the continuous and discrete Dirichlet forms using the data sets.
\begin{proposition}\label{prop:dirichletofcont}
  There exist constants $C_1 = C_1(m,L)>0$ and $C_2= C_2(m, K, D,\alpha, \mathcal{L})>0$  
  such that for all $\epsilon, a\in (0,1)$ with $\tau < \epsilon <1/\sqrt{K}$,
  we have the following estimate:
  Let $k \in \N$ and $f_1,\dots, f_k \in \Lip(M)$.
  With probability at least $1-2(ne+1)k^2\exp(-na^2\epsilon^m)$,
  \begin{equation}
    \begin{split}
    &\frac{1}{n(n-1)\omega_m \epsilon^m}\sum_{i=1}^n\sum_{x_j \in \tilde B (x_i, \epsilon)} 
    (\diffgraph f_{x_ix_j})^2 
    - \frac{1 + C_2 \epsilon}{m+2}\int_M 
    |\diffcont f|^2\rho^2 \, d\vol_g\\
    &\quad \leq \bigl(a(1 + \max\rho) + \epsilon^{-m}S_{\epsilon}(M,d_g,\tilde{d})(\max\rho)^2\bigr)C_1k \max_{1\leq i\leq k}(\Lip f_k)^2,
    \end{split}
  \end{equation}
  holds for every $f = \sum_{s=1}^ka_s f_s$ satisfying $\sum_{s=1}^k a_s^2 = 1$.
\end{proposition}

\begin{proof}
  We obtain this proposition by combining Lemma~\ref{lemma1ofdirichlet} and Lemma~\ref{lemma2ofdirichlet}.
\end{proof}

We can also compare the continuous and discrete $L^2$-norms of Borel functions on $M$.
\begin{proposition}\label{prop:normofcont}
  There exists a positive constant $C = C(m, K, D, L, \mathcal{L})$ such that for all $\epsilon, a\in (0, 1)$, 
  if $\tau < \epsilon < {1}/{\sqrt{K}}$, the following estimate holds:
  Let $k \in \N$, and let functions $f_1,\dots,f_k \in L^2 (M, \vol_g)$.
  Then, with probability at least $1-(2ne+4)k^2\exp(-na^2\epsilon^m)$, we have
  \begin{equation}\label{eq1ofcontnorm}
    \left| \frac{1}{n}\sum_{i=1}^n f(x_i)^2 - \int_M f^2 \rho\, d\vol_g \right|
    \leq 3a\epsilon^{m/2} k \max_{1 \leq l \leq k}\{\|f_l\|_{\infty}\}
  \end{equation}
  and
  \begin{equation}\label{eq:2ofcontnorm}
    \begin{split}
      &\left| \frac{1}{n(n-1)\omega_m\epsilon^m}\sum_{i=1}^n f(x_i)^2 \deg(x_i) - \int_M f^2\rho^2 \, d\vol_g \right|\\
      &\quad\leq  C\left((a+\epsilon)( 1+ \max\rho) + (\tau \epsilon^{-1} + {V_{1,\epsilon}(M)} +  \epsilon^{-m}S_\epsilon(M,\tilde d))(\max\rho)^2\right) k \max_{1 \leq l \leq k}\{\|f_l\|^2_{\infty}\}
    \end{split}
  \end{equation}  
  for any $f=\sum_{s=1}^k f_s a_s $ with $\sum_{s=1}^k a_s^2$.
\end{proposition}

\begin{proof}
  We have the inequality \eqref{eq1ofcontnorm} with probability at least $1-2e^{-na^2\epsilon^m}$ from Lemma~\ref{lem:Bern}.
  It remains to prove \eqref{eq:2ofcontnorm}.
  Similar to Lemma~\ref{lemma1ofdirichlet}, with probability at least $1 - 2(ne + 1)k^2 \exp( -na^2 \epsilon^m )$,
  \begin{equation}\label{eq1ofnormofcont}
    \begin{split}
      &\left| \frac{1}{n(n-1)}\sum_{i=1}^n f(x_i)^2 \deg(x_i) - \int_M f(x)^2(\rho\vol_g)(\tilde{B}(x, \epsilon))\rho(x) \, dx \right|\\
      &\quad\le a\epsilon^m C(m,L) (1 + \max\rho) k\max_{1 \leq l \leq k}\{\|f_l\|^2_{\infty}\}
    \end{split}
  \end{equation}
  holds for any $f= \sum_{s=1}^k f_s a_s$ satisfying $\sum_{s=1}^k a_s^2 = 1$.
  Next, we have
  \begin{align}
    &\left|(\rho\vol_g)(\tilde{B}(x, \epsilon)) - \rho(x)V_K(\epsilon)\right| \\
    &\quad\leq 
    \begin{aligned}[t]
    &\left|(\rho\vol_g)(\tilde{B}(x, \epsilon)) - \rho(x)\vol_g(\tilde{B}(x, \epsilon))\right|\\
    & +\rho(x)\vol_g\left(\tilde{B}(x, \epsilon)\triangle B(x, \epsilon)\right)\\
    & +\rho(x)\Bigl(V_K(\epsilon)- \vol_g(B(x, \epsilon)) \Bigr)
    \end{aligned}\\
    &\quad = 
    \begin{aligned}[t]
      & \;\epsilon\mathcal L \vol_g(\tilde B (x, \epsilon))\\
      & + (\max\rho)\vol_g\left(\tilde{B}(x, \epsilon)\setminus B(x, \epsilon)\right) +  (\max\rho)\vol_g\left(B(x, \epsilon + \tau) \setminus B(x, \epsilon)\right)\\
      & + (\max\rho)\Bigl(V_K(\epsilon)- \vol_g(B(x, \epsilon)) \Bigr),
    \end{aligned}
  \end{align}
  where $A\triangle B$ denotes the symmetric difference of any two sets $A$ and $B$.
  Then, By Theorem~\ref{thm:RiccComp}, we have
  \begin{equation}\label{eq2ofnormofcont}
    \begin{split}
    &\left|\int_M f(x)^2 (\rho\vol_g)(\tilde{B}(x,\epsilon)) \rho(x) \, d\vol_g(x) - V_K(\epsilon)\int_M f^2 \rho^2\,d\vol_g\right|\\
    &\quad\leq  \|f\|_\infty^2 \Bigl(\epsilon^{m+1}  C(m, L) \mathcal{L} + \tau \epsilon^{m-1}C(m)\max\rho +\bigl(S_\epsilon(M, \tilde d)  + \epsilon^m C(m) V_{1,\epsilon}(M)\bigr)(\max\rho)^2 \Bigr).
    \end{split}
  \end{equation}
  Combining \eqref{eq2ofnormofcont} with \eqref{eq1ofnormofcont}, we obtain \eqref{eq:2ofcontnorm}.
\end{proof}

The main theorem in this section is the followings.
These properties (i) and (ii) in Theorem~\ref{thm:EigenvDFromC} are close to \cite{MR4804972}*{Lemma 3.13} and \cite{MR4804972}*{Lemma 3.24}, respectively, but Theorem~\ref{thm:EigenvDFromC} does not rely on injectivity radius.
Instead, it uses an upper bound $ H$ of $\Hess({\log \rho})$.
Moreover, (i) in Theorem~\ref{thm:EigenvDFromC} is independent to a lower bound $v$ of $\vol_g(M)$ if $S_\epsilon(M,\tilde d) = 0$,
and (ii) in this theorem holds for all $(M, g) \in \RicClassBV$ without an upper bound of sectional curvature;
These assumptions are weaker than \cite{MR4804972}*{Lemma 3.13} and \cite{MR4804972}*{Lemma 3.24}.

\begin{theorem}\label{thm:EigenvDFromC}
  Let $\epsilon, a\in (0,1)$ with $\tau < \epsilon$, and let $k \in \N$.
  Then, we have the following properties.
  \begin{enumerate}
    \item For $ H \in \R$, there exist $C_1 = C_1(m,k,K,D,\alpha)>0$ and $C_2 = C_2(m,k, K, D,L, \alpha, \mathcal{L},  H)>0$ such that assuming $\rho \in \mathcal{P}(M\colon \alpha,\mathcal L, H)$ and $\diam (M, d_g) \ge D^{-1}$, if $a^2\epsilon^m C_1\le1$, then we have
      \begin{equation}
        \mathbb{P}
        \left(
          \begin{aligned}
          &(m + 2)\lambda_k(\datagraph)\\
          &\quad\leq \lambda_k(\Delta_\rho) + \left( a + \epsilon + \epsilon^{-m}S_{\epsilon}(M)(\max\rho) \right)C_2 
          \end{aligned}
        \right)
        \geq 1- (4ne + 6)k(k+1)e^{-na^2\epsilon^m}.
      \end{equation}
    \item For $v \in (0, 1)$ and $ H \in \R$, there exist $C_3 = C_3(m,k,K,D,v,\alpha,\mathcal{L},L)>0$ and $C_4 = C_4(m,k, K, D, \alpha, \mathcal{L},  H, L,v)>0$ such that assuming $(M,g) \in \RicClassBV$ and $\rho \in \mathcal P(M \colon \alpha, \mathcal L,  H)$, if
      $\left(a + \epsilon\right)C_3 \leq 1$,
      then we have
      \begin{equation}
        \mathbb{P}
        \left(
          \begin{aligned}
          &(m + 2)\lambda_k(\datangraph)\\
          &\quad \leq \lambda_k(\Delta_\rho^N) + \left( a + \epsilon + \tau \epsilon^{-1} + \epsilon^{-m}S_{\epsilon}(M) + V_{1, \epsilon}(M)\right)C_4
          \end{aligned}
        \right)
        \geq 1- (4ne + 6)k(k+1)e^{-na^2\epsilon^m}.
      \end{equation}
  \end{enumerate}
\end{theorem}

\begin{proof}
  We will show (i).
  We can derive (ii) similarly.
  Let $f_1, \dots, f_k \in L^2(M, \rho\vol_g)$ be orthonormal functions satisfying $\Delta_\rho f_j = \lambda_j(\Delta_\rho) f_j$ for each $j$.
  We have $\vol_g(M) \max\rho \le \alpha$.
  Hence, by Proposition~\ref{prop:dirichletofcont},  Proposition~\ref{prop:normofcont}, and Lemma~\ref{lem:EstCEigenF},
  there are two constants $C_1=C_1(m,k,K,D,\alpha)>0$ and $C_2=C_2(m,k,K,D,L,\alpha, H,\mathcal{L})>0$ such that,
  if $a^2\epsilon^m  C_1\leq1$,
  \begin{equation}
    \frac{(m+2)\sum_{i=1}^n \sum_{x_j\in \tilde{B}(x_i,\epsilon)} (\diffgraph f_{ij})^2 }{\sum_{i=1}^n f(x_i)^2 {(n-1)\omega_m\epsilon^m}} \leq \int_M |\diffcont f|^2\rho^2\,d\vol_g + \frac{ a + \epsilon + \epsilon^{-m}S_{\epsilon}(M)(\max\rho)}{\vol_g(M)}C_2
  \end{equation}
  for any $f \in \spanv\{f_1, \dots,f_k\}$ with $\Lnorm{f}{2}{M}{\rho\vol_g}^2$=1.
  By the min-max principle, taking the supremum of the above inequality, we obtain (i).
\end{proof}

\begin{remark}\label{rem:UpperEigenGraph}
   By Remark~\ref{rem:EstCEigenV}, under the conditions of Theorem~\ref{thm:EigenvDFromC}, we have
  \[
    \lambda_k(\datagraph), \lambda_k(\datangraph) \leq C(m,k,K,D,v,\alpha,\mathcal{L}, H,L)
  \]
  with the same probability as in this Theorem.
\end{remark}

We can say that our weighted graphs $\datagraph$ and $\datangraph$ constructed from data sets $\data$ satisfy the conditions of Theorem~\ref{thm:EstEigenfGraph} with high probability by Propositions~\ref{prop:DataGraphVolD1}, \ref{prop:DataGraphVolD2}, and \ref{prop:DataPoincare1} using Theorem~\ref{lem:ExistT} with $\tilde \epsilon = {\epsilon}/{24}$ and $\tilde a = 24^{{m}/{2}}a$.
Hence, combining Theorem~\ref{thm:EstEigenfGraph} with Theorem~\ref{thm:EigenvDFromC} and Remark~\ref{rem:UpperEigenGraph} yields the following theorem.
\begin{theorem}\label{thm:EstEigenfData}
  For $k\in \N$, $ H\geq 1$, and $v\in (0,1)$, there exist constants $C_1=C_1(m,K,D,k)>0$ and $C_2 = C_2(m,k,K,D,L,\alpha,\mathcal{L}, H,v)>0$ such that given $\epsilon, a\in (0,1)$ with $4 \tau < \epsilon$, assuming $(M,g) \in \RicClassBV$ and $\rho \in \mathcal{P}(M\colon \alpha,\mathcal{L}, H)$, we have the following estimates.
  \begin{enumerate}
    \item   Let $\Gamma = \datagraph$.
      If $(a + \epsilon) C_2\leq 1$,
      with probability at least $1-C_1(\epsilon^{-m}+n)\exp(-na^2\epsilon^m)$ we have
      \begin{equation} 
        \Lnorm{\phi}{p}{\data}{\vol_{\Gamma}}\leq C_2 p^{C_2 \epsilon^2}\Lnorm{\phi}{1}{\data}{\vol_{\Gamma}}
      \end{equation}
      for every eigenfunction $\phi\colon \data \to \R$ of $\Delta_\Gamma$ associated with $\lambda_k(\Gamma)$ and for every $p \geq 1$.
    \item   Let $\Gamma = \datangraph$. If $(a + \epsilon) C_2 \leq 1,$
      with the same probability bound as in (i),
      we have 
      \begin{equation}
        \frac{\Lnorm{\phi}{p}{\data}{\vol_{\Gamma}}}{\vol_\Gamma(\data)^{\frac{1}{p}}}\leq C_2 p^{C_2 \epsilon^2}\frac{\Lnorm{\phi}{1}{\data}{\vol_{\Gamma}}}{\vol_\Gamma(\data)}
      \end{equation}
      for every eigenfunction $\phi\colon \data \to \R$ of $\Delta_\Gamma$ associated with $\lambda_k(\Gamma)$ and for every $p \geq 1$.
  \end{enumerate}
\end{theorem}

\section{Interpolation maps: Upper bounds for the eigenvalues of Laplacians on Riemannian manifolds}\label{sec:IntMap}

Let $v\in (0,1)$, and let $(M, g) \in \RicClassBV$.
Fix a density $\rho \in \mathcal P(M\colon \alpha,\mathcal L)$, and
draw a data set $\data\colon \Omega \to M^n$ from the measure $\rho\vol_g$.
Let $\tilde d \in \embed$.
This section compares the discrete and continuous Dirichlet energy using the following interpolation map.
This map is specific to \cite{MR4804972}*{Definition 2.2}.
Similar to Section~\ref{sec:DiscMap}, for each $k\in \N$, we will show that the comparison yields sharp upper bounds on $\lambda_k(\Delta_\rho)$ and $\lambda_k(\Delta_\rho^N)$ in terms of $\lambda_k(\Gamma_{m, \epsilon}(\data, \tilde d))$ and $\lambda_k( \Gamma_\epsilon^N(\data, \tilde d))$, respectively, with high probability.
\begin{Definition}\label{def:IMap}
  For $\epsilon\in (0,1)$, 
  define $\psi_\epsilon\colon M \times M \to \R$ by 
  \[
    \psi_\epsilon (x, y)= 
    \left\{
      \begin{aligned}
        &\frac{1}{2}\left(1-\left(\frac{d_g(x,y)}{\epsilon}\right)^2\right) ,  \quad & \textrm{if $d_g(x,y) \leq \epsilon$},\\
        & 0, \quad & \textrm{otherwise}.
    \end{aligned}\right.
  \]
  and define $\theta_{n,\epsilon}\colon M \to \R$ by $\theta_{n,\epsilon}(x) = \frac{1}{n-1}\sum_{i=1}^n\psi_\epsilon(x,x_j)$ for $x\in M$.
  Moreover, set $\tilde{\psi}_\epsilon\colon  M\times M \to \R$ by $\tilde{\psi}_\epsilon(x,y) = {\psi_\epsilon (x,y)}/{\theta_{n,\epsilon}(x)}$.
  Then, we define \emph{an interpolation map}
  \begin{equation}
    \Lambda_\epsilon \colon  \mathrm{Map}(\data,\R) \to \Lip(M\setminus \theta_{n,\epsilon}^{-1}(\{0\}))
  \end{equation}
  by $\Lambda_\epsilon \phi(x) = \frac{1}{n-1}\sum_{i=1}^n \tilde{\psi}_\epsilon(x,x_i) \phi(x_i)$ for $x \in M\setminus \theta_{n,\epsilon}^{-1}(0)$.
\end{Definition}

To approximate $\theta_{n,\epsilon}$, we set 
\begin{equation}
  \theta_\epsilon(x) = \int_M \psi(x,y) \rho(y)\,d\vol_g(y)
\end{equation}
for $x \in M$ .
The following lemma compares $\theta_{n,\epsilon}$ and $\theta_\epsilon$.

\begin{lemma}\label{lem:InterNormApp}
  There exist constants $C_1 = C_1(m,K,D)>0$ and $C_2=C_2(m, K, D,\alpha,v)>0$ such that for every $\epsilon, a \in (0,1)$ with $(\epsilon + a) C_2\leq 1$,
  the following holds with probability at least $1-(\epsilon^{-2m} + n)C_1\exp( -na^2\epsilon^m)$:
  \begin{enumerate}
    \item
      $|\theta_{n, \epsilon}(x) - \theta_\epsilon (x)|
      \leq  \epsilon^m(\epsilon + a)C_2$ for every $x\in M$;
    \item
      $|\diffcont\theta_{n, \epsilon}(x_i) - \diffcont\theta_\epsilon (x_i)| 
      \leq \epsilon^{m-1} a C_2$ for $i\in \{1,\dots,n\}$.
  \end{enumerate}
\end{lemma}

\begin{proof}
  Let $p_1, \dots, p_N \in M$ with $\cup_{s=1}^N B(p_s, {\epsilon}^2) = M$ and $B(p_s, {\epsilon}^2/3) \cap B(p_t, {\epsilon}^2/3) = \emptyset$ for any $s \neq t$.
  Then, similarly to the proof of Lemma~\ref{lem:ExistT}, we have $N\leq C(m,K,D)\epsilon^{-2m}$.
  Then, using Lemma~\ref{lem:Bern}, 
  \begin{equation}
    \begin{split}
    &\mathbb{P}\left(
      \begin{aligned}
        &\textrm{For every $s\in\{1,\dots,N\}$,}\\
        &|\theta_{n, \epsilon}(p_s) - \theta_\epsilon (p_s)|\leq a \epsilon^m C(m,K,D,\alpha,v),\\
        &\left|\frac{\# B(p_s,\epsilon+\epsilon^2)}{n} - \rho\vol_g(B(p_s,\epsilon+\epsilon^2)\right| \leq  a \epsilon^m C(m,K,D,\alpha,v).
      \end{aligned}
    \right)\\
      &\qquad\geq  1-\epsilon^{-m}C(m,K,D)\exp(-na^2\epsilon^m)
    \end{split}
  \end{equation}
  holds.
  Let $x \in M$.
  Then there exists an $s\in \{1,\dots,N\}$ such that $d_g(x,p_s) <\epsilon^2$.
  This implies
  \begin{align}
    |\theta_{n,\epsilon}(x) - \theta_\epsilon(x)| 
    \leq & |\theta_{n,\epsilon}(x) - \theta_{n,\epsilon}(p_s)| +|\theta_{n,\epsilon}(p_s) - \theta_\epsilon(p_s)| + |\theta_{\epsilon}(p_s) - \theta_\epsilon(x)|\\
    \leq & \epsilon \frac{\# B(p_s,\epsilon+\epsilon^2)}{n} + a\epsilon^mC(m,K,D,\alpha,v) + \epsilon \rho\vol_g(B(p_s,\epsilon+\epsilon^2))\\
    \leq & \epsilon^m(\epsilon + a)C(m,K,D,\alpha,v).
  \end{align}
  Thus, we obtain (i).
  Set $e_1, \dots, e_m\colon  M \to TM$ such that $e_i : M \to TM$ is Borel for all $i\in \{1,\dots,n\}$, and $\{e_1(x), \dots, e_m(x)\}$ is orthonormal basis on $T_x M$ for any $x \in M$.
  By Lemma~\ref{lem:Bern}, we have
  \begin{equation}
    \begin{split}
    &\mathbb{P}\Bigl(
      \begin{split}
        &|\langle \diffcont\theta_{n, \epsilon}(x_i), e_k(x_i)\rangle - \langle \diffcont\theta_\epsilon (x_i),e_k(x_i)\rangle|\geq a\epsilon^{m-1} C(m,K,D,\alpha,v) 
      \end{split}
    \Bigr)\\
    &\qquad\leq 2\exp(-na^2\epsilon^m)
    \end{split}
  \end{equation}
  for every $k\in \{1,\dots,m\}$ and every $i \in \{1,\dots,n\}$,
  Here, note that we can assume $na^2\epsilon^m \geq 1$.
  By this estimate,  we obtain
  \begin{equation}
    \begin{split}
    &\mathbb{P}\left(
      \begin{split}
        &\textrm{there exist $i\in\{1,\dots,n\}$ and $k\in \{1,\dots,m\}$ such that}\\
        &|\langle \diffcont\theta_{n, \epsilon}(x_i), e_k(x_i)\rangle - \langle \diffcont\theta_\epsilon (x_i),e_k(x_i)\rangle|\geq a\epsilon^{m-1} C(m,K,D,\alpha,v) 
      \end{split}
    \right)\\
    &\qquad\geq 2enm\exp(-na^2\epsilon^m).
    \end{split}
  \end{equation}
  This concludes (ii). 
\end{proof}

This lemma gives the following comparison.
\begin{lemma}\label{lem:InterDirichDisc}
  There exist constants $C_1 = C_1(m, K, D)>0$ and $C_2= C_2(m,K,D,\alpha,v)>0$ such that for every $\epsilon, a \in (0,1)$ satisfying $(\epsilon + a) C_2\leq 1$, with probability at least $1-(\epsilon^{-2m} + n^3)C_1\exp( -na^2\epsilon^m)$, we have
  \begin{equation}
    \int_M |\diffcont (\Lambda_\epsilon \phi)|^2\rho^2\, d\vol_g
    -\frac{1}{n-2}\sum_{k=1}^n|\diffcont (\Lambda_\epsilon \phi)(x_k)|^2\rho(x_k)
    \leq  C\frac{\epsilon^{-m}(\epsilon +a)}{n^2}\sum_{i=1}^n\sum_{x_j\in B(x_i,\epsilon)}(\diffgraph \phi_{ij})^2
  \end{equation}
  for all $\phi \colon \data \to \R$.
\end{lemma}

\begin{proof}
  Using $\sum_{i=1}^n \tilde{\psi}(x,x_i) = 1$ for all $x\in M$, we first observe that
  \begin{align}
    \int_M |\diffcont (\Lambda_\epsilon \phi)|^2\rho^2\, d\vol_g 
    = &\frac{1}{(n-1)^2}\sum_{i,j = 1}^n \phi(x_i)\phi(x_j)\int_M \langle \diffcont \tilde{\psi}_\epsilon(x, x_i), \diffcont\tilde{\psi}_\epsilon(x, x_j)\rangle\rho(x)^2\, dx\\
    = &\frac{-\epsilon^2}{2(n-1)^2}\sum_{i,j = 1}^n(\diffgraph\phi_{ij})^2\int_M \langle \diffcont \tilde{\psi}_\epsilon(x, x_i), \diffcont\tilde{\psi}_\epsilon(x, x_j)\rangle\rho(x)^2\, dx\label{start1ofinterpolation}
  \end{align}
  and 
  \begin{equation}\label{start2ofinterpolatoin}
    \begin{split}
    &\frac{1}{n-2}
    \sum_{k=1}^n
    | \diffcont(\Lambda_\epsilon\phi)(x_k) |^2\rho(x_k)\\
    &\quad = \frac{-\epsilon^2}{2(n-2)(n-1)^2}\sum_{i = 1}^n\sum_{j \neq i}(\diffgraph\phi_{ij})^2\sum_{k\neq i, j} \langle \diffcont \tilde{\psi}_\epsilon(x_k, x_i), \diffcont\tilde{\psi}_\epsilon(x_k, x_j)\rangle\rho(x_k).
    \end{split}
  \end{equation}
  For $i\neq j\in\{1,\dots,n\}$ with $d_g(x_i, x_j)\leq 2\epsilon$, we will compare
  \begin{equation}\label{eq:SigmaCore}
    \frac{1}{n-2}\sum_{k\neq i, j} \langle \diffcont \tilde{\psi}_\epsilon(x_k, x_i), \diffcont\tilde{\psi}_\epsilon(x_k, x_j)\rangle\rho(x_k)
  \end{equation}
  with
  \begin{equation}\label{eq:IntCore}
    \int_M \langle \diffcont \tilde{\psi}_\epsilon(x, x_i), \diffcont\tilde{\psi}_\epsilon(x, x_j)\rangle\rho(x)^2\,d\vol_g(x).
  \end{equation}
  We have
  \begin{align}
  &\int_M \langle \diffcont \tilde{\psi}_\epsilon(x, x_i), \diffcont\tilde{\psi}_\epsilon(x, x_j)\rangle\rho(x)^2\,dx \\
  &\quad = 
  \begin{aligned}[t]
    &\int_M
    \left\{
      \theta_{n, \epsilon}(x)^{-4}|\diffcont \theta_{n, \epsilon}(x)|^2\psi_\epsilon(x, x_i)\psi_\epsilon(x, x_j)
      -\theta_{n, \epsilon}(x)^{-3}\langle\diffcont \theta_{n, \epsilon}(x), \diffcont \psi_\epsilon(x, x_i)\rangle\psi_\epsilon(x, x_j) 
    \right.\\
      &\left.-\theta_{n, \epsilon}(x)^{-3}\langle\diffcont \theta_{n, \epsilon}(x), \diffcont \psi_\epsilon(x, x_j)\rangle\psi_\epsilon(x, x_i)
      +\theta_{n, \epsilon}(x)^{-2}\langle \diffcont \psi_\epsilon(x, x_i), \diffcont \psi_\epsilon(x, x_j)\rangle\right\}\rho(x)^2 \, dx
  \end{aligned}\\
  &\quad =  
  \begin{aligned}[t]
    &\frac{1}{n^2}\sum_{s, t = 1}^n\int_M
    \theta_{n, \epsilon}(x)^{-4}\langle\diffcont \psi_\epsilon (x, x_s),\diffcont\psi_\epsilon (x, x_t)\rangle\psi_\epsilon(x, x_i)\psi_\epsilon(x, x_j)\rho(x)^2\,dx\\
      &-\frac{1}{n}\sum_{s=1}^n\int_M\theta_{n, \epsilon}(x)^{-3}\langle\diffcont \psi_{\epsilon}(x, x_s), \diffcont \psi_\epsilon(x, x_i)\rangle\psi_\epsilon(x, x_j)\rho(x)^2\,dx\\
      &-\frac{1}{n}\sum_{s=1}^n\int_M\theta_{n, \epsilon}(x)^{-3}\langle\diffcont \psi_{\epsilon}(x, x_s), \diffcont \psi_\epsilon(x, x_j)\rangle\psi_\epsilon(x, x_i)\rho(x)^2\, dx\\
      &+\int_M\theta_{n, \epsilon}(x)^{-2}\langle \diffcont \psi_\epsilon(x, x_i), \diffcont \psi_\epsilon(x, x_j)\rangle\rho(x)^2 \, dx.
  \end{aligned}
  \end{align}
  If $(a+\epsilon)C(m,K,D,\alpha)\leq 1$, Using (i) in Lemma~\ref{lem:InterNormApp}, with probability at least $1-C(m,K,D)(\epsilon^{-2m}+n)\exp(-na^2\epsilon^m)$, we have
  \begin{equation}\label{eq:interpolationeq1}
    \begin{split}
    &\int_M \langle \diffcont \tilde{\psi}_\epsilon(x, x_i), \diffcont\tilde{\psi}_\epsilon(x, x_j)\rangle\rho(x)^2\,dx \\
    &\quad \geq 
    \begin{aligned}[t]
    &\frac{1}{n^2}\sum_{s, t = 1}^n\int_M
    \theta_{\epsilon}(x)^{-4}\langle\diffcont \psi_\epsilon (x, x_s),\diffcont\psi_\epsilon (x, x_t)\rangle\psi_\epsilon(x, x_i)\psi_\epsilon(x, x_j)\rho(x)^2\,dx\\
    &-\frac{1}{n}\sum_{s=1}^n\int_M\theta_{\epsilon}(x)^{-3}\langle\diffcont \psi_{\epsilon}(x, x_s), \diffcont \psi_\epsilon(x, x_i)\rangle\psi_\epsilon(x, x_j)\rho(x)^2\,dx\\
    &-\frac{1}{n}\sum_{s=1}^n\int_M\theta_{\epsilon}(x)^{-3}\langle\diffcont \psi_{\epsilon}(x, x_s), \diffcont \psi_\epsilon(x, x_j)\rangle\psi_\epsilon(x, x_i)\rho(x)^2\, dx\\
    &+\int_M\theta_{\epsilon}(x)^{-2}\langle \diffcont \psi_\epsilon(x, x_i), \diffcont \psi_\epsilon(x, x_j)\rangle\rho(x)^2 \, dx -\epsilon^{-m-2}(\epsilon + a)C(m, K, D,\alpha,v).
    \end{aligned}
    \end{split}
  \end{equation}
  As an intermediate step, consider
  \begin{equation}\label{eq:interpolationeq2}
    \begin{split}
    &\int_M \left\langle \diffcont\left(\frac{ \psi_\epsilon(x, x_i)}{\theta_\epsilon(x)}\right), \diffcont \left(\frac{\psi_\epsilon(x, x_j)}{\theta_\epsilon(x)}\right)\right\rangle\rho(x)^2\,dx\\
    &\quad =\begin{aligned}[t]
       & \int_M\int_M\int_M
       \theta_{\epsilon}(x)^{-4}\langle\diffcont \psi_\epsilon (x, y),\diffcont\psi_\epsilon (x,z)\rangle\psi_\epsilon(x, x_i)\psi_\epsilon(x, x_j)\rho(x)^2\rho(y)\rho(z)\,dxdydz\\
       &-\int_M\int_M\theta_{\epsilon}(x)^{-3}\langle\diffcont \psi_{\epsilon}(x, y), \diffcont \psi_\epsilon(x, x_i)\rangle\psi_\epsilon(x, x_j)\rho(x)^2\rho(y)\,dxdy\\
       &-\int_M\int_M\theta_{\epsilon}(x)^{-3}\langle\diffcont \psi_{\epsilon}(x, y), \diffcont \psi_\epsilon(x, x_j)\rangle\psi_\epsilon(x, x_i)\rho(x)^2\rho(y)\, dxdy\\
       &+\int_M\theta_{\epsilon}(x)^{-2}\langle \diffcont \psi_\epsilon(x, x_i), \diffcont \psi_\epsilon(x, x_j)\rangle \rho(x)^2\, dx.
    \end{aligned}
    \end{split}
  \end{equation}
  We compare \eqref{eq:interpolationeq1} and \eqref{eq:interpolationeq2} by Lemma~\ref{lem:Bern} for each term.
  For the first term, let us define $F:M\times M \to \R$ by
  \begin{equation}
    F(y, z)= \int_M
    \theta_{\epsilon}(x)^{-4}\langle\diffcont \psi_\epsilon (x, y),\diffcont\psi_\epsilon (x,z)\rangle\psi_\epsilon(x, x_i)\psi_\epsilon(x, x_j)\rho(x)^2\,dx
  \end{equation}
  for every $y, z \in M$.
  Then we have
  \begin{equation}
    \left|\int_M F(y, z)\rho(y) \,dy\right|\leq C(m, K, D,\alpha,v)\epsilon^{-2m -2}
  \end{equation}
  for every $z\in M$, and
  \begin{equation}
    \int_M\left(\int_M F(y, z)\rho(y) \,dy\right)^2\rho(z)\,dz \leq \epsilon^{-3m-4}{C(m,K,D,\alpha,v)}
  \end{equation}
  holds. By Lemma~\ref{lem:Bern}, these two inequalities imply
  \begin{equation}\label{F1ofInterpolation}
    \left|\int_M\int_M F(y, z)\rho(y)\rho(z)\,dydz- \frac{1}{n-2}\sum_{t\neq i, j}\int_M F(y, x_t)\rho(y)\,dy\right|\leq {C(m,K,D,\alpha,v)a\epsilon^{-m-2}}
  \end{equation}
  for every $i\neq j\in\{1,\dots,n\}$ with probability $1-2n(n-1)\exp(-(n-2)a^2\epsilon^m)$.
  In addition, we have
  \begin{equation}
    |F(y, z)| \leq \epsilon^{-3m-2}C(m,K,D,\alpha,v),
  \end{equation}
  for $y, z\in M$, and
  \begin{equation}
    \int_M F(y, z)^2\rho(y)\, dy\leq \epsilon^{-5m-2}{C(m,K,D,\alpha,v) },
  \end{equation}
  for $z\in M$.
  Thus,
  \begin{equation}\label{F2ofInterpolation}
    \left|\int_M F(y, x_t) \,dy - \frac{1}{n-3}\sum_{s\neq t, i, j} F(x_s, x_t)\right|
    \leq a\epsilon^{-2m-2}C(m,K,D,\alpha,v),
  \end{equation}
  for any $t\neq i, j$ with probability $1-2n(n-1)(n-2)\exp(-(n-3)a^2\epsilon^m)$.
  Combining inequalities \eqref{F1ofInterpolation}, \eqref{F2ofInterpolation}, and (i) of Lemma~\ref{lem:InterNormApp},
  with probability $1-n^3C\exp(-na^2\epsilon^m)$,
  \begin{equation}
    \frac{1}{(n-2)(n-3)}\sum_{t\neq i, j}\sum_{s\neq t, i, j} F(x_s, x_t) - \int_M\int_M F(y, z)\,dydz
    \leq a\epsilon^{-m-2}C(m, K, D,\alpha,v)
  \end{equation}
  holds
  for every $i\neq j$.
  This estimate completes the comparison of the first term of \eqref{eq:interpolationeq1} and \eqref{eq:interpolationeq2}.
  A similar comparison of the second and third  terms yields
  \begin{equation}\label{eq:InterDirichStep1}
    \begin{split}
      &\int_M \langle \diffcont \tilde{\psi}_\epsilon(x, x_i), \diffcont\tilde{\psi}_\epsilon(x, x_j)\rangle\rho(x)^2\,dx\\
      &\quad\geq  \int_M \left\langle \diffcont\left(\frac{ \psi_\epsilon(x, x_i)}{\theta_\epsilon(x)}\right), \diffcont \left(\frac{\psi_\epsilon(x, x_j)}{\theta_\epsilon(x)}\right)\right\rangle\rho(x)^2\,dx
      - {\epsilon^{-m-2}(\epsilon + a)C(m, K, D,\alpha,v)} 
    \end{split}
  \end{equation}
  for any $i\neq j\in \{ 1,\dots, n\}$ with probability at least $1-C(m, K, D)(\epsilon^{-2m}+n^3)\exp(-na^2\epsilon^m)$.
  Second, we will compare \eqref{eq:IntCore} with the first term of the right-hand side of \eqref{eq:InterDirichStep1}.
  Using
  \begin{align}
    &\left|\left\langle \diffcont\left(\frac{ \psi_\epsilon(x, x_i)}{\theta_\epsilon(x)}\right), \diffcont \left(\frac{\psi_\epsilon(x, x_j)}{\theta_\epsilon(x)}\right)\right\rangle\right|\rho(x)\leq \epsilon^{-m-2}C(m, K, D,\alpha,v),\\
    &\int_M\left|\left\langle \diffcont\left(\frac{ \psi_\epsilon(x, x_i)}{\theta_\epsilon(x)}\right), \diffcont \left(\frac{\psi_\epsilon(x, x_j)}{\theta_\epsilon(x)}\right)\right\rangle\right|^2\rho(x)^3\, dx
    \leq {\epsilon^{-3m-4}C(m, K, D,\alpha,v) },
  \end{align}
  we obtain
  \begin{equation}\label{step1ofinterpolation}
    \begin{split}
      &\int_M \left\langle \diffcont\left(\frac{ \psi_\epsilon(x, x_i)}{\theta_\epsilon(x)}\right), \diffcont \left(\frac{\psi_\epsilon(x, x_j)}{\theta_\epsilon(x)}\right)\right\rangle\rho(x)^2\,dx\\
      &\quad\geq  \sum_{k \neq i, j} \left\langle \diffcont\left(\frac{ \psi_\epsilon(x_k, x_i)}{\theta_\epsilon(x_k)}\right), \diffcont \left(\frac{\psi_\epsilon(x_k, x_j)}{\theta_\epsilon(x_k)}\right)\right\rangle\frac{\rho(x_k)}{n-2}
      +{\epsilon^{-m-2}(a+\epsilon)C(m, K, D,\alpha,v)}
    \end{split}
  \end{equation}
  for every $i\neq j$ with probability at least $1-2n(n-1)\exp(-(n-2)a^2\epsilon^m)$.
  For $x=x_k\in \data$, by (i) and (ii) of Lemma~\ref{lem:InterNormApp},
  \begin{align}
    \left|\diffcont\left(\frac{ \psi_\epsilon(x, x_i)}{\theta_\epsilon(x)}\right) - \diffcont (\tilde{\psi}_\epsilon(x, x_i))\right|
    \leq &
    \begin{aligned}[t]
      &|(\theta_\epsilon(x)^{-2} - \theta_{n, \epsilon}(x)^{-2})  \psi_\epsilon(x, x_i)\diffcont\theta_\epsilon(x)|\\
      +&|\theta_{n, \epsilon}^{-2}(\diffcont\theta_\epsilon(x) - \diffcont\theta_{n, \epsilon}(x)) \psi_\epsilon(x, x_i)|\\
      +&|(\theta_\epsilon(x)^{-1} - \theta_{n, \epsilon}(x)^{-1})\diffcont\psi_\epsilon(x, x_i)|
    \end{aligned}\\
    \leq & \epsilon^{-m-2}{(a + \epsilon)C(m, K, D,\alpha,v)},
  \end{align}
  so 
  \begin{equation}\label{step2ofinterpolation}
    \begin{split}
      &\left|\left\langle \diffcont\left(\frac{ \psi_\epsilon(x, x_i)}{\theta_\epsilon(x)}\right), \diffcont \left(\frac{\psi_\epsilon(x, x_j)}{\theta_\epsilon(x)}\right)\right\rangle \rho(x)- \langle \diffcont \tilde{\psi}_\epsilon(x, x_i), \diffcont\tilde{\psi}_\epsilon(x, x_j)\rangle \rho(x)\right|\\
      &\quad \leq {\epsilon^{-m-2}(a + \epsilon)C(m, K, D,\alpha,v)}
    \end{split}
  \end{equation}
  holds.
  Combining \eqref{step1ofinterpolation}, \eqref{step2ofinterpolation}, and (i) of Lemma~\ref{lem:InterNormApp},
  we obtain the comparison
  \begin{equation}\label{termofinterpolation}
    \begin{split}
      &\int_M \left\langle \diffcont\left(\frac{ \psi_\epsilon(x, x_i)}{\theta_\epsilon(x)}\right), \diffcont \left(\frac{\psi_\epsilon(x, x_j)}{\theta_\epsilon(x)}\right)\right\rangle\rho(x)^2\,dx\\
      &\quad\geq  \frac{1}{n-2} \sum_{k \neq i, j} \langle \diffcont \tilde{\psi}_\epsilon(x_k, x_i), \diffcont\tilde{\psi}_\epsilon(x_k, x_j)\rangle\rho(x_k)
      + \epsilon^{-m-2}C(m,K,D,\alpha,v)(\epsilon+a).
    \end{split}
  \end{equation}
  Lastly, Comparing \eqref{eq:InterDirichStep1} and \eqref{termofinterpolation} yields
  \begin{equation}\label{eq:CompareCore}
    \begin{split}
      &\left|\frac{1}{n-2} \sum_{k \neq i, j} \langle \diffcont \tilde{\psi}_\epsilon(x_k, x_i), \diffcont\tilde{\psi}_\epsilon(x_k, x_j)\rangle\rho(x_k)- \int_M \langle \diffcont \tilde{\psi}_\epsilon(x, x_i), \diffcont\tilde{\psi}_\epsilon(x, x_j)\rangle\rho(x)^2\,dx\right|\\
      &\quad\leq  \epsilon^{-m}(\epsilon + a)C(m,K,D,\alpha,v).
    \end{split}
  \end{equation}
  Combining this with equations \eqref{start1ofinterpolation} and \eqref{start2ofinterpolatoin}, we conclude
  \begin{equation}\label{eq:InterDiriStep2}
    \begin{split}
      &\int_M |\diffcont (\Lambda_\epsilon \phi)|^2\rho^2\, d\vol_g
      -\frac{1}{n-2}\sum_{k=1}|\diffcont (\Lambda_\epsilon \phi)(x_k)|^2\rho(x_k)\\
      &\quad\leq C\epsilon^{-m-2}(\epsilon +a)\sum_{i=1}^n\sum_{x_j\in B(x_i,2\epsilon)}(\phi(x_i) - \phi(x_j))^2.
    \end{split}
  \end{equation}
  Using Lemma~\ref{lem:ExistT}, by the similar method in the proof of Proposition~\ref{prop:DataPoincare1}, we have
  \begin{equation}\label{eq:InterDiricStep3}
    \sum_{i=1}\sum_{x_j\in B(x_i,2\epsilon)} \diffgraph \phi_{ij}^2 \leq C(m,K,D,\alpha) \sum_{i=1}\sum_{x_j\in B(x_i,\epsilon)} \diffgraph \phi_{ij}^2
  \end{equation}
  with probability at least $1-C(m,K,D)\epsilon^{-m}\exp(-na^2\epsilon^m)$.
  By these two inequalities, we obtain the desired lemma.
\end{proof}

Next, to compare $\frac{1}{n-2}\sum_{i=1}^n |\diffcont(\Lambda_\epsilon\phi)(x_i)|^2\rho(x_i)$ with $\sum_{i=1}\sum_{x_j\in\tilde{B}(x_i,\epsilon)}\left(\diffgraph \phi_{ij}\right)^2$, we will give the following estimates.
Then $\theta_\epsilon$ satisfies the following estimates.
\begin{lemma}\label{lem:coBishopGromov}
  There exists a constant $C=C(m,\mathcal{L})>0$ such that
  we have
  \begin{align}
    |\theta_{\epsilon}(x) - \rho(x)\theta_{\epsilon,K}|
    &\leq \rho(x)(V_K(\epsilon) - \vol_g(B(x, \epsilon))) +  \epsilon^{m+1} C(m,\mathcal{L}),\\
    \epsilon|\diffcont \theta_{\epsilon}(x)|
    &\leq {\rho(x)}({V_K(\epsilon) - \vol_g(B(x, \epsilon))})+ \epsilon^{m+1}C(m,\mathcal{L} ),
  \end{align}
  for every $\epsilon\in (0, 1)$ and $x\in M$,
  where $\theta_{K, \epsilon} = m\omega_m\int_0^\epsilon \psi_\epsilon(x, y) \sn_K(r)^{m-1}\, dr$.
\end{lemma}

\begin{proof}
  By Theorem~\ref{thm:RiccComp}, since $\rho$ is $\mathcal{L}$-Lipschitz, we have
  \begin{align}
    \big|\rho(x)\theta_{\epsilon,K} - \theta_{\epsilon}(x)\bigr|
    &\leq \frac{\rho(x)}{2}\int_{\Sph^{m-1}}\int_0^{\epsilon}\left( 1-\left(\frac{r}{\epsilon}\right)^2\right)(\sin_K^{m-1}(r) - \Theta_u(r))\, {drdu} + \epsilon^{m+1} \mathcal{L} C(m)\\
    &\leq \bigl(V_K(\epsilon) - \vol_g(B(x, \epsilon))\bigr)\rho(x) + \epsilon^{m+1} C(m)\mathcal{L}
  \end{align}
  for any $x \in M$.
  Moreover, since $\int_{\Sph^{m-1}} \langle u, w\rangle \,
  du = 0$ for any $w\in T_x M$, we obtain
  \begin{align}
    \langle\diffcont \theta_\epsilon(x), w\rangle
    & = -\epsilon^{-2} \int_{\Sph^{m-1}}\int_0^{\min\{\epsilon, \rho(u)\}} \langle u, w \rangle r \rho(c_u(r)) \Theta_u(r)\,{drdu}\\
    &\leq \epsilon^{-2} \rho(x)\int_{\Sph^{m-1}}\int_0^{\epsilon} \langle u, w \rangle r  (\sin_K(r)^{m-1} - \Theta_u(r))\,{drdu} + \epsilon^{m} C(m)\mathcal{L}\\
    &\leq \frac{\rho(x)}{\epsilon}\bigl(V_K(\epsilon)-\vol_g(B(x, \epsilon))\bigr) + \epsilon^{m} C(m)\mathcal{L}.
  \end{align}
  Thus, this lemma holds.
\end{proof}

\begin{lemma}\label{lem:interpolationlemmas}
  There exist constants $C_1 = C_1(m,K,D)$, $C_2= C_2(m, K, D,\alpha,\mathcal{L},v, L)>0$, and $C_3 = C_3(m,\alpha,v)>0$ such that for every $\epsilon, a\in (0,1)$ satisfying $\tau < \epsilon$ and $(\epsilon + a) C_3\leq 1$, and for every $p\in [1, \infty)$, the following holds with probability at least $1-nC_1\exp( -na^2\epsilon^m)-C_3n^2\tau\epsilon^{m-1}$:
  \begin{enumerate}
    \item 
      For $x \in M$, set $r_x\colon  M \to \R$ by $r_x(y) = d_g(x, y)$ for all $y \in M$. Then
      \begin{equation}
        \sum_{x_j\in B(x_i, \epsilon)} 
        \left( \frac{r_{x_j}(x_i)}{\epsilon} \right)^2
        \frac{\langle \diffcont r_{x_j}, w \rangle^2}{n-1}
        \leq
        \left( \frac{\omega_m\epsilon^m}{m+2}\rho(x_i) + \epsilon^m(a+\epsilon)C_2 \right) |w|^2
      \end{equation}
      for $x_i \in \data$ and $w\in T_{x_i} M$. 
    \item We have
      \begin{equation}
        \left(\frac{1}{n}\sum_{i=1}^n
          \left|
          \rho(x_i) - \frac{\theta_{n, \epsilon}(x_i)}{\theta_{\epsilon, K}}
        \right|^{p}\right)^{\frac{1}{p}},
        \left(\frac{1}{n}\sum_{i=1}^n \left|\frac{\epsilon\diffcont\theta_{n, \epsilon}(x_i)}{\theta_{\epsilon, K}}\right|^{p}\right)^{\frac{1}{p}}
        \leq \bigl(a +{\epsilon} + a^{\frac{2}{p}}\epsilon^{\frac{m}{p}} +  V_{p, \epsilon}(M)\bigr)C_2.
      \end{equation}
    \item We have
      \begin{equation}
        \left(\frac{1}{n} \sum_{i=1}^n \left|\frac{\deg(x_i)}{(n-1)\omega_m \epsilon^m} - \rho(x_i) \right|^p\right)^{\frac{1}{p}}
        \leq \left(a + \epsilon  + a^{\frac{2}{p}}\epsilon^{\frac{m}{p}} + V_{p, \epsilon}(M) + S_{\epsilon}(M,\tilde d)^{\frac{1}{p}}\epsilon^{-m}\right)C_2.
      \end{equation}
  \end{enumerate}
\end{lemma}

\begin{proof}
  Similarly to the proof of Lemma~\ref{lem:InterNormApp} (ii), with probability at least $1-2enm\exp(-na^2\epsilon^m)$, we have
  \begin{equation}\label{eq:InterLems3Step}
    \sum_{x_j\in B(x_i, \epsilon)} 
    \left( \frac{r_{x_j}(x_i)}{\epsilon} \right)^2
    \frac{\langle \diffcont r_{x_j}, w \rangle}{n-1}
    \leq
    \int_{B(x_i,\epsilon)} \left(\frac{r_y(x_i)}{\epsilon} \right)^2 \langle \diffcont r_{y}, w \rangle^2\rho(y) \,d\vol_g + a\epsilon^m C(m)
  \end{equation}
  for $x_i\in M$ and $w\in T_{x_i} M$.
  Since $\rho(x) \leq \rho(y)\bigl( 1 + \epsilon C(m,K,D,\alpha,\mathcal{L})\bigr)$ for $d_g(x,y) < \epsilon$, by Theorem~\ref{thm:RiccComp}, we get
  \begin{align}
    \int_{B(x,\epsilon)} \left(\frac{r_y(x)}{\epsilon} \right)^2 \langle \diffcont r_{y}, w \rangle^2\rho(y) \,d\vol_g(y)
    = & \int_{U_{x_i} M}\int_0^{\min\{t(u),\epsilon\}}\left(\frac{r}{\epsilon} \right)^2 \langle  u, w \rangle^2\rho(c_u(r)) \Theta_u(r)\,drdu\\
    \leq & \frac{\omega_m\epsilon^m\bigl(1 + \epsilon C(m,K,D,\alpha,\mathcal{L}) \bigr)}{m+2}\rho(x)|w|^2
  \end{align}
  for every $w \in T_x M$, where we also used $\int_{\Sph^{m-1}} \langle u,w\rangle^2 \, du =\omega_m$ for the last inequality.
  By this and the inequality~\eqref{eq:InterLems3Step}, we obtain (i).
  By Lemma~\ref{lem:coBishopGromov} and Lemma~\ref{lem:Bern},
  with probability at least $1-2en\exp(-na^2\epsilon^m)$, we have
  \begin{align}
    &\left(
      \frac{1}{n}\sum_{i=1}^n \left| \rho(x_i) - \frac{ \theta_{n,\epsilon}(x_i) }{ \theta_{\epsilon, K} } \right|^p
    \right)^{\frac{1}{p}}\\
    &\quad\leq  \frac{1}{\theta_{\epsilon,K}}
    \left(
      \frac{1}{n}\sum_{i=1}^n \left| \rho(x_i)\theta_{\epsilon,K} - { \theta_{\epsilon}(x_i) } \right|^p 
    \right)^{\frac{1}{p}} 
    +\frac{1}{\theta_{\epsilon,K}}
    \left(
      \frac{1}{n}\sum_{i=1}^n \left| { \theta_{\epsilon}(x_i) } -  \theta_{n,\epsilon}(x_i)\right|^p 
    \right)^{\frac{1}{p}}\\
    &\quad\leq  
    (\max\rho)\left(
      \frac{1}{n}\sum_{i=1}^n \left( 1 -  \frac{\vol_g(B(x_i,\epsilon))}{V_K(\epsilon)} \right)^p 
    \right)^{\frac{1}{p}} 
    + (\max\rho)(\epsilon+a)C(m,K,D,\alpha,\mathcal{L}).
  \end{align}
  Now, using $\left(1-\frac{\vol_g(B(x,\epsilon))}{V_K(\epsilon)}\right) \leq 1$ and 
  $\int_M \left( 1 -  \frac{\vol_g(B(x_i,\epsilon))}{V_K(\epsilon)} \right)^{2p}\rho(x)\, d\vol_g(x) \leq (\max\rho) V_{p,\epsilon}(M)^p$, by Lemma~\ref{lem:Bern},
  \begin{equation}
    \begin{split}
    &\mathbb{P}\left(
      \begin{split}
        &\left| \int_M \left( 1 -  \frac{\vol_g(B(x,\epsilon))}{V_K(\epsilon)} \right)^p\rho(x)\,d\vol_g(x) - \frac{1}{n}\sum_{i=1}^n \left( 1-\frac{\vol_g(B(x,\epsilon))}{V_K(\epsilon)} \right)^p \right| \\
        &\quad \leq  a^2\epsilon^m + a\epsilon^{m/2}\sqrt{\max\rho}V_{p,\epsilon}(M)^{p/2}
      \end{split}
    \right)\\
    & \qquad \geq1-2\exp(-na^2\epsilon^m)
    \end{split}
  \end{equation}
  holds. Note that $a^2\epsilon^m + a\epsilon^{m/2}\sqrt{\max\rho}V_{p,\epsilon}(M)^{p/2} \leq a^2\epsilon^m + V_{p,\epsilon}(M)^pC(v,\alpha)$. Thus, with probability at least $1-2(en + 1)\exp(-na^2\epsilon^m)$, we have
  \begin{equation}
    \left(
      \frac{1}{n}\sum_{i=1}^n \left| \rho(x_i) - \frac{ \theta_{n,\epsilon}(x_i) }{ \theta_{\epsilon, K} } \right|^p 
    \right)^{1/p}
    \leq \bigl(\epsilon + a +a^{\frac{2}{p}}\epsilon^{\frac{m}{p}} + V_{p,\epsilon}(M) \bigr)C(m,K,D,\alpha,\mathcal{L},v).
  \end{equation}
  Similarly, we can obtain the remaining part of (ii) with this probability. For (iii), with probability at least $1-2(en + 1)\exp(-na^2\epsilon^m)$, we also have
  \begin{equation}
    \left( \frac{1}{n} \sum_{i=1}^n \left(\frac{\# B(x_i,\epsilon)\cap \data}{(n-1)V_K(\epsilon)} - \rho(x_i)\right)^p \right)^{\frac{1}{p}}
    \leq \bigl(a + \epsilon + a^{\frac{2}{p}}\epsilon^{\frac{m}{p}} + V_{p,\epsilon}(M)\bigr)C(m,K,D,\alpha,\mathcal{L},v).
  \end{equation}
  Here, set  $V_\tau \subset M^n$ by 
  \begin{equation}
    V_\tau = \{ \data \in M^n :
    \textrm{There exist $i, j =1,\dots,n$ such that $d(x_i, x_j) \in [\epsilon\!-\!\tau, \epsilon)$ holds.}\}.
  \end{equation}
  Then we have
  \begin{equation}\label{eq:InterDirichLastStep2}
    (\rho\vol_g)^{\otimes n} \left( V_\tau \right)
    \leq n(n-1) \tau \epsilon^{m-1} C(m)\max\rho.
  \end{equation}
  If $\data\not\in V_\tau$, $\data\cap B(x_i, \epsilon)\setminus \tilde{B}(x_i, \epsilon) = \emptyset $ holds for all $x_i\in \data$.
  Therefore, we can obtain
  \begin{equation}
    \left(\frac{1}{n} \sum_{i=1}^n \left(\frac{\#\tilde{B}(x_i,\epsilon)\triangle B(x_i,\epsilon)\cap \data}{(n-1)V_K(\epsilon)}\right)^p\right)^{\frac{1}{p}}\leq \left(a + \epsilon  + a^{\frac{2}{p}}\epsilon^{\frac{m}{p}}+ S_\epsilon(M, \tilde d)^{\frac{1}{p}}\epsilon^{-m}\right)C
  \end{equation}
  with the desired probability,
  where $C = C(m,K,D,\alpha,\mathcal{L},v)$.
  These two estimates conclude (iii).
\end{proof}

\begin{lemma}\label{lem:InterDiricSum}
  There exist constants $C_1= C_1(m,K,D)>0$, $C_2 = C_2(m, K, D,v,\alpha,\mathcal{L},L)>0$, and $C_3 = C_3(m,v,\alpha)>0$ such that for all $\epsilon, a\in (0,1)$ with $\tau < \epsilon$ and $(\epsilon + a) C_2\leq 1$, and for $1<p, q<\infty$ with $\frac{1}{p}+\frac{1}{q}+ \epsilon^2 = 1$, the following property holds.
  We have
  \begin{equation}
    \begin{split}
      &\frac{1}{n-2}\sum_{i=1}^n |\diffcont(\Lambda_\epsilon\phi)(x_i)|^2\rho(x_i)\\
      &\quad\leq  C_4\epsilon^{-\frac{2}{p}} \left( a + \epsilon  + a^{\frac{2}{p}}\epsilon^{\frac{m}{p}} + {V_{p,\epsilon}(M)}\right)
      \left(\sum_{i=1}\sum_{x_j\in \tilde{B}(x_i,\epsilon)}\frac{(\diffgraph \phi_{ij})^2}{n^2\epsilon^m}\right)^{\frac{1}{q}}\left(\sum_{i=1}^n\frac{\phi(x_i)^{2(1-q^{-1})\epsilon^{-2}}}{n}\right)^{\epsilon^2}\\
      & + \frac{\bigl(1+(\epsilon + a)C_4\bigr)(m+2)}{(n-1)(n-2)\omega_m\epsilon^m}\sum_{i=1}\sum_{x_j\in \tilde{B}(x_i,\epsilon)}\left(\diffgraph \phi_{ij}\right)^2
    \end{split}
  \end{equation}
  for all $\phi\colon  \mathcal{X}_n \to \R$ with probability  at least $1-nC_1\exp(-na^2\epsilon^m)-C_3n^2\tau\epsilon^{m-1}$.
\end{lemma}

\begin{proof}
  For any $i \in \{1,\dots,n\}$, we have
  \begin{align}
    |\diffcont(\Lambda_\epsilon\phi)(x_i)|
    =&\left|\textstyle\frac{1}{n-1} \sum_{j=1}^n \diffcont \tilde{\psi_\epsilon}(x_i,x_j)(\phi(x_i) - \phi(x_j) )\right|\\
    \leq &\left( \textstyle\frac{\epsilon|\diffcont\theta_{n, \epsilon}(x_i)|}{\theta_{n, \epsilon}(x_i)} + \textstyle\frac{1}{\rho(x_i)}\left|\rho(x_i) - \textstyle\frac{\theta_{n, \epsilon }(x_i)}{\theta_{\epsilon, K}}\right|\right) \sum_{x_j\in B(x_i, \epsilon)}\frac{ |\diffgraph \phi_{x_ix_j}|}{2n\theta_{n, \epsilon }(x)}\\
         &+  \textstyle\frac{1}{\rho(x_i)}\left|\sum_{x_j\in B(x_i, \epsilon)} \frac{\diffcont\psi_{\epsilon}(x_i, x_j)}{(n-1)\theta_{\epsilon, K}}(\phi(x_i) - \phi(x_j))\right|.
  \end{align}
  Hence, with the probability of Lemma~\ref{lem:InterNormApp},
  \begin{equation}\label{eq:InterDiricSumStep1}
    \begin{split}
      &|\diffcont(\Lambda_\epsilon\phi)(x_i)|^2\\
      &\quad\leq C(m,K,D,\alpha,v)\textstyle\left(\frac{\epsilon|\diffcont\theta_{n, \epsilon}(x_i)|}{\theta_{n, \epsilon}(x_i)} + \left|\rho(x_i) - \frac{\theta_{n, \epsilon }(x_i)}{\theta_{\epsilon, K}}\right|\right) \left(\sum_{x_j\in B(x_i, \epsilon)}\frac{ |\diffgraph \phi_{x_ix_j}|}{2n\theta_{n, \epsilon }(x)}\right)^2\\
      & \qquad +  \textstyle\frac{1}{\rho(x_i)^2}\left|\sum_{x_j\in B(x_i, \epsilon)} \frac{\diffcont\psi_{\epsilon}(x_i, x_j)}{(n-1)\theta_{\epsilon, K}}(\phi(x_i) - \phi(x_j))\right|^2
    \end{split}
  \end{equation}
  holds. For every $r\in [1, \infty)$, we have
  \begin{equation}\label{eq:InterDiricSumStep2}
    \frac{1}{n}\sum_{i=1}\left(\sum_{x_j\in B(x_i,\epsilon)}\frac{|\diffgraph \phi(x_ix_j)|}{n\theta_{n, \epsilon}(x_i)}\right)^{r}
    \leq \frac{C(m, K, D)^r}{n\epsilon^{r}}\sum_{i=1}^n\phi(x_i)^{r}.
  \end{equation}
  For the first term of \eqref{eq:InterDiricSumStep1}, by the H\'{o}lder inequality, combining (ii) of Lemma~\ref{lem:interpolationlemmas} with \eqref{eq:InterDiricSumStep2} yields
  \begin{equation}\label{eq:InterDiricSumStep3}
    \begin{split}
      &\textstyle\frac{1}{n-2}\sum_{i=1}^n \left(\frac{\epsilon|\diffcont\theta_{n, \epsilon}|}{\theta_{n, \epsilon}(x_i)} + \left|\rho(x_i) - \frac{\theta_{n, \epsilon }(x)}{\theta_{\epsilon, K}}\right|\right) \left(\sum_{x_j\in B(x_i, \epsilon)}\frac{ \diffgraph \phi_{x_ix_j}}{n\theta_{n, \epsilon }(x)}\right)^{2( q^{-1}+1-q^{-1})}\\
      &\quad\leq  C(m,K,D,v,\alpha,\mathcal{L})\epsilon^{-\frac{2}{p}}\left( a + \epsilon + a^{\frac{2}{p}}\epsilon^{\frac{m}{p}} + {V_{p,\epsilon}(M)}\right)
      \left(\sum_{i=1}^n\frac{\phi(x_i)^{2(1-q^{-1})\epsilon^{-2}}}{n}\right)^{\epsilon^2}\\
      &\qquad\textstyle\left(\sum_{i=1}\sum_{x_j\in B(x_i,\epsilon)}\frac{\left(\diffgraph \phi(x_ix_j)\right)^2}{n^2\epsilon^m}\right)^{\frac{1}{q}}.
    \end{split}
  \end{equation}
  We used $\sup_{\epsilon\in (0,1)}\epsilon^{-\epsilon^2} < \infty$ here.
  For the second term of \eqref{eq:InterDiricSumStep1}, by (i) of Lemma~\ref{lem:interpolationlemmas}, with the probability of this lemma, we get
  \begin{equation}
    \begin{split}
      &\frac{1}{\rho(x_i)^2}\left\langle \sum_{j=1}^n \frac{\diffcont \psi_\epsilon(x_i,x_j)}{(n-1)\theta_{\epsilon,K}}(\phi(x_i) - \phi(x_j)), w \right\rangle^2\\
      &\quad\leq \frac{|w|^2}{(n-1)\rho(x_i)\theta_{\epsilon,K}^2}\left(\frac{\omega_m\epsilon^m}{m+2} + \epsilon^m(a+\epsilon)C(m,K,D,\alpha,\mathcal{L},L) \right)
      \sum_{x_j\in B(x_i,\epsilon)}( \diffgraph\phi_{ij})^2
    \end{split}
  \end{equation}
  for any $i \in \{1,\dots,n\}$ and any $w \in T_{x_i} M$.
  Hence, using $\left|\theta_{\epsilon,K} - \frac{\omega_m\epsilon^m}{m+2}\right| \leq  \epsilon^{m+1}C(m,K)$, we have
  \begin{equation}
    \begin{split}
    &\frac{1}{n-2}\sum_{i=1}^n\frac{1}{\rho(x_i)}\left|\sum_{x_j\in B(x_i, \epsilon)} \frac{\diffcont\psi_{\epsilon}(x_i, x_j)}{(n-1)\theta_{\epsilon, K}}(\phi(x_i) - \phi(x_j))\right|^2 \\
    &\quad \leq 
    \frac{\bigl((m+2) + (a+\epsilon)C(m,K,D,\alpha,\mathcal{L},L) \bigr)}{(n-1)(n-2)\omega_m\epsilon^m}\sum_{i=1}^n\sum_{x_j\in B(x_i,\epsilon)} (\diffgraph\phi_{ij})^2.
    \end{split}
  \end{equation}
  Combining \eqref{eq:InterDiricSumStep1} with this and \eqref{eq:InterDiricSumStep3}, we obtain
  \begin{equation}\label{eq:DiscDirichLastStep}
    \begin{split}
    &\frac{1}{n-2}\sum_{i=1}^n |\diffcont(\Lambda_\epsilon\phi)(x_i)|^2\rho(x_i)\\
    &\quad\leq  C_2\epsilon^{-\frac{2}{p}} \left( a + \epsilon  + a^{\frac{2}{p}}\epsilon^{\frac{m}{p}} + {V_{p,\epsilon}(M)}\right)
    \left(\sum_{i=1}\sum_{x_j\in {B}(x_i,\epsilon)}\frac{(\diffgraph \phi_{ij})^2}{n^2\epsilon^m}\right)^{\frac{1}{q}}\left(\sum_{i=1}^n\frac{\phi(x_i)^{2(1-q^{-1})\epsilon^{-2}}}{n}\right)^{\epsilon^2}\\
    & \qquad+ \frac{\bigl(1+(\epsilon + a)C_2\bigr)(m+2)}{(n-1)(n-2)\omega_m\epsilon^m}\sum_{i=1}\sum_{x_j\in {B}(x_i,\epsilon)}\left(\diffgraph \phi_{ij}\right)^2.
    \end{split}
  \end{equation}
  Similar to Lemma~\ref{lem:interpolationlemmas}, we can assume $\data\cap B(x_i,\epsilon)\setminus \tilde{B}(x_i,\epsilon) = \emptyset$.
  Hence, the inequality~\eqref{eq:DiscDirichLastStep} implies the desired inequality.
\end{proof}

\begin{proposition}\label{prop:InterNormLast}
  There exist constants $C_1 = C_1(m)$,  $C_2=  C_2(m,K,D,v,\alpha,\mathcal{L},L)$, and  $C_3= C_3(m,v,\alpha)$ such that 
  for $\epsilon, a \in (0,1)$ with $(\epsilon + a) C_3\leq 1$ and $\tau < \epsilon$,
  with probability at least $1-C_2(n^2+ \epsilon^{-2m})\exp(-na^2\epsilon^m)-C_3n^2\tau \epsilon^{m-1}$, we have
  \begin{equation}\label{eq:unnormalizednorm}
    \begin{split}
      \left|\int_M|\Lambda_\epsilon\phi(x)|^2\rho(x)\, dx-\sum_{i=1}^n\frac{\phi(x_i)^2}{n-2}\right|
      \leq (\epsilon + a) C_3\left(\sum_{i=1}^n\frac{\phi(x_i)^2}{n} + \sum_{i=1}^n\sum_{x_j\in \tilde{B}(x_i,\epsilon)}\frac{(\diffgraph \phi_{ij})^2}{n^2\epsilon^m}\right)
    \end{split}
  \end{equation}
  and
  \begin{equation}\label{eq:normalizednorm}
    \begin{split}
          &\left|\int_M| (\Lambda_\epsilon\phi) |^2\rho(x)^2 \, dx-\sum_{i=1}^n\frac{\phi(x_i)^2\deg(x_i)}{(n-2)(n-1)\omega_m\epsilon^m}\right|\\
          &\quad \leq {C_2\left(a + \epsilon + V_{1, \epsilon}(M) + \epsilon^{-m}S_{\epsilon}\right)}
          \left( \left(\sum_{i=1}\frac{\phi(x_i)^{2\epsilon^{-2}}\deg(x_i)}{n^2\epsilon^m}\right)^{\epsilon^2} + \sum_{i=1}\sum_{x_j\in \tilde{B}(x_i,\epsilon)}\frac{\left(\diffgraph \phi_{ij}\right)^2}{n^2\epsilon^m}\right)
    \end{split}
  \end{equation}
  for all $\phi: \mathcal{X}_n \to \R$.
\end{proposition}

\begin{proof}
  For $l = 1, 2$, We have
  \begin{equation}
    \int_M
    | (\Lambda_\epsilon\phi) |^2\rho(x)^l \, dx
    = \frac{1}{(n-1)^2}\sum_{i, j=1}^m \phi(x_i)\phi( x_j)\int_M\tilde{\psi_\epsilon}(x,x_i)\tilde{\psi_\epsilon}(x,x_j)\rho(x)^l\,dx.
  \end{equation}
  Then similar to \eqref{eq:CompareCore}, with probability at least $1-C(m,K,D)(\epsilon^{-2m} + n^2)\exp(-na^2\epsilon^m)$, we have
  \begin{equation}
    \left| \int_M \tilde{\psi_\epsilon}(x,x_i)\tilde{\psi_\epsilon}(x,x_j)\rho(x)^l \,dx - \sum_{k\neq i,j} \frac{\tilde{\psi_\epsilon}(x_k,x_i)\tilde{\psi_\epsilon}(x_k,x_j)}{n-2}\rho(x)^{l-1} \right| \leq \epsilon^{m}(\epsilon + a) C(m,K,D,v,\alpha,\mathcal{L}).
  \end{equation}
  Hence, by this dependent $C$,
  \begin{equation}\label{interpolationnorm1}
    \begin{split}
      &\left|\int_M| \Lambda_\epsilon\phi |^2\rho^i \, d\vol_g
      - \frac{1}{n-2}\sum_{k=1}^n|\Lambda_\epsilon\phi(x_k)|^2\rho(x_k)^{i-1}\right|\\
      & \quad \leq \frac{(\epsilon + a)C}{\epsilon^m n^2}\sum_{i=1}\sum_{x_j\in B(x_i,2\epsilon)}{\phi(x_i)\phi(x_j)}
      \leq \frac{(\epsilon + a)C}{n}\sum_{i=1}^n\phi(x_i)^2
    \end{split}
  \end{equation}
  holds.
  Since $\Lambda_\epsilon \phi(x_i) - \phi(x_i) = \frac{1}{(n-1)\theta_{n,\epsilon}(x_i)}\sum_{j=1}^n{\psi}_\epsilon(x_i,x_j)(\phi(x_j) -\phi(x_i))$, we obtain 
  \begin{equation}
    |\Lambda_\epsilon \phi(x_i)^2 - \phi(x_i)^2| \leq \epsilon\left(\frac{1}{(n-1)\theta_{n,\epsilon}(x_i)} \sum_{x_j\in B(x_i,\epsilon)} \diffgraph\phi_{ij}^2 + \phi(x_i)^2 \right).
  \end{equation}
  Thus,
  \begin{equation}\label{interpolationnorm2}
    \begin{split}
      &\left|\frac{1}{n-2}\sum_{i=1}^n|\Lambda_\epsilon\phi(x_i)|^2\rho(x_i)^{i-1}-\frac{1}{n-2}\sum_{i=1}^n\phi(x_i)^2\rho(x_i)^{i-1} \right|\\
      &\quad \leq \epsilon  C(m,K,D,v,\alpha)\left(\frac{1}{n}\sum_{i=1}^n\phi(x_i)^2 + \frac{1}{n^2\epsilon^m}\sum_{i=1}^n\sum_{x_j\in B(x_i,\epsilon)}\left(\diffgraph \phi_{ij}\right)^2\right).
    \end{split}
  \end{equation}
  By the inequality \eqref{interpolationnorm1} and \eqref{interpolationnorm2} for $i=1$, we obtain \eqref{eq:unnormalizednorm}.
  For the remaining part of this proposition, by (iii) of Lemma~\ref{lem:interpolationlemmas} with $p^{-1} = 1-\epsilon^2$, we have
  \begin{align}\label{interpolationnorm3}
    &\left|\frac{1}{n-2}\sum_{i=1}^n\phi(x_i)^2\rho(x_i)- \frac{1}{(n-1)(n-2)\omega_m\epsilon^m}\sum_{i=1}^n\phi(x_i)^2\deg(x_i)\right|\\
    \leq &\frac{1}{n-2}\sum_{i=1}^n\phi(x_i)^2 
    \left| \rho(x_i) - \frac{\deg(x_i)}{(n-1)\omega_m\epsilon^m}\right|\\
    \leq &C(m, K, D,\alpha,L,\mathcal{L},v)\left(a  + \epsilon + (a^2\epsilon^{m}) + V_{1, \epsilon}(M) + \epsilon^{-m}S_{\epsilon}\right) \left(\sum_{i=1}^n\frac{ \phi(x_i)^{2\epsilon^{-2}}}{n}\right)^{\epsilon^2}.
  \end{align}
  holds.
  Combining this, \eqref{interpolationnorm1}, and \eqref{interpolationnorm2}, using \eqref{eq:DegRegul}, we obtain \eqref{eq:normalizednorm}.
\end{proof}

\begin{theorem}\label{thm:EigenvCFromD}
  For $k\in\N$ and $ H \in \R$, there exist constants $C_1 = C_1(m,K,D,k)>0$, $C_2 =C_2(m,k,K,D,\alpha,\mathcal{L}, H,v, L)>0$, and $C_3 = C_3(m, \alpha,v)>0$ such for $\epsilon, a\in (0,1)$ with $\epsilon> 4\tau$,
  for $p > (1-\epsilon^2)^{-1}$, setting
  \[
    \eta_p:=\epsilon^{-\frac{1}{p}}( a + \epsilon + \epsilon^{\frac{m}{p}}a^{\frac{2}{p}} + V_{p,\epsilon}(M)),
  \]
  and assuming $\rho \in \mathcal P(M\colon \alpha, \mathcal L,  H)$,
  we have the following estimates:
  \begin{enumerate}
    \item If $(a+\epsilon)C_2\leq 1$, we have
      \begin{equation}\label{eq:AofEigenvCFromD}
        \lambda_k(\Delta_\rho) \leq (m+2)\lambda_k({\datagraph}) +C_2\eta_p
      \end{equation}
      with probability at least $1-(\epsilon^{-2m} + n^3)C_1\exp( -na^2\epsilon^m) - n^2\tau\epsilon^{m-1}C_3$,
    \item If $(a + \epsilon) C_2 \leq 1$, we have
      \begin{equation}\label{eq:BofEigenvCFromD}
        \lambda_k(\Delta_\rho^N) \leq (m+2)\lambda_k({\datangraph}) +C_2(\eta_p + \epsilon^{-m} S_\epsilon(M))
      \end{equation}
      with the same probability bound as in \textnormal{(ii)}.
  \end{enumerate}
\end{theorem}

\begin{proof} 
  Similar to proof of Theorem~\ref{thm:EigenvDFromC}, we show this theorem by comparison of Rayleigh quotients through $\Lambda_\epsilon$.
  Let us show the inequality~\eqref{eq:AofEigenvCFromD}.
  Set $\Gamma =\datagraph$.
  Set an orthonormal functions $\phi_1, \dots, \phi_k\colon M \to \R$ with $\Delta_\Gamma \phi_j = \lambda_j(\Gamma) \phi_j$ for each $j$.
  By Lemma~\ref{lem:InterDirichDisc}, Lemma~\ref{lem:InterDiricSum}, Proposition~\ref{prop:InterNormLast}, and Theorem~\ref{thm:EstEigenfData}, using Remark~\ref{rem:UpperEigenGraph}, there are two constants $C_1=C_1(m,k,K,D,v,\alpha)>0$ and $C_2=C_2(m,k,K,D,v,L,\alpha, H,\mathcal{L})>0$ such that,
  if $(a+\epsilon)C_1\leq 1$,
  with probability at least $1-C(m,K,D,k)(\epsilon^{-2m} + n^3)\exp(-na^2\epsilon^m) -n^2\tau\epsilon^{m-1}C(m, \alpha, v)$, we have
  \begin{equation}
    \begin{split}
      \frac{ \vol_g(M)\Lnorm{\diffcont \Lambda_\epsilon \phi}{2}{M}{\rho^2\vol_g}^2 }{(m+2)\Lnorm{\Lambda_\epsilon\phi}{2}{M}{\rho\vol_g}^2}\leq \sum_{i=1}^n\sum_{x_j\in \tilde{B}(x_i,\epsilon)} \frac{\vol_g(M)(\diffgraph \phi_{ij})^2}{n(n-1)\omega_m\epsilon^m} +\eta_pC_2
    \end{split}
  \end{equation}
  for any $\phi \in \spanv\{\phi_1, \dots,\phi_k\}$ with $\Lnorm{\phi}{2}{M}{\vol_\Gamma}$=1.
  By the min-max principle, taking the supremum of the above inequality, we obtain (i).
  We can derive the inequality~\eqref{eq:BofEigenvCFromD} similarly.
\end{proof}

\section{Estimates for the eigenvalues and eigenfunctions of Laplacians on Riemannian manifolds and non-collapsed Ricci limit spaces}\label{sec:Limit}
This section provides our main results on discrete approximations to the eigenvalues and eigenfunctions of weighted Laplacians, both on Riemannian manifolds in \(\RicClassBV\) and on non-collapsed Ricci limit spaces approximated by them.

For convenience, for $p >2$ and $\epsilon, a, \tau \in (0, 1)$, we introduce the quantity
\[
  \delta_{p,\epsilon,a}(M, d_M, \tilde{d})
  :=  a +\epsilon^{\min\bigl\{1-\tfrac{2}{p},\tfrac{m}{p-2}-\tfrac{2}{p}\bigr\}}
  + V_{p,\epsilon}(M,d_M)\epsilon^{-\tfrac{2}{p}}
  + \epsilon^{-m}S_\epsilon(M,d_M,\tilde{d}),
\]
where \(\bigl(M,d_M\bigr)\) is a compact metric space with Hausdorff dimension $m$, and $\tilde d$ is a Borel pseudo-metric on $M$.

First, we show that the eigenvalues (and, subsequently, eigenfunctions) of the graph Laplacians approximate those of \(\Delta_\rho\) and \(\Delta_\rho^N\) (Theorem~\ref{thm:AppEigenv} and \ref{thm:AppEigenNonSingf}).

\begin{theorem}[Eigenvalue approximation on Riemannian manifolds]\label{thm:AppEigenv}
  For $k\in\N$, there exist constants
  \[
    C_1 = C_1(m,K,D,k) >0,\quad
    C_2 = C_2(m,K,D,v,\alpha,\mathcal{L}, H,L,k)>0,\quad
    C_3 = C_3(m,v,\alpha)
  \]
  such that the following property holds.

  Let $\epsilon, a, \tau \in (0,1)$ with $\epsilon > 4\tau$,
  and let $p>2$.
  Let $(M, g)$ in $\RicClassBV$, $\rho \in \mathcal{P}(M\colon \alpha,\mathcal{L}, H)$, and $\tilde d \in \mathcal{I}_{L,\tau}(M, d_g)$.
  Let $\data=(x_1,\dots,x_n)$ be a data set drawn from $\rho\,\vol_g$.
  Then the following estimates hold:
  \begin{enumerate}
    \item If $(a+\epsilon)C_2 \le 1$, then
      \[
        \bigl|\lambda_k(\Delta_\rho)-(m+2)\lambda_k(\Gamma_{m,\epsilon}(\data, \tilde{d}))\bigr| 
        \le
        C_2\delta_{p,\epsilon,a}(M, d_g, \tilde{d})
      \]
      with probability at least
      \[
        1-
        \Bigl(\epsilon^{-2m} + n^3\Bigr)C_1\exp\bigl(-na^2\epsilon^{m+\frac4p}\bigr)
        -
        n^2\tau\epsilon^{m-1}C_3.
      \]
    \item If $\bigl(a + \epsilon\bigr) C_2 \le 1$,
      then
      \[
        \bigl|\lambda_k(\Delta_\rho^N)-(m+2)\lambda_k(\Gamma^N_{\epsilon}(\data, \tilde{d}))\bigr|
        \le
        C_2\delta_{p,\epsilon,a}(M, d_g, \tilde{d})
      \]
      with the same probability bound as in \textnormal{(i)}.
  \end{enumerate}
\end{theorem}

\begin{proof}
  Note the elementary estimate
  \[
    \Bigl(a\epsilon^{\frac2p}\Bigr)^{\frac2p}
    \epsilon^{\frac{m-2}{p}}
    \le
    a
    +
    \epsilon^{\tfrac{m}{p-2}-\tfrac{2}{p}}.
  \]
  The result then follows by applying Theorem~\ref{thm:EigenvDFromC} and Theorem~\ref{thm:EigenvCFromD} with $(a\epsilon^{2/p}, \epsilon)$ in place of $(a, \epsilon)$.
\end{proof}

Next, we provide the approximation for the eigenfunctions on $\RicClassBV$.

\begin{theorem}[Eigenfunction approximation on Riemannian manifolds]\label{thm:AppEigenNonSingf}
  Let $k, l\in \N$ with $l \ge k$.
  Then there exist positive constants
  \[
    C_1 = C_1(m,K,D,l)>0,\;
    C_2 = C_2\bigl(m,K,D,v,\alpha,\mathcal{L}, H,L,l\bigr)>0,\;
    C_3 = C_3(m,\alpha,v)>0
  \]
  such that the following property holds.

  Let $\epsilon, a, \tau \in (0,1)$ and $p>2$.
  Let $(M,g)\in \RicClassBV$, $\rho \in \mathcal{P}(M\colon \alpha,\mathcal{L}, H)$, and $\tilde{d}\in \mathcal{I}_{L,\tau}(M, d_g)$.
  Set the weighted Laplacians
  \[
    \Delta_1 := \Delta_{\rho}\quad\textrm{and}\quad
    \Delta_2 := \Delta^N_{\rho}.
  \]
  For \(i = 1,2\), set
  \begin{align*}
    s(\rho) &:= \lambda_l(\Delta_i)-\lambda_k(\Delta_i), \\
    \gamma(\rho) &:= \tfrac12 \,\min\Bigl\{\,\lambda_k(\Delta_i) - \lambda_{k-1}(\Delta_i),\,\lambda_{l+1}(\Delta_i) - \lambda_l(\Delta_i),\,1\Bigr\}.
  \end{align*}
  Let \(\{f_s\}_{s=0}^\infty\) be an orthonormal family of eigenfunctions of \(\Delta_i\) in $L^2(M, \rho^i\vol_g)$ corresponding to the eigenvalues \(\{\lambda_s(\Delta_i)\}_{s=0}^\infty$.
  Suppose that
  \[
    F := \left(\delta_{p,\epsilon,a}(M, d_g, \tilde{d})+\tau\epsilon^{-1}
      +
    s(\rho)\gamma(\rho)\right)C_2
    \le\gamma(\rho)^2.
  \]
  Let \(\data=(x_1,\dots,x_n)\) be a data set drawn from \(\rho\,\vol_g\). 
  Set $\Gamma_1 = \Gamma_{m,\epsilon}(\data, \tilde{d})$ and $\Gamma_2 = \Gamma^N_{\epsilon}(\data, \tilde{d})$.
  Then the following holds with probability at least
  \[
    1 -
    \Bigl(\epsilon^{-2m} + n^3\Bigr)C_1\exp\bigl(-na^2\epsilon^{m+\frac{4}{p}}\bigr)    -
    n^2\tau\epsilon^{m-1}C_3:
  \]
  Let \(\{\phi_s\}_{s=k}^l\) be an orthonormal family of eigenfunctions in $L^2(\data,\vol_{\Gamma_i})$ corresponding to the eigenvalues \(\lambda_k(\Gamma_i),\dots,\lambda_l(\Gamma_i)\).
  Let $p \colon  L^2(\data,\vol_{\Gamma_i}) \to \mathrm{span}\{\phi_k,\dots,\phi_l\}$ be the orthogonal projection onto this subspace.
  Then,
  \begin{enumerate}
    \item for every \(f\in \mathrm{span}\{f_k,\dots,f_l\}\), 
      \begin{align}
        \|(I-p)(f\vert_{\data}) \|_{L^2(\data,\vol_{\Gamma_i})} \leq&  2\sqrt{F}\| f\|_{L^2(\data,\rho^i\vol_g)}, \label{eq:AppEigenNonSingG1_riem} \\ 
        \Bigl|
        \|f\|_{L^2(M,\rho^i\vol_g)} 
        -
        \|p\bigl(f\vert_{\data}\bigr)\|_{L^2(\data,\vol_{\Gamma_i})}
        \Bigr|
        \le & 2F\|f\|_{L^2(M,\rho^i\vol_g)};
        \label{eq:AppEigenNonSingG2_riem}
      \end{align}
    \item there exists an orthonormal basis $\{\tilde{f}_k,\dots,\tilde{f}_l\}$ of $\mathrm{span}\{f_k,\dots,f_l\}$ such that 
      \begin{equation}\label{eq:AppEigenNonSingG3_riem}
        \|\tilde{f_j}\vert_{\data} - \phi_j \|_{L^2(\data,\vol_{\Gamma_i})} \leq \sqrt{F}
      \end{equation}
      holds for each $j \in \{k, \dots, l\}$.
  \end{enumerate}
\end{theorem}

\begin{proof}
  We prove the case \(i=1\); the argument for \(i=2\) is analogous. 
  Let $\Gamma = \Gamma_{m,\epsilon}(\data, \tilde{d})$.
  Define the spaces
  \[
    H_1 :=\mathrm{span}\bigl\{f_0,\dots,f_{l}\bigr\} +\Lambda_\epsilon\bigl(L^2(\data,\vol_{\Gamma})\bigr),
    \quad
    H_2 := L^2(\data,\vol_{\Gamma}),
  \]
  where \(\Lambda_\epsilon\) is the interpolation map defined in Section~6.
  Using Lemma~\ref{lem:EstCEigenF}, for any \(f\in \mathrm{span}\{f_0,\dots,f_{l}\}\), we have
  \begin{equation}\label{eq:AppEigenNonSingPf1}
    \|f - \Lambda_\epsilon\bigl(f\vert_{\data}\bigr)\|_{L^2(M,\rho\vol_g)}
    \le
    \mathrm{Lip}(f)\epsilon
    \le
    \|f\|_{L^2(M,\rho\vol_g)}C(m,K,D,\alpha,L)\epsilon.
  \end{equation}
  Next, by Proposition~\ref{prop:normofcont} and Lemma~\ref{lem:EstCEigenF}, we also have
  \begin{equation}\label{eq:AppEigenNonSingPf2}
    \Bigl|\,
    \|f\|_{L^2(M,\rho\vol_g)}
    -
    \|f\vert_{\data}\|_{L^2(\data,\vol_{\Gamma})}
    \Bigr|
    \le C(m,K,D,\alpha)(\epsilon + a)\|f\|_{L^2(M,\rho\vol_g)}
  \end{equation}
  for any $f \in \mathrm{span}\{f_0, \dots, f_{l}\}$.
  In addition, for $\phi\in \mathrm{span}\{\phi_1\dots,\phi_{l}\}$, by Theorem~\ref{thm:EstEigenfData}, we have $\|\phi\|_{L^q(\data, \vol_\Gamma)} \leq C q^{ C\epsilon^2}\|\phi\|_{L^2(\data, \vol_\Gamma)}$ for any $q>2$, where $C = C(m,K,D,\alpha,v,L)$.
  By Remark~\ref{rem:UpperEigenGraph}, $\|\diffgraph \phi \|_{L^2(\data, \vol_\Gamma)} \leq C \|\phi\|_{L^2(\data, \vol_\Gamma)}$ also holds, where $C= C(m,K,D,v,\alpha,\mathcal L, H,L)$.
  Using these two inequalities, by Lemma~\ref{lem:InterDirichDisc}, Lemma~\ref{lem:InterDiricSum}, and Proposition~\ref{prop:InterNormLast}, we obtain
  \begin{equation}\label{eq:AppEigenNonSingPf3}
    \frac{\displaystyle\int_M \bigl|\diffcont\bigl(\Lambda_\epsilon\phi\bigr)\bigr|^2 \rho^2\,d\vol_g}
    {\displaystyle\int_M\bigl|\Lambda_\epsilon\phi\bigr|^2 \rho\,d\vol_g}
    \;-\;
    \frac{\displaystyle(m+2)\sum_{i=1}^n\sum_{x_j\in \tilde{B}(x_i,\epsilon)}(\diffgraph \phi_{ij})^{2}}
    {\displaystyle (n-1)\omega_m\epsilon^m\sum_{i=1}^n \phi(x_i)^2}\\
    \le
    C\delta_{p,\epsilon,a}(M, d_g, \tilde{d})
  \end{equation}
  for any \(\phi \in \mathrm{span}\{\phi_1,\dots,\phi_{l}\}\).
  By Lemma~\ref{lem:EstCEigenF}, the same inequality also holds for $\phi$ being the restriction of functions in $\mathrm{span}\{f_1,\dots,f_{l+1}\}$ to $\data$.
  Meanwhile, by Proposition~\ref{prop:dirichletofcont}, Proposition~\ref{prop:normofcont}, and Lemma~\ref{lem:EstCEigenF},
  \begin{equation}\label{eq:AppEigenNonSingPf4}
    \frac{\displaystyle(m+2)\sum_{i=1}^n\sum_{x_j\in \tilde{B}(x_i,\epsilon)} \left(\diffgraph\bigl( f\vert_{\data}\bigr)_{ij}\right)^2}
    {\displaystyle (n-1)\omega_m \epsilon^m \sum_{i=1}^n f(x_i)^2}
    \;-\;
    \frac{\displaystyle \int_M \bigl|\diffcont f\bigr|^2\rho^2\,d\vol_g}
    {\displaystyle \int_M \bigl|f\bigr|^2\rho\,d\vol_g}\\
    \le
    (\epsilon + a)C
  \end{equation}
  for any \(f \in \mathrm{span}\{f_1,\dots,f_{l+1}\}\).
Finally, applying \cite{MR4804972}*{Lemma B.4 (i), (ii)} with the bounds \eqref{eq:AppEigenNonSingPf1}, \eqref{eq:AppEigenNonSingPf2}, \eqref{eq:AppEigenNonSingPf3}, and \eqref{eq:AppEigenNonSingPf4}, we derive the conclusions \eqref{eq:AppEigenNonSingG1_riem} and \eqref{eq:AppEigenNonSingG2_riem}.
  The remaining part \eqref{eq:AppEigenNonSingG3_riem} is an easy consequence of \eqref{eq:AppEigenNonSingG1_riem} and \eqref{eq:AppEigenNonSingG2_riem}.
\end{proof}

Letting $p=m+2$, $\epsilon = \left(\frac{\log n}{n}\right)^{\frac{1}{m+2}}$, and $a = \sqrt{\beta + 3} \epsilon$, using Theorems~\ref{thm:AppEigenv} and~\ref{thm:LimAppEigenf}, we obtain Theorems~\ref{thm:MainEV} and~\ref{thm:MainEF}.

The following theorem extends these results to the Laplacian $\Delta_\rho^N$ on non-collapsed Ricci limit spaces approximated by manifolds in $\RicClassBV$.
\begin{theorem}[Eigenvalue approximation on Ricci limit spaces]\label{thm:LimAppEigenv}
  For $k\in\mathbb{N}$, there exist constants
  \[
    C_1 = C_1(m,K,D,k)>0 \quad \textrm{and}\quad 
    C_2 = C_2(m,k,K,D,\alpha,\mathcal{L}, H,v,L)>0
  \]
  such that the following holds.

  Let $\bigl(M,d_M,\rho\bigr)\in\RicClassL$ and $\tilde d\in\mathcal{I}_{L}(M,d_M)$.
  Let $\epsilon,a\in(0,1)$, fix $p>2$, and draw a data set $\mathcal X_n=(x_1,\dots,x_n)\colon \Omega \to M^n$ from $\rho\volmm$.
  If
  \begin{equation}\label{eq:LimEigenvAssum}
    \bigl(\epsilon+a\bigr) C_2 \le 1,
  \end{equation}
  then we have 
  \begin{equation}\label{eq:LimEigenvGoal}
    \bigl|
    \lambda_k(\Delta_\rho^{N})-(m+2)\lambda_k\bigl(\Gamma^{N}_{\epsilon}(\mathcal X_n,\tilde d)\bigr)
    \bigr|
    \le
    C_2\delta_{p,\epsilon,a}(M, d_M, \tilde{d})
  \end{equation}
  with probability at least
  \[
    1-
    \left(\epsilon^{-2m}+n^{3}\right) C_1
    \exp\bigl(-na^{2}\epsilon^{\,m+\frac{4}{p}}\bigr).
  \]
\end{theorem}

\begin{proof}
  Let \(\bigl((M,d_M),\rho\bigr)\in\RicClassL\).
  Let \(\{(M_t,g_t)\}_{t=1}^{\infty}\subset\RicClassBV\) and
  \(\{\rho_t\colon M_t\to[0,\infty)\}_{t=1}^{\infty}\) be the sequence of manifolds and density functions from Definition~\ref{def:ClassLimit}.
  Moreover, let \(\{\delta_t>0\}_{t=1}^{\infty}\) and \(\{\Phi_t\colon  M_t\to M\}_{t=1}^{\infty}\) be the sequence of constants and maps satisfying conditions \textup{(i)}--\textup{(iii)} of Definition~\ref{def:GH}.

  Choose \(\tilde\epsilon\in(\epsilon/2,\epsilon)\).
  Then, by Egorov's theorem and the Portmanteau theorem, using standard arguments from functional analysis, we obtain
  \[
    V_{p,\tilde\epsilon}(M_t, d_{g_t})\to V_{p,\tilde\epsilon}(M, d_M)
    \quad\textrm{and}\quad
    S_{\tilde\epsilon}(M_t, d_{g_t}, \Phi_t^*\tilde{d})\to S_{\tilde\epsilon}(M,d_M,\tilde{d})
  \] as \(t\to\infty\).  Hence, by $\lambda_k(\Delta^N_{\rho_t}) \to \lambda_k(\Delta_\rho^N)$ (Theorem~\ref{thm:mGHSpecConv}),  we have 
  \begin{equation}\label{eq:LimAppEigenvStep1_pf}
    \sup_{s\in[\tilde\epsilon,\epsilon]}
    \delta_{p,s,a}\bigl(M, d_M, \tilde d\bigr)
    \ge\bigl|\lambda_k(\Delta_\rho^N)-\lambda_k(\Delta^N_{\rho_t})\bigr| + \delta_{p, \tilde\epsilon,a}(M, d_M, \tilde d)
  \end{equation}
  for sufficiently large \(t\).

  We will show this theorem by applying Theorem~\ref{thm:AppEigenv} to approximating manifolds.
  Let $C_1$, $C_2$, and $C_3$ be the constants from Theorem~\ref{thm:AppEigenv}.
  We have $\Phi_t^*\tilde{d} \in \mathcal I_{L, L\delta_t}( M_t, d_{g_t})$ by Remark~\ref{rem:PMet}.
  We can assume that $(\tilde\epsilon+a)\,C_2\le1$, and $4L\delta_t < \epsilon$ for sufficiently large $t$.
  Set
  \[
    W_{\tilde\epsilon}:=
    \bigl\{\mathcal X_n = (x_1, \dots, x_n) \in M^n:
      \text{for all }i, j,\
      \tilde d(x_i,x_j)\notin[\tilde\epsilon,\epsilon)
    \bigr\}.
  \]
  Then, we estimate the probability on $M$: 
  \begin{align*}
   &\bigl(\rho\volmm\bigr)^{\otimes n}
   \left\{
     \mathcal X_n\in W_{\tilde\epsilon}:
     \begin{aligned}
       &\left|\lambda_k(\Delta_\rho^N)-(m+2)\lambda_k(\Gamma^N_{\epsilon}(\mathcal X_n,\tilde d))\right|\\ 
       &\quad>C_2\sup_{s\in[\tilde\epsilon,\epsilon]}
       \delta_{p,s,a}\left(M,d_M, \tilde d\right)
     \end{aligned}
   \right\}\\
   &\quad =\lim_{t\to \infty}\bigl(\rho_t\vol_{g_t}\bigr)^{\otimes n}
    \left\{
      \mathcal X_n^t\in M_t^n:
      \begin{aligned}
           &\left|\lambda_k(\Delta^N_{\rho})-(m+2)\lambda_k\left(\Gamma^N_{\tilde\epsilon}(\mathcal X_n^t,\Phi_t^*\tilde d)\right)\right|\\
           &\quad>C_2\sup_{s\in[\tilde\epsilon,\epsilon]}
       \delta_{p,s,a}\left(M,d_M, \tilde d\right)
      \end{aligned}
    \right\} \\
    &\quad\leq 
    \liminf_{t\to \infty}\bigl((\Phi_t)_*(\rho_t\vol_{g_t})\bigr)^{\otimes n}
    \left\{
      \mathcal X_n^t\in M_t^n:
      \begin{aligned}
           &\left|\lambda_k(\Delta^N_{\rho_t})-(m+2)\lambda_k\left(\Gamma^N_{\tilde\epsilon}(\mathcal X_n^t,\Phi_t^*\tilde d)\right)\right|\\
           &\quad>  C_2 \delta_{p, \tilde \epsilon,a}\left(M_t,d_{g_t},\Phi_t^*\tilde d\right)
      \end{aligned}
    \right\} \\
    &\quad\le 
    \liminf_{t\to \infty} \left[ \Bigl(\tilde{\epsilon}^{-2m} + n^3\Bigr)C_1\exp\bigl(-na^2\tilde\epsilon^{m+\frac4p}\bigr) + n^2(L\delta_t)\tilde\epsilon^{m-1}C_3 \right] \\
    &\quad = \exp\!\bigl(-na^{2}\tilde\epsilon^{\,m+\frac{4}{p}}\bigr)\, C_1\bigl(\tilde\epsilon^{-2m}+n^{3}\bigr), 
  \end{align*}
  where  we used the weak* convergence $((\Phi_t)_*(\rho_t\vol_{g_t}))^{\otimes n} \to (\rho\volmm)^{\otimes n}$ in the second line.
  We used $\Gamma^N_{\tilde\epsilon}(\Phi_t(\mathcal{X}_n), \tilde{d}) = \Gamma^N_\epsilon(\mathcal X_n, \Phi_t^*\tilde{d})$ and \eqref{eq:LimAppEigenvStep1_pf} in the third line.
  In the forth line, we applied Theorem~\ref{thm:AppEigenv}(ii) to $(M_t, g_t, \rho_t, \Phi_t^*\tilde{d})$ with $\tau = L\delta_t$.

  This bound holds for every \(\tilde\epsilon\in(\epsilon/2,\epsilon)\).
  Since $\lim_{\tilde\epsilon \to \epsilon} (\rho\volmm)^{\otimes n}(W_{\tilde\epsilon}^{\,c}) = 0$, and using the upper semi-continuity
  \[
    \limsup_{s\nearrow\epsilon}S_s(M, d_M, \tilde{d})\leq S_\epsilon(M, d_M, \tilde{d}) \quad\textrm{and}\quad
  \limsup_{s\nearrow\epsilon}V_{p,s}(M, d_M)\leq V_{p,\epsilon}(M, d_M),  \]
  letting $\tilde{\epsilon} \uparrow \epsilon$ yields the desired estimate~\eqref{eq:LimEigenvGoal}.
\end{proof}

\begin{theorem}[Eigenfunction approximation on Ricci limit spaces]\label{thm:LimAppEigenf}
  Let $k, l\in \mathbb{N}$ with $l\ge k$, and let $p>2$.
  The constants $C_1, C_2, C_3$ depend on parameters similar to those in Theorem~\ref{thm:AppEigenNonSingf}.
  Let $\bigl(M,d_M,\rho\bigr)\in\RicClassL$ and $\tilde d\in\mathcal{I}_{L}(M,d_M)$.
  Let $\lambda_s := \lambda_s(\Delta_\rho^N)$.
  Define 
  \[
    s(\rho):=\lambda_l-\lambda_k \quad \textrm{and} \quad
    \gamma(\rho):=\tfrac12\min\{\lambda_k-\lambda_{k-1},\lambda_{l+1}-\lambda_l,1\}.
  \]
  Let $\epsilon, a \in (0, 1)$ and assume,
  \[
    F:= C_2\left( \delta_{p,\epsilon,a}(M, d_M, \tilde{d}) + s(\rho)\,\gamma(\rho)\right)\;\le\;\gamma(\rho)^2.
  \]
  Let $\{f_k,\dots,f_l\}$ be an orthonormal eigenfunctions in $L^2(M, \rho\volmm)$ for the eigenvalues $\{\lambda_k,\dots,\lambda_l\}$.
  Let $\mathcal X_n$ be an data set from $\rho\,\volmm$. Let $\Gamma = \Gamma^N_\epsilon(\mathcal X_n, \tilde{d})$ be the graph Laplacian, and let $\{\phi_k, \dots, \phi_l\}$ be an orthonormal eigenfunctions in $L^2(\mathcal{X}_n, \vol_\Gamma)$ for the corresponding eigenvalues $\lambda_k(\Gamma), \dots, \lambda_l(\Gamma)$.
  Let $p:L^2(\mathcal X_n,\vol_{\Gamma})\to \mathrm{span}\{\phi_k,\dots,\phi_l\}$ be the orthogonal projection.
  Then, with the same probability as in Theorem~\ref{thm:LimAppEigenv},
  \begin{enumerate}
    \item for every \(f\in \mathrm{span}\{f_k,\dots,f_l\}\),
      \begin{align}
        \|(I-p)(f\vert_{\data}) \|_{L^2(\data,\vol_{\Gamma})} \leq&  2\sqrt{F}\| f\|_{L^2(M,\rho^i\volmm)}, \label{eq:LimEigenfG1} \\ 
        \Bigl|
        \|f\|_{L^2(M,\rho \volmm)}
        -
        \|p\bigl(f\vert_{\mathcal{X}_n}\bigr)\|_{L^2(M,\rho^i\vol_g)}
        \Bigr|
        \le & 2F\|f\|_{L^2(M,\rho \volmm)}; 
        \label{eq:LimEigenfG2} 
      \end{align}
    \item there exists an orthonormal basis $\{\tilde{f}_k,\dots,\tilde{f}_l\}$ of $\mathrm{span}\{f_k,\dots,f_l\}$ such that for each $j \in \{k, \dots, l\}$,
      \begin{equation}\label{eq:LimEigenfG3} 
        \|\tilde{f}_j\vert_{\mathcal{X}_n} - \phi_j \|_{L^2(\mathcal{X}_n,\vol_{\Gamma})} \leq \sqrt{F}.
      \end{equation}
  \end{enumerate}
\end{theorem}

\begin{proof}
  We can assume that, for all $j\in \{1, \dots, n\}$, there exist $\{f^t_j \in  L^2(M,\volmm)\}_{t=1}^\infty $ such that $\Delta_\rho^N f^t_j = \lambda_j(\Delta_\rho^N)f^t_j$ and $f_j^t \to f_j$ in $L^2$ by Theorem~\ref{thm:mGHSpecConv}. Then,
  \begin{equation} 
    \|f^t_j - f_j\circ\Phi_t\|_{L^2(M_t,\rho^2_t\vol_{g_t})} \to 0
  \end{equation}
  as $t \to \infty$ for all $j \in \{1,\dots, k\}$. 
  For $f \in \spanv\{f_1,\dots,f_k\}$, set $f^t = \sum_{i=1}^k\langle f, f_i\rangle_{L^2(M,\rho^2\volmm)} f^t_i$.
  By Theorem~\ref{thm:mGHSpecConv}, for sufficiently large $t \in \N$, 
  \begin{equation}\label{eq:EigenAppStep1}
    \frac{C_2}{\gamma(\rho_t)}\sqrt{\delta_{p,\epsilon,a}(M_t) +4 s(\rho_t)\gamma(\rho_t)} -  \frac{C_2}{\gamma(\rho)}\sqrt{\delta_{p,\epsilon,a}(M) + 4s(\rho)\gamma(\rho)} - 2\|f^t - f\circ \Phi_t \|_{L^2(M_t,\rho^2_t\vol_{g_t})}
    \leq \sqrt{a}C
  \end{equation}
  holds.

  Hence, by Theorem~\ref{thm:AppEigenNonSingf} and Lemma~\ref{lem:Bern}, for every data set $\data^t$ from $\rho_t\vol_{g_t}$ with probability at least \[1- \exp(-na^2\epsilon^m)C(m,K,D,l)(\epsilon^{-2m} + n^3) - C(m,\alpha,v)n^2L\delta_t\epsilon^{m-1},\] the following statement holds.
  Set $\Gamma(t) = \Gamma^N_\epsilon(\Phi_t(\data^t), \tilde{d})$,
  and let $p_t$ denote the projection to the eigenspace corresponding to $\{\lambda_k(\Gamma(t)), \dots,\lambda_l(\Gamma(t))\}$ in $L^2(\Phi_t(\data^t),\vol_{\Gamma(t)})$,
  then we have the following estimates: 
  \begin{enumerate}
    \item For any $f^t \in \spanv\{f^t_k,\dots,f^t_l\}$, we have
      \begin{equation}\label{eq:EigenAppStep2}
        | \|f^t\|_{L^2(M_t,\rho_t^2\vol_{g_t})} - \| p_t (f^t|_{\data}) \|_{L^2(\data,\volgraph)} | \leq \| f \|_{L^2(M,\rho\volmm)}\frac{C_2}{\gamma(\rho_t)}\sqrt{\delta_{p,\epsilon,a}(M_t) + s(\rho_t)}.
      \end{equation}
    \item For any $f\in \spanv\{f_k,\dots, f_l\}$, $f^t = \sum_{j=k}^l \langle f, f_j\rangle f_j^t$ satisfy
      \begin{equation}\label{eq:EigenAppStep3}
        |\|f^t - f\circ \Phi_t \|_{L^2(\data,\vol_{\Gamma(t)})}^2 - \|f^t - f\circ \Phi_t \|_{L^2(M_t,\rho_t^2\vol_{g_t})}^2| \leq (\epsilon + a)C(m,K,D,l)\|f\|_{L^2(M,\rho^2\vol)}. 
      \end{equation}
  \end{enumerate}
  By inequalities \eqref{eq:EigenAppStep1}, \eqref{eq:EigenAppStep2}, and \eqref{eq:EigenAppStep3}, for any $f \in \spanv\{f_k,\dots,f_l\}$, 
  \begin{equation}
    | \|f\|_{L^2(M,\rho^2\volmm)} - \| p_t (f|_{\Phi_t(\data^t)}) \|_{L^2(\data^t,\vol_{\Gamma(t)})} | \leq \| f \|_{L^2(M,\rho^2\volmm)}\frac{C_2}{\gamma(\rho)}\sqrt{\delta_{p,\epsilon,a}(M) + s(\rho)}
  \end{equation}
  for data set $\data^t:\Omega \to M_t^n$ from $\rho_t\vol_{g_t}$ with this probability.
  Therefore, using weak* convergence $(\Phi_t)_*(\rho_t\vol_{g_t}) \to \rho\volmm$, we obtain \eqref{eq:LimEigenfG2}.
  We can show \eqref{eq:LimEigenfG1} similarly.
  Then the inequality \eqref{eq:LimEigenfG3} is an easy consequence of \eqref{eq:LimEigenfG1} and \eqref{eq:LimEigenfG2}.
\end{proof}

\appendix
\section{$L^\infty$ and gradient bounds for eigenfunctions of Laplacians on weighted Riemannian manifolds}\label{app:EstEigen}

In this appendix, we derive supremum and gradient bounds for eigenfunctions of the Laplacians $\Delta_\rho$ and $\Delta_\rho^N$.
We apply these estimates in Section~\ref{sec:DiscMap}.
A similar analysis appears in \cite{MR4804972}*{Section~A} when a lower bound on the injectivity radius is available.
Our approach does not rely on injectivity radius bounds.
Instead, we use an upper bound on the Hessian of $\rho$.

Throughout this appendix, for any measurable function \(f \colon  M \to \R\) and \(p \in [1,\infty)\), we define
\[
  \|f\|_{p}^p
  \;=\;
  \frac{1}{\vol_g(M)} \int_M |f|^p \,d\vol_g,
  \quad\text{and}\quad
  \|f\|_{\infty} = \esssup_{x\in M} |f(x)|.
\]

\begin{lemma}\label{lem:EstCEigenFStep1}
  For $m \in \N$ and $K, D, \lambda$, where $\lambda \ge 0$, there exists a constant $C= C(m, K, D,\lambda)>0$ such that for every $(M, g) \in \RicClass$ and any non-negative function $f \in H^{1, 2}(M)$, if $\|\diffcont f^q \|_2 \leq \sqrt{q\lambda}\|f\|_{2q}^q$ for all $q\in \N$, then we have $\|f\|_{\infty} \leq C\|f\|_1$.
\end{lemma}

\begin{proof}
  For any $(M,g) \in \mathcal{M}_m(K,D)$, there exist $C>0$ and $\nu > 2$, depending only on $m, K, D$, such that the following Nash inequality holds:
  \begin{equation}
    \| f \|_{2}^{2 +\frac{4}{\nu}} 
    \leq \left( C\|\diffcont f \|_{2}^2 + \|f\|_{2}^{2} \right) \|f\|_{1}^{\frac{4}{\nu}}
  \end{equation}
  for all $f\in H^{1,2}(M)\cap L^1(M)$.
  See p.~31 of \cite{MR1150597}.
  Combining this with $\|\diffcont f^q\|_{2} \leq \sqrt{q\lambda} \|f\|_{2q}^q $, we have
  \begin{equation}
    \|f\|_{{2^{k+1}}} \leq \exp\left(2^{-k/2}C(m,K,D,\lambda)\right)\|f\|_{{2^{k}}}
  \end{equation} 
  for every $k\in\N$.
  Hence, iterating this inequality, we obtain the desired inequality.
\end{proof}

\begin{lemma}\label{lem:EstCEigenF}
  For $m,k\in \N$ and $K,D,\alpha,\mathcal L,  H\geq1$, there exist constants $C_1=C_1(m,K,D,\alpha,k)$ and $C_2=C_2(m,K,D,\alpha,\mathcal L,  H,k)>0$ such that the following property holds.
  Let $(M,g) \in \RicClass$ with $\diam (M, d_g) \ge D^{-1}$, and let $\rho\in \mathcal{P}(M\colon\alpha,\mathcal{L}, H)$.
  Then for any solution $f\in H^{1,2}(M)$ to the equation 
  $\Delta_\rho^N f = \lambda_k(\Delta_\rho^N) f$ or $\Delta_\rho f = \lambda_k(\Delta_\rho) f$,
  we have $\|f\|_{\infty}\leq C_1\|f\|_{2}$ and $\Lip(f)\leq C_2 \|f\|_{2}$.
\end{lemma}

\begin{proof}
  For every $h\colon M\to [0,\infty)$, $\lambda \geq 0$, and $q\in \N$, if $h\in H^{1,2}(M)$ and $\Delta_\rho^N h \leq \lambda h$, we have
  \begin{align}
    \int_M | \diffcont h^{q}|^2 \rho^2 \, d\vol_g =  & \frac{q^2}{2q-1} \int_M \langle \diffcont h, \diffcont h^{2q-1} \rangle \, \rho^2 \vol_g\\
    \leq & \frac{q^2\lambda}{2q-1} \int_M h^{2q} \rho^2 \, d\vol_g.
  \end{align}
  Now $\Delta_\rho^N |f| \leq C(m,K,D,\alpha,k) |f|$ holds for our choice of $f$ by Remark~\ref{rem:EstCEigenV}.
  Thus $\|\diffcont f^q\|_{2} \leq \sqrt{2q C} \|f\|_{2q}^q$ for any $q\in\N$.
  Combining this with Lemma~\ref{lem:EstCEigenFStep1} yields $\|f\|_{\infty}\leq C(m,K,D,\alpha,k)\|f\|_{2}$.

  Applying the Bochner formula, we obtain
  \begin{equation}
    \Delta_\rho^N |\diffcont f|^2 = \langle \diffcont f, \diffcont \Delta_\rho^N f \rangle + 2 \Hess({\log\rho})(\diffcont f, \diffcont f) - \Ric_g(\diffcont f, \diffcont f) -|\Hess f|^2,
  \end{equation}
  if $\Delta_\rho^N f = \lambda_k(\Delta_\rho^N) f$, we obtain $\Delta_\rho^N |\diffcont f|^2 \leq C(m,K,D,\alpha, H,k) |\diffcont f|^2$.
  Hence $\Lip(f) \leq C\|f\|_{2}$.
  In the case of $\Delta_\rho f = \lambda_k(\Delta_\rho) f$, since
  \begin{equation}
    \Delta_\rho^N |\diffcont f|^2 \leq ((m-1)K + 2 H)|\diffcont f|^2 + \frac{\langle \diffcont f,\diffcont \Delta_\rho f \rangle}{\rho} - \frac{\langle \diffcont \rho, \diffcont f\rangle\Delta_\rho f}{\rho^2},
  \end{equation} 
  we have $\Delta_\rho^N |\diffcont f|^2 \leq C(m,K,D,\alpha,\mathcal L, H)(|\diffcont f|^2 + |f|^2)$.
  Then,
  \begin{align}
    \int_M \Bigl| \diffcont |\diffcont f|^{2q}\Bigr|^2 \rho^2 \, d\vol_g & = \frac{q^2}{2q-1} \int_M |\diffcont f|^{4q-2}\Bigl(\Delta_\rho^N|\diffcont f|^{2}\Bigr) \rho^2 \, d\vol_g\\
                                                                         & \leq \frac{q^2C(m,K,D,\alpha,\mathcal L,  H) }{2q-1} \int_M \left(|\diffcont f|^{4q}  + |\diffcont f|^{4q-2}|f|^2\right) \rho^2\, d\vol_g,
  \end{align}
  so we have $\bigl\|\diffcont |\diffcont f|^{2q} \bigr\| \leq \sqrt{q}C \left( \| \diffcont f \|_{4q}^{2q} + \|\diffcont f\|_{{4q}}^{2q-1} \|f\|_{{4q}}\right)$.
  By the Poincar\'e inequality, $\|f\|_{4q} \leq C(m,K,D)\|\diffcont f\|_{4q} + \|f\|_{2}$ holds, we obtain
  \begin{align}
    \| \diffcont |\diffcont f|^{2q} \|_{2} \leq &\sqrt{q} C(m,K,D,\alpha,\mathcal L,  H) \| \diffcont f \|_{4q}^{2q}\left( 1  +  \frac{\|f\|_{{2}}}{\|\diffcont f\|_{2}} \right)\\
    \leq& \sqrt{q} C(m,K,D,\alpha,\mathcal L,  H)\| \diffcont f \|_{{4q}}^2.
  \end{align}
where we used $\lambda_1(M) \geq C(m,K,D)>0$ (\cite{MR1736580}*{Theorem B}) in the last inequality.
  Thus, by Lemma~\ref{lem:EstCEigenFStep1}, we obtain $\Lip(f) \leq C\|f\|_{2}$.
\end{proof}

\bibliographystyle{plain}
\bibliography{list}
\end{document}